\documentclass[11pt,amstex]{amsart}
\usepackage[shortlabels]{enumitem}
\usepackage{txfonts}
\usepackage{geometry}
\usepackage{amsmath, amsthm, amssymb,amsfonts}
\usepackage{epsfig}
\usepackage{amsmath}
\usepackage{amssymb}
\usepackage{amscd}
\usepackage{graphicx}
\usepackage{pstricks}
\usepackage{pst-node}
\usepackage[english]{babel}
\usepackage{float}
\theoremstyle{definition}
\newtheorem{thm}{Theorem}[section]
\newtheorem{lem}[thm]{Lemma}
\newtheorem{defn}[thm]{Definition}
\newtheorem{rem}[thm]{Remark}
\newtheorem{prop}[thm]{Proposition}
\newtheorem{ques}[thm]{Question}
\newtheorem{cor}[thm]{Corollary}
\newtheorem{ex}[thm]{Example}
\newtheorem{conj}[thm]{Conjecture}
\newtheorem{conv}[thm]{Convention}

\def \B {\mathcal{B}}
\def \E {\mathbb{E}}
\def \G {\mathfrak{X}=(G_{\bullet},\mathcal{X},\psi,d_{G},d_{X})}
\def \Gc {\mathfrak{X}_{c}=(G_{c,\bullet},\mathcal{X}_{c},\psi_{c},d_{c,G},d_{c,X})}
\def \K {\bold{K}=(K,\mathcal{O}_{K},D,\B=\{b_{1},\dots,b_{D}\})}

\def \N {\mathbb{N}}
\def \O {\mathcal{O}}
\def \P {\mathcal{P}}
\def \R {\mathbb{R}}
\def \T {\mathbb{T}}
\def \X {\mathfrak{X}}
\def \Y {\mathfrak{Y}}
\def \Z {\mathbb{Z}}

\def \c {\bold{c}}
\def \g {\mathfrak{g}}
\def \h {\mathfrak{h}}
\def \j {\bold{j}}

\def \m {\bold{m}}
\def \n {\bold{n}}
\def \u {\bold{u}}
\def \v {\bold{v}}
\def \w {\bold{w}}
\def \x {\bold{x}}
\def \y {\bold{y}}
\def \z {\bold{z}}

\def \Aut {\text{Aut}}
\def \Lip {\text{Lip}}
\def \rank {\text{rank}}
\def \poly {\text{poly}}

\def \e {\epsilon}
\def \d {\delta}

\def \A {\n}
\def \Bb {\m}

\def \Rr {\mathbb{Z}_{\tilde{N}}^{D}}
\def \RR {\mathbb{Z}_{N}^{D}}
\def \OO {\mathcal{O}_{K}^{\ast}}

\def \mult {\text{mult}}

\newcommand{\subscript}[2]{$#1 _ #2$}

\title[Sarnak's Conjecture on arbitrary number fields]{
	Sarnak's Conjecture for nilsequences on arbitrary number fields and applications
	}
\author{Wenbo Sun}
\address{Department of Mathematics, Virginia Tech University, 225 Stanger Street, Blacksburg VA, 24061-1026, USA}
\email{swenbo@vt.edu}



\begin{document}

\maketitle

\begin{abstract}
	We formulate the generalized Sarnak's M\"obius disjointness conjecture for an arbitrary number field $K$, and prove a quantitative disjointness result between polynomial nilsequences $(\Phi(g(\bold{n})\Gamma))_{\bold{n}\in\mathbb{Z}^{D}}$ and aperiodic multiplicative functions on $\mathcal{O}_{K}$, the ring of integers of $K$. Here $D=[K\colon\mathbb{Q}]$, $X=G/\Gamma$ is a nilmanifold, $g\colon\mathbb{Z}^{D}\to G$ is a polynomial sequence, and $\Phi\colon X\to \mathbb{C}$ is a Lipschitz function. This result, being a generalization of a previous theorem of the author in \cite{S}, requires a significantly different approach, which involves with multi-dimensional higher order Fourier analysis, multi-linear analysis, orbit properties on nilmanifold, and an orthogonality criterion of K\'atai in $\mathcal{O}_{K}$. 
	
	We also use variations of this result to derive applications in number theory and combinatorics: (1) we prove a structure theorem for multiplicative functions on $K$, saying that every bounded multiplicative function can be decomposed into the sum of an almost periodic function (the structural part) and a function with small Gowers uniformity norm of any degree (the uniform part); (2) we give a necessary and sufficient condition for the Gowers norms of a bounded multiplicative function in $\mathcal{O}_{K}$ to be zero; (3) we provide partition regularity results over $K$ for a large class of homogeneous equations in three variables. For example, for $a,b\in\mathbb{Z}\backslash\{0\}$, we show that for every partition of  $\mathcal{O}_{K}$  into finitely many cells, where $K=\mathbb{Q}(\sqrt{a},\sqrt{b},\sqrt{a+b})$, there exist distinct and non-zero $x,y$ belonging to the same cell and $z\in\mathcal{O}_{K}$ such that $ax^{2}+by^{2}=z^{2}$.
\end{abstract}	


\section{Introduction}\label{s:1}
%
%
%
%
%
%
%
%
%
%
%
%

\subsection{Sarnak's Conjecture on number fields}
Let $\mu\colon\Z\to\{-1,0,1\}$ be the \emph{M\"obius function}, which  is defined by $\mu(n)=(-1)^{k}$ if $\vert n\vert$ is the product of $k$ distinct prime numbers in $\N$, and $\mu(n)=0$ otherwise.\footnote{The definition of the M\"obius function  is usually stated for $\N$, but for the convenience of this paper we state it for $\Z$.}  
It is widely believed the function $\mu$ satisfies the ``M\"obius randomness law" (see Section 13.1 of \cite{39}), in the sense that $\mu$ is not
correlated with any sequence of complex numbers of ``low complexity". This vague
principle turns out to often provide heuristic asymptotics for various averages along primes (see \cite{61} for examples). In \cite{sar12}, a precise conjecture was formulated by Sarnak: 
\begin{conj}[Sarnak's Conjecture for integers]\label{sarnakz}
	Let $(X,T)$ be a topological system with zero topological entropy. Then for all $\Phi\in C(X)$ and $x\in X$, 
	$$\lim_{N\to\infty}\frac{1}{N}\sum_{n=1}^{N}\mu(n)\Phi(T^{n}x)=0.$$
\end{conj}	

Many instances of Sarnak’s Conjecture have been proven. We give a few examples
but stress that this is an incomplete list: 
\cite{EAKL16,ab,Bou13,Bou13b,BSZ13,DK15,FJ18,F17,7,HW,HLSY17,LS15,MMR14,MC19,Ml17,Pec18,Vee17,Wan17}.
It is natural to ask whether Sarnak's Conjecture holds with $\mu$ replaced by other functions which are interesting in analytic number theory. We say that a function $\chi\colon\Z\to\mathbb{C}$ is \emph{multiplicative} (written as $\chi\in\mathcal{M}_{\mathbb{Q}}$) if  $\chi(mn)=\chi(m)\chi(n)$ for all $(m,n)=1$. Let $\mathcal{M}^{a}_{\mathbb{Q}}$ denote the set of all multiplicative functions $\chi$ of modulus at most 1 which is \emph{aperiodic}, meaning that 
\begin{equation}\label{ap1}
	\lim_{N\to\infty}\frac{1}{N}\sum_{n=0}^{N-1}\chi(an+b)=0
\end{equation}	 
for all $a,b\in\Z, a\neq 0$.
It is a classical result that the M\"obius function $\mu$ is aperiodic.
 One can ask the following question:
\begin{ques}[Generalized Sarnak's Conjecture for integers]\label{sarnakzg}
	For which $\chi\in\mathcal{M}^{a}_{\mathbb{Q}}$ does the following hold:
	for every topological system $(X,T)$ with zero topological entropy, every $\Phi\in C(X)$, and every $x\in X$, we have that
	$$\lim_{N\to\infty}\frac{1}{N}\sum_{n=1}^{N}\chi(n)\Phi(T^{n}x)=0.$$
\end{ques}
It is not hard to see that the answer to Question \ref{sarnakzg} is false when $\chi$ is not aperiodic. Motived by the results in \cite{counter}, a natural conjecture is that the answer to Question \ref{sarnakzg} is affirmative if $\chi$ is ``strongly aperiodic" (which in particular includes all real-valued aperiodic multiplicative functions, see \cite{counter} for the definition).

In this paper, we enhance the scope of Sarnak's Conjecture (and related topics) to multiplicative functions on general number fields, and seek applications in this broader setting. Let $K$ be an algebraic number field and $\O_{K}$ be its ring of integers (see Section \ref{s:2} for definitions). Denote $D:=[K\colon\mathbb{Q}]$ and let $\B=\{b_{1},\dots,b_{D}\}$ be an integral basis of $\O_{K}$. For convenience we call $\K$ an \emph{integral tuple}.
In analog to the case $K=\mathbb{Q}$, for a general number field $K$,
one can define 
the set of bounded multiplicative functions $\mathcal{M}_{K}$, and that of bounded aperiodic multiplicative functions $\mathcal{M}^{a}_{K}$ in a natural way (see Section \ref{s:12} for the precise definitions). One can also formulate the  generalized Sarnak's Conjecture for algebraic number fields:


\begin{ques}[Generalized Sarnak's Conjecture for algebraic number fields]\label{sarnak}
	Let $\K$ be an integral tuple.  For which $\chi\in\mathcal{M}^{a}_{\mathbb{Q}}$ does the following hold: for every topological system $(X,T_{1},\dots,T_{d})$  with commuting transformations $T_{1},\dots,T_{d}$ with zero topological entropy, every $\Phi\in C(X)$,  every $x\in X$, and every $D$-dimensional arithmetic progression $P$,\footnote{A set $P\subseteq \Z^{D}$ is a \emph{$D$-dimensional arithmetic progression} if $P=\{\n_{0}+\sum_{i=1}^{D}M_{i}n_{i}\colon n_{i}\in\{0,\dots,N_{i}-1\}, 1\leq i\leq D\}$ for some $M_{i},N_{i}\in\N_{+}$ and $\n_{0}\in\Z^{D}$. $(N_{1},\dots,N_{D})$ is called the \emph{length} and $(M_{1},\dots,M_{D})$ the \emph{step} of $P$.} we have that
	$$\lim_{N\to\infty}\frac{1}{N^{D}}\sum_{1\leq n_{1},\dots,n_{D}\leq N}\bold{1}_{P}(n_{1},\dots,n_{D})\chi(n_{1}b_{1}+\dots+n_{D}b_{D})\Phi(T^{n_{1}}_{1}\cdot\ldots\cdot T^{n_{D}}_{D}x)=0.$$
\end{ques}

%
%

The main result of this paper is to provide an affirmative answer to Question \ref{sarnak} for nilsystems (see Section \ref{s:3} for definitions) with respect to all aperiodic functions, in a more general sense that one can replace $T^{n_{1}}_{1}\cdot\ldots\cdot T^{n_{D}}_{D}x$ by any polynomial sequence (see Section \ref{s:4} for definitions), and taking the average along any arithmetic progression:
\begin{thm}[Generalized Sarnak's Conjecture along nilsequences]\label{s2}
Let $\K$ be an integral tuple.  Let $X=G/\Gamma$ be a nilmanifold and $g\colon\Z^{D}\to \mathbb{C}$ be a polynomial sequence. Then for every $\chi\in\mathcal{M}^{a}_{K}$, every $\Phi\in C(X)$, and every $D$-dimensional arithmetic progression $P$, we have that
$$\lim_{N\to\infty}\frac{1}{(2N+1)^{D}}\sum_{-N\leq n_{1},\dots,n_{D}\leq N}\bold{1}_{P}(n_{1},\dots,n_{D})\chi(n_{1}b_{1}+\dots+n_{D}b_{D})\Phi(g(n_{1},\dots,n_{D})\cdot e_{X})=0.$$
\end{thm}	

In particular, Theorem \ref{s2} implies that 
Conjecture \ref{sarnak} holds for every integral tuple $\bold{K}$ and every nilmanifold $X$ with $T_{1},\dots,T_{d}$ being translations on $X$ (not necessarily commuting with each other).

If the sequence $(g(n_{1},\dots,n_{D})\cdot e_{X})_{(n_{1},\dots,n_{D})\in\Z^{D}}$ in Theorem \ref{s2} is \emph{totally equidistributed} on $X$, meaning that 
\begin{equation}\label{ed}
	\begin{split}
		\lim_{N\to\infty}\frac{1}{(2N+1)^{D}}\sum_{-N\leq n_{1},\dots,n_{D}\leq N}\bold{1}_{P}(n_{1},\dots,n_{D})\Phi(g(n_{1},\dots,n_{D})\cdot e_{X})=0
	\end{split}
\end{equation}	
for  every (infinite) $D$-dimensional arithmetic progression $P$ and every
$\Phi\in C(X)$ such that $\int_{X}\Phi\,d m_{X}=0$ (where $m_{X}$ is the Haar measure on $X$), then one can deduce a generalization of a result of Daboussi, which can be viewed as a variation of Theorem \ref{s2}:

\begin{thm}[Generalized Daboussi's Theorem]\label{key0}
	Let $\K$ be an integral tuple.
	Let $X=G/\Gamma$ be a nilmanifold 
	and $g\colon\Z^{D}\to\mathbb{C}$ be a polynomial sequence such that $(g(n_{1},\dots,n_{D})\cdot e_{X})_{(n_{1},\dots,n_{D})\in\Z^{D}}$ is totally equidistributed on  $X$.
Then for 
every $\Phi\in C(X)$ such that $\int_{X}\Phi\,d m_{X}=0$,\footnote{The additional assumption  $\int_{X}\Phi\,d m_{X}=0$ is necessary, as otherwise the theorem fails for $\Phi\equiv 1$ and $\chi\equiv 1$.} and every $D$-dimensional arithmetic progression $P$, we have that 
	\begin{equation}\nonumber
		\begin{split}
			\lim_{N\to\infty}\sup_{\chi\in\mathcal{M}_{K}}\Bigl\vert\frac{1}{(2N+1)^{D}}\sum_{-N\leq n_{1},\dots,n_{D}\leq N}\bold{1}_{P}(n_{1},\dots,n_{D})\chi(n_{1}b_{1}+\dots+n_{D}b_{D})\Phi(g(n_{1},\dots,n_{D})\cdot e_{X})\Bigr\vert=0.
		\end{split}
	\end{equation}		
\end{thm}

For 
the special case $\O_{K}=\Z$ of Theorem \ref{key0}, when $X=\T$ and $g$ is a linear polynomial, this is also known as Daboussi's theorem (\cite{D1,D2,D3}).  when $X=\T$ and $g$ is a general polynomial, this was essentially proved by K\'atai \cite{katai}. The general case for  $\O_{K}=\Z$ was proved by Frantzikinakis and Host (Theorem 2.2 of \cite{FH}), and the case when $\O_{K}=\Z[i]$ and $X$ is at most 2-step was proved in Propositions 7.8 and 7.9 of \cite{S}.


In addition to Theorems \ref{s2} and \ref{key0}, we also provide quantitative versions of them. See Theorems \ref{key} and \ref{s3} for details.

\begin{rem}
	The proofs of the quantitative Theorems \ref{key} and \ref{s3}, being much more complicated than the qualitative Theorems \ref{s2} and \ref{key0}, occupy the bulk of this paper.
	The paper could be largely shortened if one is satisfied with the qualitative results only. However, in order for these results to be useful for applications, it is important to have quantitative versions of them. 
\end{rem}


\subsection{Applications}\label{s:12}
It turns out that there are many applications of Theorems \ref{s2} and \ref{key0} (and their quantitative versions), which we explain in this section. 

\subsubsection{Structure theorem for multiplicative functions}

We start with the precise definition for multiplicative functions on $\O_{K}$ (see Section \ref{s:2} for terminologies arising from algebraic number theory):

 \begin{defn}[Multiplicative functions on $\O_{K}$]\label{d7}
 	Let $K$ be a number field and $\O_{K}$ be its ring of integers. We say that a function $\chi\colon \O_{K}\to\mathbb{C}$ is  \emph{multiplicative} on  $\O_{K}$ if for all  $m,n\in \O_{K}$ such that the $K$-norm $N_{K}(m)$ of $m$ is coprime with $N_{K}(n)$ in $\Z$, we have that $\chi(mn)=\chi(m)\chi(n)$.\footnote{In fact, multiplicative functions can be defined on all the ideals of $\O_{K}$ rather than just on the principal ideals $(n)$ in a natural way. In this paper, we restrict the domain of the functions to the principal ideals only since such functions are already good enough for applications.} 
 	
 	Let $\mathcal{M}_{K}$ denote the collection of all multiplicative functions on   $\O_{K}$ with modulus at most 1, and 	$\mathcal{M}^{a}_{K}$  denote the collection of all \emph{aperiodic} functions $\chi$ in $\mathcal{M}_{K}$, meaning that 
 	\begin{equation}\label{ap2}
 		\lim_{N\to\infty}\sup_{P}\Bigl\vert\frac{1}{(2N+1)^{D}}\sum_{-N\leq n_{1},\dots,n_{D}\leq N}\bold{1}_{P}\Bigr(n_{1}b_{1}+\dots+n_{D}b_{D}\Bigl)\chi\Bigr(n_{1}b_{1}+\dots+n_{D}b_{D}\Bigl)\Bigr\vert=0,
 	\end{equation}	
 where $\sup_{P}$ is taken over all $D$-dimensional arithmetic progressions $P$. 
 \end{defn}	
 One can show that
 Definition \ref{d7} coincide with the definition in (\ref{ap1}) when $K=\mathbb{Q}$ (see  Appendix \ref{appe} for the proof).
An application of the Sarnak's Conjecture along nilsequences is to provide structure theorems for multiplicative functions $\chi$ in $\mathcal{M}_{K}$ on an arbitrary number field $K$. 
Roughly speaking, our
structure theorem says that $\chi$ can be written as the sum of two functions $\chi_{s}$ (the ``structural part") and $\chi_{u}$ (the ``uniform part"), where $\chi_{s}$ is approximately periodic and $\chi_{u}$ behaves randomly enough to have a negligible contribution for the applications we are interested in.
The uniformity of a function is measured by the Gowers norms.

\begin{defn}[Gowers uniformity norms on $\RR$]
	For $d,D,N\in\N_{+}$, we define the \emph{$d$-th Gowers uniformity norm} of $f$ on $\RR$ \footnote{$\Z_{N}:=\Z/N\Z$.} inductively by
	\begin{equation}\nonumber
		\begin{split}
			\Vert f\Vert_{U^{1}(\RR)}:=\Bigl\vert\frac{1}{N^{D}}\sum_{\A\in \RR}f(\A)\Bigr\vert
		\end{split}
	\end{equation}
	and
	\begin{equation}\nonumber
		\begin{split}
			\Vert f\Vert_{U^{d+1}(\RR)}:=\Bigl(\frac{1}{N^{D}}\sum_{\Bb\in \RR}\Vert f_{\Bb}\cdot\overline{f}\Vert_{_{U^{d}(\RR)}}^{2^{d}}\Bigr)^{1/2^{d+1}}
		\end{split}
	\end{equation}
	for $d\geq 1$, where $\overline{f}$ denotes the conjugate of $f$ and $f_{\Bb}(\A):=f(\Bb+\A)$ for all $\A\in\RR$.
\end{defn}
Gowers \cite{G} showed that this defines a norm on functions on $\mathbb{Z}_{N}$ for $d>1$. These norms were later used by Green, Tao, Ziegler and others in studying the primes (see, for example, \cite{LE,65,6}). Analogous semi-norms were defined in the ergodic setting by Host and Kra \cite{HK}. 

\begin{conv}
	For   an integral tuple $\K$, let $\iota_{\B}\colon\Z^{D}\to\O_{K}$ denote the bijection given by $$\iota_{\B}(n_{1},\dots,n_{D})=n_{1}b_{1}+\dots+n_{D}b_{D}.$$
Let $\chi\colon\O_{K}\to\mathbb{C}$ be a function and $N,\tilde{N}\in\N$ be such that $N<\tilde{N}$. We use  $\chi_{N,\tilde{N}}\colon \Rr\to\mathbb{C}$ to denote the \emph{truncated function}  given by $\chi_{N,\tilde{N}}(\n)=\chi\circ\iota_{\B}(\n)$ for all $\n\in\{1,\dots,N\}^{D}$ and $\chi_{N,\tilde{N}}(\n)=0$ otherwise. Through this paper,  we write $\chi_{N}:=\chi_{N,\tilde{N}}$ to simplify the notations of truncated functions. The quantity $\tilde{N}$ will be always clear from the context.	
\end{conv}	

We assume that the set $\mathcal{M}_{K}$ is endowed with the topology of
pointwise convergence and thus is a compact metric space. The main structure theorem we have is the following, which generalizes Theorem 1.1 of \cite{FH} and answers Problem 1 of \cite{FH}:

\begin{thm}[$U^{d}$ structure theorem for multiplicative functions]\label{U3}
	Let $\Omega\in\N$. For all $N\in\N$, let $\tilde{N}$ denote the smallest prime integer greater than $\Omega N$. 
	Let $\K$ be an integral tuple and
	$\nu$ be a probability measure  on the group $\mathcal{M}_{K}$.
	For every $\e>0$ and $d\geq 2$, there exist $Q:=Q(\B,d,\e,\Omega),$ \footnote{If $\mathcal{I}$ is a collection of parameters, then the notion $C:=C(\mathcal{I})$ means that $C$ is a quantity depending only on the parameters in $\mathcal{I}$.} $R:=R(\B,d,\e,\Omega)$ and $N_{0}:=N_{0}(\B,d,\e,\Omega)\in\N_{+}$ \footnote{If a quantity  depends on  $\B$ (such as $Q,R$ and $N_{0}$), then it also implicitly depends on $K,\O_{K}$ and $D$.}
	 such that for every
	$N\geq N_{0}$ and $\chi\in\mathcal{M}_{K}$, the truncated function $\chi_{N}\colon \Rr\to\mathbb{C}$ can be written as
	\begin{equation}\nonumber
		\begin{split}
			\chi_{N}(\A)=\chi_{N,s}(\A)+\chi_{N,u}(\A)
		\end{split}
	\end{equation}
	for all $\A\in\Rr$ such that the following holds:
	\begin{enumerate}[(i)]	
		\item $\vert\chi_{N,s}\vert\leq 1$ and $\vert\chi_{N,s}(\A+Q\bold{e}_{i})-\chi_{N,s}(\A)\vert\leq\frac{R}{N}$ for every $\A\in \Rr$ and $1\leq i\leq D$;\footnote{$\bold{e}_{i}$ denotes the vector whose $i$-th coordinate is 1 and all other coordinates are 0.}
		\item $\Vert\chi_{N,u}\Vert_{U^{d}(\Rr)}\leq \e$.
	\end{enumerate}	
\end{thm}


We remark that  by Bertrand's postulate, $\Omega N<\tilde{N}< 2\Omega N$.
The reason that we work on $\Z^{D}_{\tilde{N}}$ rather than $\Z^{D}_{N}$ is that for all $a\in\O_{K}$ such that $0<\vert N_{K}(a)\vert<\tilde{N}$, the map $\n\to\n A_{\B}(a) \mod\Rr$ (see Section \ref{s:2} for the definition of $A_{\B}(a)$) is a bijection from $\Rr$ to itself (see also the discussion after Proposition \ref{c4} for the reason).

We prove Theorem \ref{U3} (and  its stronger version Theorem \ref{nU3s}) in Section \ref{s:s}.

\subsubsection{A criteria for aperiodic multiplicative functions}
Another application of the main results of the paper is to provide a criteria for aperiodic  multiplicative functions using Gowers norms. Denote $[N]:=\{1,\dots,N\}$. The Gowers norms can be extended to functions taking values in $\Z^{D}$ (as was done in \cite{FH,LE}).
\begin{defn}[Gowers uniformity norms on intervals]
	Let $d\geq 2$ and $D,N\in\N$. For every functions $f\colon [N]^{D}\to\mathbb{C}$, by Lemma A.2 of Appendix A of \cite{FH}, the quantity
	$$\Vert f\Vert_{U^{d}([N]^{D})}:=\frac{1}{\Vert \bold{1}_{[N]^{D}}\Vert_{U^{d}(\Z_{N^{\ast}})}}\cdot\Vert \bold{1}_{[N]^{D}}\cdot f\Vert_{U^{d}(\Z_{N^{\ast}})}$$
	is independent of $N^{\ast}$ if $N^{\ast}>2N$, which is called the \emph{$U^{d}([N]^{D})$-norm} of $f$.
\end{defn}

If the $U^{2}([N]^{D})$-norm of a function goes to 0 as $N\to\infty$, then it is an aperiodic function. However, the converse is not always the case. But for multiplicative functions, these two conditions are equivalent:

\begin{thm}[Structure theorem for aperiodic multiplicative functions]\label{ap}
	Let $\K$ be an integral tuple
	 and $\chi\in\mathcal{M}_{K}$. Then $\chi$ is aperiodic if and only if $\lim_{N\to\infty}\Vert\chi\circ\iota_{\B}\Vert_{U^{d}([N]^{D})}=0$ for all $d\geq 2$.
\end{thm}

Theorem \ref{ap} generalizes Theorem 2.5 of \cite{FH}, and we prove it in Section \ref{s:p}.

\subsubsection{Partition regularity for homogeneous equations}

An important question in Ramsey theory is to determine which algebraic equations, or systems of equations, are partition regular. 
As there are various formulations of the partition regular questions, we start with a technical definition in order to cover as many cases as possible:  

\begin{defn}[Partition regularity]
	Let $\K$ be an integral tuple, $s\in\N_{+}, r\in\N$, and $p\in\mathbb{C}[x_{1},\dots,x_{s};z_{1},\dots,z_{r}]$ be a polynomial. We say that $p$ is \emph{partition regular} over $\O_{K}$ with respect to $x_{1},\dots,x_{s}$ if for every finite partition $\O_{K}=\bigcup_{i=1}^{m}U_{i}$ of $\O_{K}$, there exist $1\leq i\leq m$, $x_{1},\dots,x_{s}\in U_{i}$ non-zero and pairwise distinct  such that  $p(x_{1},\dots,x_{s};z_{1},\dots,z_{r})=0$  for some $z_{1},\dots,z_{r}\in\O_{K}$. 
\end{defn}	

It was proved by Rado \cite{Rado} that for $a,b,c\in\Z\backslash\{0\}$, the linear polynomial $p(x_{1},x_{2},x_{3})=ax_{1}+bx_{2}+cx_{3}$  is partition regular over $\Z$ with respect to $x_{1}, x_{2}$ and $x_{3}$ (in this case $r=0$) if and only if one of $a+b,a+c,b+c, a+b+c$ is 0.\footnote{The original result of \cite{Rado} was stated for $\N$ but a similar argument holds for $\Z$.} 
 The situation is much less clear for second or higher degree equations $p$, or for integer rings $\O_{K}$ other than $\Z$, unless we allow some of the variables in $p$ to take values freely (namely $r>0$).
It is a classical result of Furstenberg  \cite{77} and Sark\"{o}zy \cite{25} that the equation $p(x_{1},x_{2};z_{1})=x_{1}-x_{2}-z^{2}_{1}$ is partition regular over $\mathbb{Z}$ with respect to $x_{1}$ and $x_{2}$. Bergelson and Leibman \cite{BL} provided other examples of translation invariant equations by proving a polynomial version of the van der Waerden Theorem. 
A result of Khalfalah and Szemer\'{e}di \cite{KS} showed that the equation $p(x_{1},x_{2};z_{1})=x_{1}+x_{2}-z^{2}_{1}$ is partition regular over $\mathbb{Z}$ with respect to $x_{1}$ and $x_{2}$.

In the work of Frantzikinakis and Host \cite{FH}, by using the structure theorem for multiplicative functions on $\mathbb{Z}$, they proved that certain class of quadratic equations with two restricted variables $x_{1}$ and $x_{2}$ are partition regularity (as well as scattered examples of higher degree equations). For example,  they showed that the equation 
\begin{equation}\label{p}
p(x_{1},x_{2};z_{1})=ax^{2}_{1}+bx^{2}_{2}-z^{2}_{1}, 	
\end{equation}	
 is partition regular over $\mathbb{Z}$ with respect to $x_{1}$ and $x_{2}$ if $a,b,a+b$ are non-zero square integers (for example, $a=16, b=9$).

Note that not all equations of the form (\ref{p}) are partition regular over $\mathbb{Z}$. For example, the equation $3x^{2}_{1}-x^{2}_{2}-z^{2}_{1}=0$ has even no non-trivial integer solutions. So it is natural to consider the  partition regularity problems over a larger ring of integers.
In \cite{S}, by using a partial structure theorem for multiplicative functions on $\mathbb{Z}[i]$, the author proved that (\ref{p}) is partition regular over $\mathbb{Z}[i]$ with respect to $x_{1}$ and $x_{2}$ if $\sqrt{a},\sqrt{b},\sqrt{a+b}\in\Z[i]$  (for example, $a=1, b=-1$).
We remark that the question whether $x^{2}_{1}-x^{2}_{2}-z^{2}_{1}$ is partition regular over $\Z$ with respect to $x_{1}$ and $x_{2}$  remains an open question.

In this paper, we provide partition regularity results for a larger family of polynomials over certain ring of integers. Our main result is Theorem \ref{pr}. We  postponed its precise statement to Section \ref{s:par}, but provide a few sample applications of Theorem \ref{pr} in the introduction.

The first is an example for quadric equations:

\begin{thm}\label{pr3}
	Let $p(x_{1},x_{2};z_{1})=ax_{1}^{2}+bx_{2}^{2}-z_{1}^{2}$
	for some $a,b\in\Z\backslash\{0\}$. Then $p$ is partition regular over the ring of integers of $\mathbb{Q}(\sqrt{a},\sqrt{b},\sqrt{a+b})$ with respect to $x_{1}$ and $x_{2}$. 
%
%
\end{thm}	


For example, $p(x_{1},x_{2};z_{1})=9x_{1}^{2}+16x_{2}^{2}-z_{1}^{2}$ is partition regular over $\Z$ with respect to $x_{1}$ and $x_{2}$, which reproves a result in \cite{FH}; $p(x_{1},x_{2};z_{1})=x_{1}^{2}-x_{2}^{2}-z_{1}^{2}$ is partition regular over $\Z[i]$ with respect to $x_{1}$ and $x_{2}$, recovering a theorem in \cite{S}; $p(x_{1},x_{2};z_{1})=x_{1}^{2}+m^{2}x_{2}^{2}-z_{1}^{2}$ is partition regular over the ring of integers of $\mathbb{Q}(\sqrt{m^{2}+1})$ with respect to $x_{1}$ and $x_{2}$ for all $m\in\Z$.

The second is an example for a polynomial $p(x,y;z_{1},\dots,z_{r})$ with $r\geq 2$, where more flexibility is allowed. We have

\begin{cor}\label{45}
	The equation
	$p(x_{1},x_{2};z_{1},z_{2}):=x^{2}_{1}-2x_{2}^{2}+z^{2}_{1}-z^{2}_{2}$ is  partition regular over $\Z[\frac{1+\sqrt{-3}}{2}]$  with respect to $x_{1}$ and $x_{2}$.
\end{cor}
We remark it was proved by \cite{FH} that a similar polynomial $p(x_{1},x_{2};z_{1},z_{2}):=x^{2}_{1}+x_{2}^{2}-2z^{2}_{1}-z^{2}_{2}$ is partition regular over the $\Z$ with respect to $x_{1}$ and $x_{2}$. 

We prove Theorem \ref{pr3} and Corollary \ref{45} in Section \ref{s:p2}.

\subsection{Methods and organizations}

The first part of this paper is the background material used in this paper. In Section \ref{s:2}, we provide all the results we need from algebraic number theory. In particular, we provide the K\'atai's Lemma on algebraic number fields (Lemma \ref{katai}), which is a useful tool for the study of Sarnak's Conjecture on an arbitrary number field. In Sections \ref{s:3} and \ref{s:4}, we provide basic properties on nilmanifolds and equidistribution properties for polynomial sequences, respectively. The material in these two sections is a mixture of classical knowledges and original results.

The second part of the paper is
devoted to the proof of Theorem \ref{key} (the quantitative version of Theorem \ref{key0}), which is the central result in this paper. Theorem \ref{key} can be viewed as a generalization of the main result in \cite{S}. However, the proof of  Theorem \ref{key} is significantly more difficult, and we will use a completely different approach.

To be more precise,
Sections \ref{s:d}, \ref{s:m} and \ref{s:o} are the main novelties of this paper. 
 In Section \ref{s:o}, we use the material in Sections \ref{s:2}, \ref{s:3} and \ref{s:4} to reduce Theorem \ref{key} to two questions (which are also the two main innovations of this paper): (i) the description of a special sub nilmanifold of the product space $X\times X$, which is carried out in Section \ref{s:d} (Theorem \ref{auto}); and (ii) a problem in multi-linear algebra, which is answered in Section \ref{s:m} (Theorem \ref{contri}).

The last part of this paper is to use Theorem \ref{key} to prove all other results. In Section \ref{s:p}, we prove some immediate consequences of Theorem \ref{key}, including Theorem \ref{key0}, the generalized Daboussi's Theorem, Theorem \ref{s3} (the quantitative version of Theorem \ref{s2}), the Sarnak's Conjecture for nilsequences, and Theorem \ref{ap}, the criteria for aperiodic multiplicative functions.
In Section \ref{s:s}, we prove the structure theorem for multiplicative functions, namely Theorem \ref{nU3s} (which is a stronger version of Theorem \ref{U3}). In Section \ref{s:par}, we prove the partition regularity results (i.e. Theorem \ref{pr3} in its full generality) by using the structure theorem.

Roughly speaking, by Katai's criteria (Lemma \ref{katai}), one can reduce Theorems \ref{key0} and  \ref{key} to the study of averages of the expression 
\begin{equation}\label{idea}
\Phi\otimes\overline{\Phi}(g(pn)\cdot e_{X}, g(qn)\cdot e_{X}),
\end{equation}
where $n$ ranges over $\mathcal{O}_{K}$, $p,q\in \mathcal{O}_{K}$, $X=G/\Gamma$ is an $s$-step nilmanifold, $\Phi\in C(X)$, and $g\colon\mathcal{O}_{K}\to G$. 
In \cite{FH}, where the case $K=\mathbb{Q}$ was studied, the authors obtained some partial information on the orbit closure $Y_{g}=H_{g}/(H_{g}\cap(\Gamma\times\Gamma))$ of $(g(pn)\cdot e_{X}, g(qn)\cdot e_{X})_{n\in\Z}$, showing that 
\begin{equation}\label{uuqq}
 (u^{p^{s}},u^{q^{s}})\in H_{g} \text{ for all } u\in G_{s}, 
\end{equation}
and they then used (\ref{uuqq}) to study (\ref{idea}).  However, in the case when $[K\colon \mathbb{Q}]\geq 2$, we do not have an analog of (\ref{uuqq}).
In fact, in this case
the structure of $Y_{g}$ is much more complicated, as it depends not only on $p,q$, but also on the coefficients of $g$. In \cite{S}, we were able to partially classify the structure of $Y_{g}$ for the case when $K=\mathbb{Q}(i)$ and $g$ is of degree 2. Even in this case, the computations were already very heavy. The task of classifying the structure of $Y_{g}$ for the general case is even more difficult.

The major innovation of this paper is that we find a new way to extract partial information on $Y_{g}$ which is substantially different from the work of \cite{FH}. Instead of considering the orbit closure $Y_{g}$ for all polynomial sequences $g$, we restrict ourselves to those $g$ such that $\int_{Y_{g}}\Phi\otimes\overline{\Phi}\neq 0$. The special structure of the function  $\Phi\otimes\overline{\Phi}$ provides us some extra information which can be used to describe $Y_{g}$.
To be more precise, in Section \ref{s:d}, we show that if the average of  (\ref{idea}) is bounded away from zero, then $Y$ must satisfy some algebraic condition (Theorem \ref{auto}).  Although this condition is not enough to provide a complete description of the structure of $Y_{g}$, it reduces the problem to solving a specific algebraic equation $g(pn)=\sigma\circ g(qn) \mod G_{\ker}$, where $\sigma$ is a $d$-automorphism (See section \ref{s:d} for the definitions).
By  using of the multi-linearity property of $\sigma$, we then show 
  that this equation has no solution is $g(n)$ is equidistributed (Theorem \ref{contri}), which leads to a contradiction to the initial assumption of $g$.

\begin{rem}[Overlapping with literature]
Due to the unavoidable formalism in the proofs of the results, many parts of this paper have overlapping with \cite{FH,GT,S}. To be more precise, Sections \ref{s:3}, \ref{s:4}, \ref{s:o}, \ref{s:p}, \ref{s:s} and \ref{s:par} partially overlap with \cite{FH,S} (Sections \ref{s:3} and \ref{s:4} also partially overlap with \cite{GT}); Sections \ref{s:2}, \ref{s:d} and \ref{s:m} are completely new and have no counterparts in  \cite{FH,GT,S}.   	
\end{rem}	

\subsection{Open questions}\label{12345} 
For multiplicative functions on number fields,
there are many natural questions in addition to  Sarnak's Conjecture (Conjecture \ref{sarnak}). For example, one can ask the logarithm  Sarnak's Conjecture:
\begin{ques}[Generalized logarithm  Sarnak's Conjecture]\label{C1}
	Let $\K$ be an integral tuple. For which $\chi\in\mathcal{M}^{a}_{K}$ does the following hold: for every topological system $(X,T_{1},\dots,T_{d})$  with commuting transformations $T_{1},\dots,T_{d}$ with zero topological entropy, every $\Phi\in C(X)$ and every $x\in X$, we have that
	$$\lim_{N\to\infty}\frac{1}{(\log N)^{D}}\sum_{1\leq n_{1},\dots,n_{D}\leq N}\frac{\chi(n_{1}b_{1}+\dots+n_{D}b_{D})\Phi(T^{n_{1}}_{1}\cdot\ldots\cdot T^{n_{D}}_{D}x)}{n_{1}n_{2}\dots n_{D}}=0.$$
\end{ques}
	
 It is worth noting that for the case $K=\mathbb{Q}$, Question \ref{C1} (and Question \ref{C2} below) is not true for all aperiodic multiplicative functions (to see this, one can use the example in Theorem B.1 of \cite{counter}). 
For the case $K=\mathbb{Q}$, under certain ergodicity assumption of the system, Conjecture \ref{C1} was proved by Frantzikinakis and Host in \cite{FH1} when $\chi$ is the M\"obius function, and then in \cite{FH2} when $\chi$ is strongly aperiodic.

It is also natural to ask the analog of  Chowla's Conjecture:
\begin{ques}[Generalized Chowla's (and logarithm  Chowla's) Conjecture]\label{C2}
Let $\bold{K}=(K,\mathcal{O}_{K},D,$ $\B=\{b_{1},\dots,b_{D}\})$ be an integral tuple. For which $\chi\in\mathcal{M}^{a}_{K}$ does the following hold: for every $k\in\N_{+}$ and $m_{1},\dots,m_{k}\in\O_{K}$ which are pairwise distinct, we have that
$$\lim_{N\to\infty}\frac{1}{N^{D}}\sum_{n=n_{1}b_{1}+\dots+n_{D}b_{D},1\leq n_{1},\dots,n_{D}\leq N}\chi(n+m_{1})\chi(n+m_{2})\dots \chi(n+m_{k})=0$$	
(or
$$\lim_{N\to\infty}\frac{1}{(\log N)^{D}}\sum_{n=n_{1}b_{1}+\dots+n_{D}b_{D},1\leq n_{1},\dots,n_{D}\leq N}\frac{\chi(n+m_{1})\chi(n+m_{2})\dots \chi(n+m_{k})}{n_{1}n_{2}\dots n_{D}}=0.$$
for the logarithm version).	
\end{ques}
For $K=\mathbb{Q}$, it is not hard to show that Chowla's Conjecture for $\chi=\mu$ implies Sarnak's Conjecture (see for example \cite{web3}). Moreover, if $\chi$ is the M\"obius function $\mu$ on $\Z$, then logarithm  Chowla's Conjecture is equivalent to the logarithm Sarnak's Conjecture \cite{66}, the former of which is known to be true when $k=2$ \cite{65} and when $k$ is an odd number \cite{67}. 

Another natural question is whether Sarnak's Conjecture holds in the measure theoretic setting: 
\begin{conj}[Generalized measurable  Sarnak's Conjecture]\label{C4}
	Let $\bold{K}=(K,\mathcal{O}_{K},D,\B=\{b_{1},\dots,$ $b_{D}\})$ be an integral tuple and $\chi\in\mathcal{M}^{a}_{K}$. Then for every measure preserving system $(X,\mu,T_{1},\dots,T_{d})$  with commuting transformations $T_{1},\dots,T_{d}$,\footnote{Note that there is no assumption on entropy in this conjecture.} every $\Phi\in L^{\infty}(\mu)$ and $\mu$-a.e. $x\in X$,
	$$\lim_{N\to\infty}\frac{1}{N^{D}}\sum_{1\leq n_{1},\dots,n_{D}\leq N}\chi(n_{1}b_{1}+\dots+n_{D}b_{D})\Phi(T^{n_{1}}_{1}\cdot\ldots\cdot T^{n_{D}}_{D}x)=0.$$
\end{conj}
One possible approach to prove Conjecture \ref{C4} for 
the case $K=\mathbb{Q}$ is to combine K\'atai's Lemma (see for example Lemma \ref{katai}), Bourgain's double pointwise convergence theorem \cite{Bourgain}, the Host-Kra structure theorem (Theorem 10.1 of \cite{HK}), and the orthogonality of multiplicative functions and nilsequences (see Theorem 2.5 of \cite{FH}, or Theorem \ref{s2} in this paper). In  Theorem 3.1 of \cite{ab}, by using a result from Green and Tao \cite{7}, another proof was given for Conjecture \ref{C4} for 
the case when $K=\mathbb{Q}$ and $\chi$ is the M\"obius function.  The case when $K$ is a number field other than $\mathbb{Q}$ remains open.

Finally, one can also ask all the above questions for some specific functions which are interesting in number theory and additive combinatorics. For example,
one can define the M\"obius function $\mu_{K}\colon K\to \{-1,0,1\}$ on $\O_{K}$ by letting $\mu_{K}(n)=(-1)^{k}$ if the ideal $(n)$ is the product of $k$ distinct prime ideals, and $\mu_{K}(n)=0$ otherwise. The function $\mu_{K}$ is well defined since $\O_{K}$ is a Dedekind domain, and is always multiplicative on $\O_{K}$, yet it is apriori unclear whether $\mu_{K}$ is aperiodic. 
So it is natural to ask:
\begin{ques}
	On which integral tuple $\K$ is the M\"obius function $\mu_{K}$ aperiodic? For such $\bold{K}$, are the answers to  Questions   \ref{sarnak}, \ref{C1} and  \ref{C2} affirmative for $\chi=\mu_{K}$?  Does Conjecture \ref{C4} hold for the special case $\chi=\mu_{K}$?
\end{ques}	 

\subsection{Notations}
We introduce the notations we use in this paper.
\begin{itemize}
	\item In this paper, unless a vector is written explicitly in the vertical way, all vectors are assumed to be horizontal.
	\item $\mathcal{M}_{K}$ and $\mathcal{M}^{a}_{K}$ are the sets of multiplicative and aperiodic multiplicative functions on $\O_{K}$ with modulus at most 1, respectively. 
	\item For $N\in\N_{+}$, denote $[N]:=\{1,\dots,N\}$ and $\Z_{N}:=\Z/N\Z$. 
	\item For $D\in\N_{+}$ and $N\geq 0$, denote 
	$$R_{N,D}:=\{(n_{1},\dots,n_{D})\in\Z^{D}\colon \vert n_{i}\vert\leq N, 1\leq i\leq D\}.$$
	\item Let $a\colon V\to\mathbb{C}$ be a map with $V$ being a finite set, denote 
	$$\E_{x\in V}a(x):=\frac{1}{\vert V\vert}\sum_{x\in V}a(x).$$
	\item Let $a,b\colon \N\to\mathbb{C}$ and $\mathcal{I}$ be a collection of parameters.
	The notion $C:=C(\mathcal{I})$ means that $C$ is a quantity depending only on the parameters in $\mathcal{I}$.
	We write $a\ll_{\mathcal{I}} b$ if there exist $C:=C(\mathcal{I})>0, N_{0}:=N_{0}(\mathcal{I})>0$ such that $a(N)\leq Cb(N)$ for all $N\geq N_{0}$. We write $a(N)=o_{\mathcal{I}}(b(N))$ if for every $\epsilon>0$, there exists $N_{0}:=N_{0}(\mathcal{I},\epsilon)>0$ such that $\vert a(N)\vert\leq \epsilon b(N)$ for all $N\geq N_{0}$.
	\item For $x\in\R$, $\lfloor x \rfloor$ is the largest integer which is not larger than $x$, and $\lceil x\rceil$ is the smallest integer which is not smaller than $x$.
	\item Let $e\colon\R\to\mathbb{C}$ denote the function $e(x):=e^{2\pi ix}$.
	\item For $i\in\N_{+}$, $\bold{e}_{i}$ denotes the vector whose $i$-th coordinate is 1 and all other coordinates are 0 (the dimension of $\bold{e}_{i}$ will be clear from the context). 
	\item For a vector $\v=(v_{1},\dots,v_{m})\in\mathbb{C}^{m}$ for some $m\in\N_{+}$, denote $\vert \v\vert:=\vert v_{1}\vert+\dots+\vert v_{m}\vert$.
	\item For $\v=(v_{1},\dots,v_{m})\in\mathbb{R}^{m}$, let $$\Vert \v\Vert_{\T^{m}}:= \inf_{\n=(n_{1},\dots,n_{m})\in\Z}\vert \v-\n\vert.$$ For $\w=\bold{u}+\bold{v}i\in\mathbb{C}^{m}$ for some $\bold{u},\bold{v}\in\R^{m}$, let
	$\Vert \w\Vert_{\T^{m}}:=\Vert \bold{u}\Vert_{\T^{m}}+\Vert \bold{v}\Vert_{\T^{m}}$.
	\item Let $m,n\in\Z$ and $R$ be a ring. Then $M_{m,n}(R)$ denote all the $m\times n$ matrices whose entries take values from $R$. Let $\mathcal{O}_{m\times n}$ denote the $m\times n$ matrix whose all entries are 0. 
	\item  Throughout this paper, for $A\in M_{s\times s}(\mathbb{R})$, we use $A\colon\mathbb{R}^{s}\to\mathbb{R}^{s}$ to denote the map given by $A(\x):=\x\cdot A$, $\x\in\mathbb{R}^{s}$, i.e. the \textbf{right} multiplication of $A$.
\end{itemize}	

\textbf{Acknowledgment.} We thank Bernard Host, Nikos Frantzikinakis and Bryna Kra for helpful comments. We thank Nikos Frantzikinakis for the discussion on the precise formulations of the Sarnak's Conjecture related questions stated in Section \ref{s:1}.  We also thank the anonymous referees for their suggestions, which were very helpful in improving the presentation of the paper. 

\section{Ingredients from algebraic number theory}\label{s:2}
\subsection{Algebraic number field and minimal polynomials}
	
\begin{defn}[Integral tuple]
	An \emph{(algebraic) number field} $K$ is a finite degree and (hence algebraic) field extension of the field of rational numbers $\mathbb{Q}$. The \emph{ring of integers} $\mathcal{O}_{K}$ of a number field $K$ is the ring of all integral elements contained in $K$.\footnote{An \emph{integral element} is a root of polynomial with integer coefficients and leading coefficient 1.}
	Denote
	$D=[K\colon \mathbb{Q}]$.\footnote{If $K$ is a field extension of $L$, then $[K\colon L]$ denotes the degree of this extension.} It is classical that there exists an \emph{integral basis} $\mathcal{B}=\{b_{1},\dots,b_{D}\}$ of $\mathcal{O}_{K}$, i.e. a basis $b_{1},\dots,b_{D}\in\mathcal{O}_{K}$ of the $\mathbb{Q}$-vector space $K$ such that each element $x\in \mathcal{O}_{K}$ can be uniquely represented as 
	$x=\sum_{i=1}^{D}a_{i}b_{i}$ for some $a_{i}\in\mathbb{Z}$.  
	We call $\K$ an \emph{integral tuple}.	
\end{defn}

Let $\K$ be an integral tuple.  Recall that $\iota_{\B}\colon\Z^{D}\to\O_{K}$ is the bijection given by $\iota_{\B}(n_{1},\dots,n_{D})=n_{1}b_{1}+\dots+n_{D}b_{D}$.
For $x\in K$, let $A_{\mathcal{B}}(x)\in M_{D\times D}(\mathbb{Q})$ be the unique matrix such that 
$\begin{bmatrix} 
xb_{1}\\
xb_{2}\\
\dots\\
xb_{D}
\end{bmatrix}=A_{\mathcal{B}}(x)\begin{bmatrix} 
b_{1}\\
b_{2}\\
\dots\\
b_{D}
\end{bmatrix}
$. This implies that for all $x,y\in K$,
\begin{equation}\label{iota}
	\iota^{-1}_{\B}(xy)=\iota^{-1}_{\B}(y)A_{\B}(x).
\end{equation}	
We remark that $A_{\B}(x)\in M_{D\times D}(\mathbb{Z})$ if $x\in \O_{K}$. 

\begin{defn}[$K$-norm]\label{knorm}
	The \emph{$K$-norm}  
	of $x\in K$ is $N_{K}(x):=\det(A_{\mathcal{B}}(x))$. 
\end{defn}	
Note that $N_{K}(x)$ is independent of the choice of the basis $\mathcal{B}$.

We say that a polynomial $f\in\mathbb{Q}[x]$ is \emph{monic} if the leading coefficient of $f$ is 1. We say that $f\in\mathbb{Q}[x]$ is \emph{irreducible} if $f=gh, g,h\in\mathbb{Q}[x]$ implies that one of $g$ and $h$ is a constant. We say that $f\in\mathbb{Q}[x]$ is the \emph{minimal polynomial} of an algebraic number $x$ (or a matrix $A\in M_{s\times s}(\mathbb{Q})$) if $f$ is a monic polynomial of the smallest possible positive degree such that $f(x)=0$ (or $f(A)=\mathcal{O}_{s\times s}$).

The following lemma is classical and we omit the proof:
\begin{lem}[Properties on the $K$-norm]\label{norm}
	Let $\K$ be an integral tuple. Then 
	\begin{enumerate}[(i)]
		\item 	
		For all $x,y\in K$ and $q\in\mathbb{Q}$, we have that $A_{\B}(x)A_{\B}(y)=A_{\B}(y)A_{\B}(x)=A_{\B}(xy)$, $A_{\B}(x)+A_{\B}(y)=A_{\B}(x+y)$ and $A_{\B}(qx)=qA_{\B}(x)$. In particular,
		$N_{K}(xy)=N_{K}(x)N_{K}(y)$;
		\item If $K/\mathbb{Q}$ is a normal extension and $f$ is the minimal polynomial of some $x\in K\backslash\{0\}$, then $N_{K}(x)=(-1)^{[K\colon\mathbb{Q}]}f(0)^{[K\colon\mathbb{Q}]/\deg(f)}$. 
	\end{enumerate}	
\end{lem}	

The following are some basic properties about minimal polynomials:

\begin{lem}[Minimal polynomials]\label{minimal}
	Let $\K$ be an integral tuple and $\overline{K}$ be the normal closure of $K$.  Let $x\in K\backslash\{0\}$ and $f$ denote the minimal polynomial of the matrix $A_{\B}(x)$. Then 
	\begin{enumerate}[(i)]
		\item $f$ is also the minimal polynomial of $x$;
		\item $f$ is irreducible, and has no repeated roots;
		\item $y\in \overline{K}$ is a root of $f$ if and only if $y$ is an eigenvalue of $A_{\B}(x)$. In particular, by (ii), all the eigenvalues of $A_{\B}(x)$ are distinct; 
			\item All the roots of $f$ have the same $\overline{K}$-norm as $x$.	
	\end{enumerate}		
\end{lem}
\begin{proof}
	(i) Let $g\in\mathbb{Q}[x]$ be any polynomial. 
	By Lemma \ref{norm} (i),
	$A_{\B}(g(x))=g(A_{\B}(x))$. Then
	$$g(x)=0 \Leftrightarrow A_{\B}(g(x))=\O_{D\times D} \Leftrightarrow g(A_{\B}(x))=\O_{D\times D}.$$
	  So the minimal polynomial $f$ of $A_{\B}(x)$ is also the minimal polynomial of $x$. 
	
	(ii) Since the minimal polynomial of $x$ is irreducible, by (i), $f$ is irreducible. Since $f'\not\equiv 0$, $f$ is coprime with $f'$ and so $f$ has no repeated roots.
	
	(iii) If $y$ is an eigenvalue of $A_{\mathcal{B}}(x)$, then we may assume that $vA_{\mathcal{B}}(x)=yv$ for some nonzero $v\in\mathbb{C}^{D}$. So $f(y)v=vf(A_{\mathcal{B}}(x))=\bold{0}$. Since $v$ is nonzero, we have that $f(y)=0$.
	
	Conversely, suppose $f(y)=0$. Let $g(\lambda):=\det(\lambda I_{D}-A_{\mathcal{B}}(x))$.
	By the  Cayley-Hamilton Theorem, $g(A_{\mathcal{B}}(x))=0$. By the minimality of $f$, we have that $f\vert g$. So $f(y)=0$ implies that $g(y)=0$, meaning that $y$ is an eigenvalue of $A_{\mathcal{B}}(x)$.

	(iv) By Lemma \ref{norm} (ii), the  $\overline{K}$-norm of all roots of $f$ equal to $(-1)^{[\overline{K}\colon\mathbb{Q}]}f(0)^{[\overline{K}\colon\mathbb{Q}]/\deg(f)}$, which are the same.
\end{proof}

The next is a characterization for diagonalizable matrices, which is used in later sections.

\begin{lem}[A characterization for diagonalizable matrices]\label{jordan}
	Let $f\in\mathbb{C}[x]$ be a non-constant polynomial with no repeated roots (in $\mathbb{C}$).
	Let $s\in\N_{+}$ and $B$ be an $s\times s$ matrix such that $f(B)=\O_{s\times s}$ for some $s\in\N_{+}$. Then there exist an $s\times s$ invertible matrix $S$ and a  diagonal matrix $J=\begin{bmatrix}
	\mu_{1}\\
	& \mu_{2}\\
	& & \dots\\
	& & & \mu_{s}
	\end{bmatrix}$ with $f(\mu_{1})=\dots=f(\mu_{d})=0$ such that $B=SJS^{-1}$.
\end{lem}	
\begin{proof}
	The case when $s=1$ is straightforward. So we assume that $s\geq 2$. 
	
 Let $E_{i}$ denote the $s\times s$ matrix whose $(k,k+i)$-th entry is 1 for all $1\leq k\leq s-i$ and all the other entries are 0.
	Converting $B$ to the Jordan normal form, it suffices to show that if  $f(B)=\O_{s\times s}$ and $B=\mu I_{s}$ or $B=\mu I_{s}+E_{1}$ for some $\mu\in\mathbb{C}$, then $B=\mu I_{s}$ and $f(\mu)=0$.
	
	In fact, if $B=\mu I_{s}$, then $f(B)=f(\mu)I_{s}$. So $f(\mu)=0$. If  $B=\mu I_{s}+E_{1}$, note that $E_{i}E_{j}=E_{i+j}$ for all $1\leq i,j\leq s-1$ (for convenience denote $E_{i}=\O_{s\times s}$ if $i\geq s$). Writing $f(x)=\sum_{n=0}^{M}a_{n}x^{n}$ for some $M\geq 1$, we have that 
	\begin{equation}\nonumber
		\begin{split}
			&\quad f(B)=\sum_{n=0}^{M}a_{n}\sum_{i=0}^{n}\mu^{n-i}\binom{n}{i}E_{i}
			=\sum_{i=0}^{M}E_{i}\Bigl(\sum_{n=i}^{M}\binom{n}{i}a_{n}\mu^{n-i}\Bigr)=0,
		\end{split}
	\end{equation}
	where $E_{0}=I_{s}$.
	This implies that $\sum_{n=i}^{M}\binom{n}{i}a_{n}\mu^{n-i}=0$ for all $0\leq i\leq \min\{M,s-1\}$. By assumption, $\min\{M,s-1\}\geq 1$.  Setting $i=0$, $f(\mu)=\sum_{n=0}^{M}a_{n}\mu^{n}=0$. Setting $i=1$, $f'(\mu)=\sum_{n=1}^{M}na_{n}\mu^{n-1}=0$. A contradiction to the fact that $f$ has no repeated roots. This finishes the proof.
\end{proof}

\subsection{Ideals and unique factorization}
\begin{defn}[Units]
	We say that $\epsilon$ is a \emph{unit} of $\O_{K}$ if there exists $\epsilon'\in\O_{K}$ such that $\epsilon\epsilon'=1$. Since $N_{K}(x)\in\mathbb{Z}$ for all $x\in \O_{K}$ and $N_{K}(1)=1$, it is not hard to see that the $K$-norm of a unit is $\pm 1$. 	
\end{defn}

Let $K$ be a number field and $\mathcal{O}_{K}$ be its ring of integers.
A subset $I\subseteq \O_{K}$ is an \emph{ideal} of $\O_{K}$ if for all $x,y\in I, z\in \O_{K}$, we have that $x-y, xz\in I$. An ideal $I$ of $\O_{K}$ is  \emph{principal} if there exists $c\in\O_{K}$ such that $I=\{cx\colon x\in\O_{K}\}$. If $I$ is generated by $a_{1},\dots,a_{k}\in \O_{K}$, meaning that $I=\{c_{1}a_{1}+\dots+c_{k}a_{k}\colon c_{1},\dots,c_{k}\in\O_{K}\}$, we then write $I=(a_{1},\dots,a_{k})$ for short. Since $\O_{K}$ is Noetherian, by Theorem 5.2.3 of \cite{problem}, every ideal of $\O_{K}$ is finitely generated, and so we may always write $I$ as $I=(a_{1},\dots,a_{k})$ for some $k\in\N_{+}$ and $a_{1},\dots,a_{k}\in\O_{K}$.

There are many different concepts of primes in a number field, which we clarify as follows:

\begin{defn}[Primes]\label{prime}
	Let $\K$ be an integral tuple.
	\begin{itemize}
		\item $x\in\N_{+}$ is a \emph{prime integer} if $x$ is a prime in the field $\mathbb{Q}$.
		\item An ideal $I$ of $\O_{K}$ is a \emph{prime ideal} if $I\neq \O_{K}$ for all $x,y\in \O_{K}$ such that $xy\in I$, either $x\in I$ or $y\in I$.
		\item $p\in\O_{K}$ is a \emph{prime element} if $(p)$ is a prime ideal.
		\item for $p,q \in\O_{K}$, we say that $p$ \emph{divides} $q$ (written as $q\vert p$) if $pq^{-1}\in\O_{K}$. 
	\end{itemize}	
\end{defn}

Let $I\subseteq \O_{K}$ be an ideal  of $\O_{K}$. Then $\O_{K}/ I$ is a finite set (see for example Exercise 4.4.3 of \cite{problem}), and the cardinality of this set is called the \emph{index} of $I$ in $\O_{K}$, or the \emph{$K$-norm} of the ideal $I$ (denoted as $N(I)$). 

The following lemma is standard (see for example  Exercise 5.3.15 of \cite{problem}):

\begin{lem}[Norms of ideals]\label{normii}
	Let $\K$ be an integral tuple and $I=(a)$ be a principal ideal of $\O_{K}$ for some $a\in\O_{K}$. The $K$-norm $N(I)$ of $I$ coincides with the absolute value of the $K$-norm $\vert N_{K}(a)\vert$ of $a$. 
\end{lem}

It is well known that every ideal $I$ of $\O_{K}$ can be factorized into the form $$I=\prod_{i=1}^{m}I^{e_{i}}_{i}$$
for some $m\in\N_{+}$, $e_{i}\in \N_{+}$, prime ideal $I_{i}$ for all $1\leq i\leq m$ 
in a unique way (modulo the order of the ideals $I_{i}$).

\begin{lem}[Properties of prime ideals]\label{pi}
	For every prime integer $p\in\N$, let  $\mathcal{J}_{p}$ denote the collection (possibly an empty collection) of prime ideals of $\O_{K}$ which contains $(p)$. Then 
	\begin{enumerate}[(i)]
		\item for all $J\in\mathcal{J}_{p}$, $N(J)=p^{f}$ for some $1\leq f\leq D$.
		\item  $\mathcal{J}_{p}$ is a finite set of cardinality at most $D$.
		\item every principal prime ideal of $\O_{K}$ belongs to some $\mathcal{J}_{p}$. In particular, the $K$-norm of every principal prime ideal is a power of a prime integer, and the $K$-norms of principal prime ideals from different $\mathcal{J}_{p}$ are coprime (in $\Z$).
	\end{enumerate}	
\end{lem}	
\begin{proof}
	(i) and (ii). Suppose that $$(p)=\prod_{i=1}^{m}I^{e_{i}}_{i}$$
	for some $m\in\N_{+}$, $e_{i}\in \N_{+}$, prime ideal $I_{i}$ for all $1\leq i\leq m$. By the unique factorization of $(p)$, $\mathcal{J}_{p}=\{I_{1},\dots,I_{m}\}$ and $\vert\mathcal{J}_{P}\vert=m$. Since $p\in\N$,
	$$p^{D}=\vert N_{K}(p)\vert=\prod_{i=1}^{m}N(I_{i})^{e_{i}}.$$
	So $N(I_{i})=p^{f_{i}}$ for some $f_{i}\in\N_{+}$ for all $1\leq i\leq m$. Therefore,
	$$m\leq \sum_{i=1}^{m}e_{i}f_{i}=D.$$
	
	(iii) Let $J=(a)$ be a principal prime ideal for some $a\in\O_{K}$. 
	We first claim that there exists $n\in\N_{+}$ such that $a\vert n$. 	
	Let $f(x)=\sum_{n=0}^{M}a_{n}x^{n}\in\mathbb{Q}[x]$ be the minimal polynomial of $a$. Pick $C\in\N_{+}$ such that $Cf\in\Z[x]$.
	Then $$-a(\sum_{n=1}^{M}Ca_{n}a^{n-1})=Cf(0)\in\Z.$$
	Since $Ca_{n},Cf(0)$ and $a\in\O_{K}$, we have that $ay\in\Z\backslash{\{0\}}$ for   $y:=\sum_{n=1}^{M}Ca_{n}a^{n-1}\in\O_{K}$. This implies that $a\vert n$ and finishes the proof of the claim.

	Since $J$ is a prime ideal, by the unique factorization of $(n)$, there exists a prime integer $p$ (dividing $n$) such that $J$ contains $(p)$. So $J\in\mathcal{J}_{p}$. 
\end{proof}

\subsection{Regularization of algebraic numbers}

Let $\K$ be an integral tuple. For $N\in\N_{+}$ and $a\in \O_{K}$, denote
$$R_{N,\B}:=\{z\in\B\colon \iota_{\B}^{-1}(z)\in R_{N,D}\}$$
and
$$a^{-1}R_{N,\B}:=\{z\in\O_{K}\colon az\in R_{N,\B}\}$$
throughout this section.\footnote{We use the notation $R_{N,D}$ in all other sections, but use $R_{N,\B}$ in this section as it is more convenient.} We caution the reader that the set $a^{-1}R_{N,\B}$ is a subset of $\O_{K}$, and is NOT the set of $z\in K$ such that $az\in R_{N,\B}$. 

We need to use the following estimate of the density of ideals frequently in this section:

\begin{lem}[Density of ideals]\label{normi}
	Let $\K$ be an integral tuple and $I$ be an ideal of $\O_{K}$.  We have that 
		\begin{equation}\nonumber
			\lim_{N\to\infty}\frac{\vert I\cap R_{N,\B}\vert}{(2N+1)^{D}}=\frac{1}{N(I)}.
		\end{equation}
\end{lem}	
\begin{proof}
There exist $d_{1},\dots,d_{N(I)}\in \O_{K}$ with $d_{1}=0$ such that $d_{i}+I$ are disjoint subsets of $\O_{K}$ and their union is $\O_{K}$.
	Suppose that $I=(a_{1},\dots,a_{k})$ 
	for some $k\in\N_{+}$ and $a_{1},\dots,a_{k}\in\O_{K}$. Let $M$ be a positive integer such that $a_{i}^{-1}M\in\O_{K}$ for all $1\leq i\leq k$. Then there exists a constant $C>0$ such that for all $x\in\O_{K}$, the cardinality of the set $(x+I)\cap R_{M,\B}$ equals to $(2M+1)^{D}C$. For $N\in\N$, by partitioning $R_{N,\B}$ into smaller cubes of the form $R_{M,\B}$, it is easy to see that the limit $\lim_{N\to\infty}\frac{\vert (d_{i}+I)\cap R_{N,\B}\vert}{(2N+1)^{D}}$  exists and equals to $C$ for all $1\leq i\leq N(I)$. Since the union of $d_{i}+I, 1\leq i\leq N(I)$ is $\O_{K}$, we have that $C=\frac{1}{N(I)}$. 
\end{proof}

If $a\in\Z\backslash\{0\}$, then clearly $a^{-1}R_{N,\B}$ is a subset of $R_{N/a,\B}$,  which is a cube with $\vert N_{K}(a)\vert=a^{-D}$ of the size of $R_{N,\B}$. However, this is not the case when $a\notin\Z$.
\begin{ex}
	Consider the integral tuple $$\K=(\mathbb{Q}(\sqrt{2}),\mathbb{Z}[\sqrt{2}],2,\{1,\sqrt{2}\}).$$
	For all $m,n\in\Z$, $A_{\B}(m+n\sqrt{2})=\begin{bmatrix}
	m & n\\
	2n & m
	\end{bmatrix}$ and it has two real eigenvalues $m\pm n\sqrt{2}$. Let $a=(2+\sqrt{2})^{4}$. Then $\det(A_{\B}(a))=16$, and the two eigenvalues of  $A_{\B}(a)$ are respectively $(2+\sqrt{2})^{4}\approx 135.882$ and $(2-\sqrt{2})^{4}\approx 0.118$.
	Although the ``volume" of $a^{-1}R_{N,\B}$ is approximately 1/16 of that of $R_{N,\B}$, $a^{-1}R_{N,\B}$ is not contained in $R_{N,\B}$ (it is only contained in a much larger rectangle $R_{136N,\B}$).
	
	On the other hand, if we multiply $a$ with the unit $\epsilon=(-1+\sqrt{2})^{4}$ and denote $a':=a\epsilon=4$, then ${a'}^{-1}R_{N,\B}\subseteq R_{N/4,\B}$, which is $\vert N_{K}(a)\vert=1/16$ the size of $R_{N,\B}$.
\end{ex}

We prove the following theorem in this section, which generalizes the phenomena appeared in the previous example. 

\begin{defn}[$C$-regular number]
	Let $\K$ be an integral tuple and $C>0$. We say that $a\in\O_{K}$ is \emph{$C$-regular} if for all $N\in\N$, we have that 
	$$a^{-1}R_{N,\B}\subseteq R_{C\vert N_{K}(a)\vert^{-\frac{1}{D}}N,\B}.\footnote{We are unaware whether $C$-regular number is a common notion in algebraic number theory. 
		The notation of $C$-regular numbers will be used essentially in Lemma \ref{katai0}.}$$
\end{defn}

\begin{thm}[Regularization of algebraic numbers]\label{mink2}
	Let $\K$ be an integral tuple. There exists a constant $C_{\B}>0$ depending only on $\B$ such that
	for every $a\in K$, there exists a unit $\epsilon$ of $\O_{K}$ such that $\epsilon a$ is $C_{\B}$-regular.
\end{thm}	

We start with a structure theorem of the eigenspaces of $A_{\B}(x)$. 

\begin{lem}[Structures for the eigenspaces of $A_{\B}(x)$]\label{di}
	Let $\K$ be  an integral tuple. Then there exist $r_{1},r_{2}\in\N$ with $r_{1}+2r_{2}=D$, and a decomposition of $\mathbb{C}^{D}$ into 1 dimensional subspaces $$\mathbb{C}^{D}=\mathbb{C}v_{1} \oplus\dots\oplus \mathbb{C}v_{r_{1}}\oplus (\mathbb{C}v_{r_{1}+1}\oplus \mathbb{C}\overline{v}_{r_{1}+1})\oplus\dots\oplus(\mathbb{C}v_{r_{1}+r_{2}}\oplus \mathbb{C}\overline{v}_{r_{1}+r_{2}})$$
	for some $v_{1},\dots,v_{r_{1}}\in\mathbb{R}^{D}$ and $v_{r_{1}+1},\dots,v_{r_{1}+r_{2}}\in \mathbb{C}^{D}$
	such that for all $x\in K$, there exist $\lambda_{1}(x),\dots,$ $\lambda_{r_{1}}(x)\in\mathbb{R}$ and $\lambda_{r_{1}+1}(x),\dots,\lambda_{r_{1}+r_{2}}(x)\in\mathbb{C}$ such that $v_{i}A_{\B}(x)=\lambda_{i}(x)v_{i}$ for all $1\leq i\leq r_{1}+r_{2}$. Moreover, $\lambda_{i}(x_{1})\lambda_{i}(x_{2})=\lambda_{i}(x_{1}x_{2})$ for all $x_{1},x_{2}\in K$ and $1\leq i\leq r_{1}+r_{2}$.
\end{lem}	
\begin{proof}
	Let $y\in K$ be any number the degree of whose minimal polynomial is $D$ (the existence of such $y$ is guaranteed by the Theorem of the Primitive Element, see for example Theorem 3.3.2 of \cite{problem}). By Lemma \ref{minimal}, $A_{\B}(y)$ has 
	$D$ distinct eigenvalues. So we may  decompose $\mathbb{C}^{D}$ into 1 dimensional subspaces $$\mathbb{C}^{D}=\mathbb{C}v_{1} \oplus\dots\oplus \mathbb{C}v_{D}$$
	for some $v_{1},\dots,v_{D}\in\mathbb{C}^{D}$
	such that $v_{i}A_{\B}(y)=\lambda_{i}(y)v_{i}$ for some $\lambda_{i}(y)\in\mathbb{C}$ for all $1\leq i\leq D$. 
	Since $A_{\B}(x)$ is a matrix with real coefficients, complex eigenvalues and eigenvectors come in pairs, and so we may assume that 
	there exist $r_{1},r_{2}\in\N$ with $r_{1}+2r_{2}=D$, such that $$\mathbb{C}^{D}=\mathbb{C}v_{1} \oplus\dots\oplus \mathbb{C}v_{r_{1}}\oplus (\mathbb{C}v_{r_{1}+1}\oplus \mathbb{C}\overline{v}_{r_{1}+1})\oplus\dots\oplus(\mathbb{C}v_{r_{1}+r_{2}}\oplus \mathbb{C}\overline{v}_{r_{1}+r_{2}})$$
	for some $v_{1},\dots,v_{r_{1}}\in\mathbb{R}^{D}$ and $v_{r_{1}+1},\dots,v_{r_{1}+r_{2}}\in \mathbb{C}^{D}$, $\lambda_{1}(y),\dots,\lambda_{r_{1}}(y)\in\mathbb{R}$ and $\lambda_{r_{1}+1}(y),\dots,$ $\lambda_{r_{1}+r_{2}}(y)\in\mathbb{C}$ such that $v_{i}A_{\B}(y)=\lambda_{i}(y)v_{i}$ for all $1\leq i\leq r_{1}+r_{2}$.

	Let $x\in K$. By Lemma \ref{norm} (i), for all $1\leq i\leq r_{1}+r_{2}$,
	$$(v_{i}A_{\B}(x))A_{\B}(y)=(v_{i}A_{\B}(y))A_{\B}(x)=\lambda_{i}(y)(v_{i}A_{\B}(x)).$$
	So both $v_{i}$ and $v_{i}A_{\B}(x)$ are eigenvectors of eigenvalue $\lambda_{i}(y)$ for the matrix $A_{\B}(y)$.
	Since $\lambda_{1}(y),\dots,$ $\lambda_{D}(y)$ are distinct and the eigenspace of every eigenvalue of $A_{\B}(y)$ is 1-dimensional, we have that $v_{i}A_{\B}(x)\in\mathbb{R}v_{i}$
	for $1\leq i\leq r_{1}$ and  $v_{i}A_{\B}(x)\in\mathbb{C}v_{i}$
	for $r_{1}\leq i\leq r_{1}+r_{2}$.
	So for all $1\leq i\leq r_{1}+r_{2}$, there exist  $\lambda_{1}(x),\dots,\lambda_{r_{1}}(x)\in\mathbb{R}$ and $\lambda_{r_{1}+1}(x),\dots,\lambda_{r_{1}+r_{2}}(x)\in\mathbb{C}$ such that $v_{i}A_{\B}(x)=\lambda_{i}(x)v_{i}$. 
	
	For $x_{1},x_{2}\in K$ and $1\leq i\leq r_{1}+r_{2}$, note that 
	$$\lambda_{i}(x_{1}x_{2})v_{i}=v_{i}A_{\B}(x_{1}x_{2})=v_{i}A_{\B}(x_{1})A_{\B}(x_{2})=\lambda_{i}(x_{1})v_{i}A_{\B}(x_{2})=\lambda_{i}(x_{1})\lambda_{i}(x_{2})v_{i}.$$
	So $\lambda_{i}(x_{1})\lambda_{i}(x_{2})=\lambda_{i}(x_{1}x_{2})$.
\end{proof}	

Let the notations be as in Lemma \ref{di}, and
 denote $r=r_{1}+r_{2}-1.$
Let
 $W\colon K\to \mathbb{R}^{r}$ be the map given by
$$
W(x)=(\log\vert\lambda_{1}(x)\vert,\dots,\log\vert\lambda_{r}(x)\vert).$$
Then
\begin{equation}\label{7568}
	\prod_{i=1}^{r_{1}}\lambda_{i}(x)\cdot \prod_{i=r_{1}+1}^{r+1}\vert\lambda_{i}(x)\vert^{2}=\det(A_{\B}(x))=N_{K}(x),
\end{equation}	
and so the value of $\log\vert\lambda_{r+1}(x)\vert$ is uniquely determined by $N_{K}(x)$ and $W(x)$.
The following result is essentially proved in Theorem 8.1.6 of \cite{problem}:
\begin{prop}\label{lattice}
	Let $U_{K}$ denote the group of units in $\O_{K}$.
	Then 
	there exist $\epsilon_{1},\dots,\epsilon_{r}\in U_{K}$ such that the $\mathbb{R}$-span of $W(\epsilon_{1}),\dots,W(\epsilon_{r})$ is $\mathbb{R}^{r}$.
\end{prop}


\begin{proof}[Proof of Theorem \ref{mink2}]
	Let $\epsilon_{1},\dots,\epsilon_{r}\in U_{K}$ be such that the $\mathbb{R}$-span of  $W(\epsilon_{1}),\dots,W(\epsilon_{r})$ is $\mathbb{R}^{r}$. By Proposition \ref{lattice}, there exist a constant $C_{1}:=C_{1}(K)>0$ and $x_{1},\dots,x_{r}\in\mathbb{Z}$ such that denoting $\epsilon=\epsilon^{x_{1}}_{1}\dots \epsilon^{x_{r}}_{r}$, we have that
	$$\log\vert\lambda_{i}(\epsilon a)\vert=\log\vert\lambda_{i}(a)\vert+\sum_{j=1}^{r}x_{j}\log\vert\lambda_{i}(\epsilon_{j})\vert\in [\frac{1}{D}\log\vert N_{K}(a)\vert,\frac{1}{D}\log\vert N_{K}(a)\vert+C_{1}]$$
	for all $1\leq i\leq r$. Note that $\epsilon$ is a unit of $\O_{K}$ and so $\vert N_{K}(\epsilon)\vert=1$.
	By (\ref{7568}),
	$$\log\vert\lambda_{r+1}(\epsilon a)\vert=\frac{1}{\alpha_{r+1}}(\log\vert N_{K}(\epsilon a)\vert-\sum_{i=1}^{r}\alpha_{i}\log\vert\lambda_{i}(\epsilon a)\vert)=\frac{1}{\alpha_{r+1}}(\log\vert N_{K}( a)\vert-\sum_{i=1}^{r}\alpha_{i}\log\vert\lambda_{i}(\epsilon a)\vert),$$
	where $\alpha_{i}=1$ if $1\leq i\leq r_{1}$ and $\alpha_{i}=2$ if $r_{1}+1\leq i\leq r+1$. So $$\log\vert\lambda_{r+1}(\epsilon a)\vert\in [\frac{1}{D}\log\vert N_{K}(a)\vert-rC_{1},\frac{1}{D}\log\vert N_{K}(a)\vert].$$
	
	Now let $m_{1}b_{1}+\dots+m_{D}b_{D}\in (\epsilon a)^{-1}R_{N,\B}$ for some $m_{1},\dots,m_{D}\in\Z$. By definition, there exists $n_{1}b_{1}+\dots+n_{D}b_{D}\in R_{N,\B}$ for some $n_{1},\dots,n_{D}\in\Z$ such that 
	$$n_{1}b_{1}+\dots+n_{D}b_{D}=(\epsilon a)(m_{1}b_{1}+\dots+m_{D}b_{D})=(m_{1},\dots,m_{D})A_{\B}(\epsilon a)\begin{bmatrix} 
	b_{1}\\
	b_{2}\\
	\dots\\
	b_{D}
	\end{bmatrix}.$$
	In other words, $(m_{1},\dots,m_{D})=(n_{1},\dots,n_{D})A^{-1}_{\B}(\epsilon a)$. Suppose that $$(n_{1},\dots,n_{D})=\sum_{i=1}^{r_{1}}c_{i}v_{i}+\sum_{i=r_{1}+1}^{r+1}(c_{i}v_{i}+c'_{i}\overline{v}_{i})$$ for some $c_{i},c'_{i}\in\mathbb{C}$, where $v_{i}$ is defined in Lemma \ref{di}. Then 
	$$(m_{1},\dots,m_{D})=\sum_{i=1}^{r_{1}}c_{i}\lambda^{-1}_{i}(\epsilon a)v_{i}+\sum_{i=r_{1}+1}^{r+1}(c_{i}\lambda^{-1}_{i}(\epsilon a)v_{i}+c'_{i}(\overline{\lambda}_{i})^{-1}(\epsilon a)\overline{v}_{i}).$$
	Since   $-N\leq n_{1},\dots,n_{D}\leq N$ and the basis $v_{1},\dots,v_{r+1},\overline{v}_{r_{1}+1},\dots,\overline{v}_{r+1}$ depends only on $\B$, there exists  $C_{2}:=C_{2}(\B)>0$ such that $\vert c_{i}\vert, \vert c'_{i}\vert\leq C_{2}N$ for all $1\leq i\leq r+1$. Then all of $\vert c_{i}\lambda^{-1}_{i}(\epsilon a)\vert$, $\vert c'_{i}(\overline{\lambda}_{i})^{-1}(\epsilon a)\vert$ are at most $C_{2}Ne^{2rC_{1}}\vert N_{K}(a)\vert^{-\frac{1}{D}}.$ Again there exists  $C_{3}:=C_{3}(\B)>0$  such that $\vert m_{i}\vert\leq C_{2}C_{3}e^{2rC_{1}}\vert N_{K}(a)\vert^{-\frac{1}{D}}N$ for all $1\leq i\leq D$. Setting $C_{\B}:=C_{2}C_{3}e^{2rC_{1}}$, we have that $(m_{1},\dots,m_{D})\in R_{C_{\B}\vert N_{K}(a)\vert^{-\frac{1}{D}}N,\B}$, and so $(\epsilon a)^{-1}R_{N,\B}\subseteq R_{C_{\B}\vert N_{K}(a)\vert^{-\frac{1}{D}}N,\B}$.
\end{proof}	

\begin{rem}\label{lattice2}
	The dimension $r$ of the $\mathbb{R}$-span of $W(U_{K})$ equals to 0 if and only if $r_{1}=1, r_{2}=0$ or $r_{1}=0, r_{2}=1$, which implies that $K=\mathbb{Q}$ or $\mathbb{Q}(\sqrt{-d})$ for some square-free positive integer $d$.
	In this case,  there exists $C_{\B}>0$ such that every $a\in\O_{K}$ is $C_{\B}$-regular. 
\end{rem}

The following is another property of $C$-regular numbers:

\begin{lem}\label{reg2}
	Let $\K$ be an integral tuple and $C,N\geq 0$. Then there exists $C':=C'(\B,C)>0$ such that for every prime element $a\in\O_{K}$ which is $C$-regular, we have that  
	$$\Bigl\vert\Bigl\{z\in a^{-1}R_{N,\B}\colon N_{K}(a) \text{ is not coprime with } N_{K}(z) \text{ in $\Z$}\Bigr\}\Bigr\vert\leq C'\cdot\frac{N^{D}}{\vert N_{K}(a)\vert^{1+\frac{1}{D}}}+o_{\B,a}(N^{D}).$$
\end{lem}	
\begin{proof}
	Let $\mathcal{J}_{p}$ be defined as in Lemma \ref{pi}.
	Since $a$ is a prime element, $J:=(a)$ is a prime ideal. 
	By Lemma \ref{pi}, exists a prime integer $p\in\N$ such that $N(J)=p^{i}$ for some $1\leq i\leq D$ and $J\in \mathcal{J}_{p}$. Again by Lemma \ref{pi}, for all $J'\in \mathcal{J}_{p}$,
	$$N(J')\geq p\geq \vert N_{K}(a)\vert^{\frac{1}{D}}.$$

	Let $z\in \O_{K}$ be such that   $\vert N_{K}(a)\vert=N(J)=p^{i}$ is not coprime with  $N_{K}(z)$  in $\Z$. 
	Then $N_{K}(z)$ is divisible by $p$. By the unique factorization of $(z)$, there exists
	a prime ideal $J'$ of $\O_{K}$ such that $z\in J'$ and $N(J')$ divides $p$. By Lemma \ref{pi}, $J'\in\mathcal{J}_{p}$. Then

	\begin{equation}\nonumber
		\begin{split}
			&\quad	\Bigl\vert\Bigl\{z\in a^{-1}R_{N,\B}\colon N_{K}(a) \text{ is not coprime with } N_{K}(z) \text{ in $\Z$}\Bigr\}\Bigr\vert
			\\&\subseteq	\Bigl\vert\Bigl\{z\in R_{C\vert N_{K}(a)\vert^{-\frac{1}{D}}N,\B}\colon N_{K}(a) \text{ is not coprime with } N_{K}(z) \text{ in $\Z$}\Bigr\}\Bigr\vert  \text{ (since $a$ is $C$-regular)}
			\\&\leq\sum_{J'\in\mathcal{J}_{p}}\Bigl\vert R_{C\vert N_{K}(a)\vert^{-\frac{1}{D}}N,\B}\cap J'\Bigr\vert  \text{ (by the discussion above)}
			\\&\leq\sum_{J'\in\mathcal{J}_{p}}\frac{(2C\vert N_{K}(a)\vert^{-\frac{1}{D}} N+1)^{D}}{N(J')}+o_{\B,a}(N^{D})  \text{ (by Lemma \ref{normi})}
			\\&\leq D\cdot\frac{(2C\vert N_{K}(a)\vert^{-\frac{1}{D}} N+1)^{D}}{\vert N_{K}(a)\vert^{\frac{1}{D}}}+o_{\B,a}(N^{D})  \text{ (by Lemma \ref{pi} (ii))}
			\\&\leq D2^{D-1}((2C)^{D}N^{D}\vert N_{K}(a)\vert^{-(1+\frac{1}{D})}+1)+o_{\B,a}(N^{D}).
		\end{split}	
	\end{equation}	
	This finishes the proof.
\end{proof}

\subsection{K\'atai's Lemma on algebraic number fields}
K\'atai's Lemma is an important tool in the study of correlations between a multiplicative function and an arbitrary sequence in the integer ring $\Z$. It was first proved by K\'atai \cite{katai} for $K=\mathbb{Q}$ and generalized to $K=\mathbb{Q}(\sqrt{-d})$ for all positive square-free integer $d$ by Frantzikinakis and Host \cite{FH}. In this section, we introduce a version of  K\'atai's Lemma for general number fields.

The proof of the following theorem can be found on pages 148--149 of \cite{problem}.  

\begin{thm}\label{dirich}
	Let $K$ be a number field and $\mathcal{O}_{K}$ be its ring of integers. Then 
	$$\sum_{p\in \O_{K} \text{ is a prime element}}\frac{1}{\vert N_{K}(p)\vert}=\infty \text{ and } \sum_{p\in \O_{K} \text{ is a prime element}}\frac{1}{\vert N_{K}(p)\vert^{1+c}}<\infty \text{ for all $c>0$.}$$ 
\end{thm}

Let $\mathcal{P}$ be a finite subset of $\O_{K}$ and $z\in\O_{K}$. Denote 
\begin{equation}\label{a}
	\mathcal{A}_{\mathcal{P}}=\sum_{p\in\mathcal{P}}\frac{1}{\vert N_{K}(p)\vert} \text{ and } \omega_{\mathcal{P}}(z)=\sum_{p\in\mathcal{P}\colon p\vert z}1.
\end{equation}

\begin{lem}[Tur\'an-Kubilius' Lemma]\label{tk}
	Let $\K$ be an integral tuple.
	For every $N\in\N$ and every finite subset of prime elements $\mathcal{P}$ of $\O_{K}$ whose $K$-norms are pairwise coprime (in $\Z$), we have that 
	$$\sum_{z\in R_{N,\B}}\vert\omega_{\mathcal{P}}(z)-\mathcal{A}_{\mathcal{P}}\vert\ll_{\B} \sqrt{\mathcal{A}_{\mathcal{P}}+1}\cdot N^{D}+o_{\B,\P}(N^{D}).$$
\end{lem}
\begin{proof}
	By the Cauchy-Schwartz inequality, it suffices to show that 
	\begin{equation}\label{tk1}
		\sum_{z\in R_{N,\B}}(\omega_{\mathcal{P}}(z)-\mathcal{A}_{\mathcal{P}})^{2}\ll_{\B} (\mathcal{A}_{\mathcal{P}}+1)\cdot N^{D}+ o_{\B,\P}(N^{D}).
	\end{equation}	
	Note that 
	\begin{equation}\label{tk2}
		\sum_{z\in R_{N,\B}}\mathcal{A}_{\mathcal{P}}^{2}=\mathcal{A}_{\mathcal{P}}^{2}\cdot (2N+1)^{D}.
	\end{equation}	
	By Lemma \ref{normi},
	\begin{equation}\label{tk3}
		\begin{split}
			&\quad	\sum_{z\in R_{N,\B}}2\mathcal{A}_{\mathcal{P}}\cdot\omega_{\mathcal{P}}(z)
			=2\mathcal{A}_{\mathcal{P}}\sum_{z\in R_{N,\B}}\sum_{p\in\mathcal{P}\colon p\vert z}1=2\mathcal{A}_{\mathcal{P}}\sum_{p\in\mathcal{P}}\vert  p^{-1}R_{N,\B}\vert
			\\&=2\mathcal{A}_{\mathcal{P}}\sum_{p\in\mathcal{P}}(\frac{(2N+1)^{D}}{\vert N_{K}(p)\vert}+o_{\B,\P}(N^{D}))=2\mathcal{A}_{\mathcal{P}}^{2}\cdot (2N+1)^{D}+ o_{\B,\P}(N^{D}).
		\end{split}	
	\end{equation}
	
	We claim that if $z\in\O_{K}$, $p\vert z$ and $q\vert z$ for some $p,q\in\mathcal{P}, p\neq q$, then $pq\vert z$.  It suffices to show that there is no prime ideal $I$ such that $I\supseteq(p)$ and $I\supseteq(q)$. If not, by the unique factorization of $(p)$ and $(q)$, both $N_{K}(p)$ and $N_{K}(q)$ are divisible by $N(I)$. Since  $N_{K}(p)$ is coprime with $N_{K}(q)$ in $\Z$ by assumption, we have that $\vert N(I)\vert=1$, a contradiction. This proves the claim.

	By Lemma \ref{normi} and the claim,
	\begin{equation}\label{tk4}
		\begin{split}
			&\quad	\sum_{z\in R_{N,\B}}\omega^{2}_{\mathcal{P}}(z)
			=\sum_{z\in R_{N,\B}}(\sum_{p,q\in\mathcal{P}, p\neq q\colon  p\vert z, q\vert z}1+\sum_{p\in\mathcal{P}\colon  p\vert z}1)
			=\sum_{p,q\in\mathcal{P}, p\neq q}\vert  (pq)^{-1}R_{N,\B}\vert+\sum_{p\in\mathcal{P}}\vert p^{-1} R_{N,\B}\vert
			\\&=\sum_{p,q\in\mathcal{P}, p\neq q}(\frac{(2N+1)^{D}}{\vert N_{K}(pq)\vert}+o_{\B,\P}(N^{D}))+\sum_{p\in\mathcal{P}}(\frac{(2N+1)^{D}}{\vert N_{K}(p)\vert}+o_{\B,\P}(N^{D}))
			\\&\leq \mathcal{A}^{2}_{\mathcal{P}}\cdot (2N+1)^{D}+(\mathcal{A}_{\mathcal{P}}+C)\cdot (2N+1)^{D}+o_{\B,\P}(N^{D}),
		\end{split}	
	\end{equation}	
	where $C:=\sum_{p\in \O_{K} \text{ is a prime element}}\frac{1}{\vert N_{K}(p)\vert^{2}}<\infty$ by Theorem \ref{dirich}.
	Then (\ref{tk2}), (\ref{tk3}) and (\ref{tk4}) implies (\ref{tk1}).
\end{proof}	

We are now ready to state K\'atai's Lemma on arbitrary algebraic number fields. An important difference between K\'atai's Lemma for $K=\mathbb{Q}$ or $\mathbb{Q}(\sqrt{-d}), d\in\N$ and that for arbitrary number field is that we require some regularity condition in the latter case (whereas the regularity condition always holds in the former case as is mentioned in Remark \ref{lattice2}).  

\begin{lem}[K\'atai's Lemma on algebraic number fields (multiplicative version)
	]\label{katai0}
	Let $\K$ be an integral tuple and $C>0$.
	Let $\chi\in\mathcal{M}_{K}$, and $h\colon\O_{K}\to\mathbb{C}$ be a function  with modulus at most 1.
	Let  $\mathcal{P}$ be a finite collection of $C$-regular prime elements of $\O_{K}$ whose $K$-norms are pairwise coprime in $\Z$.
	For $N\in\N_{+}$, let
	$$S(N):=\sum_{z\in R_{N,\B}}\chi(z)h(z)$$
	and
	$$C_{\mathcal{P}}(N):=\sum_{p,q\in \mathcal{P}, p\neq q}\Bigl\vert\sum_{z\in p^{-1}R_{N,\B}\cap  q^{-1}R_{N,\B}}h(pz)\overline{h}(qz)\Bigr\vert.$$
	Then	
	$$\Bigl\vert\frac{S(N)}{N^{D}}\Bigr\vert^{2}\ll_{C,\B} \frac{1}{\mathcal{A}^{2}_{\mathcal{P}}}\cdot\frac{C_{\mathcal{P}}(N)}{N^{D}}+(\frac{1}{\mathcal{A}_{\mathcal{P}}}+\frac{1}{\mathcal{A}^{2}_{\mathcal{P}}})+o_{C,\B,\P}(1)$$
\end{lem}	
\begin{proof}
	Let $$S'(N):=\sum_{z\in R_{N,\B}}\chi(z)h(z)\omega_{\mathcal{P}}(z).$$ By Lemma \ref{tk},
	$\vert S'(N)-\mathcal{A}_{\mathcal{P}}S(N)\vert\ll_{\B} \sqrt{\mathcal{A}_{\mathcal{P}}+1}\cdot N^{D}+o_{\B,\P}(N^{D})$. We may rewrite $S'(N)$ as
	$$S'(N)=\sum_{z\in R_{N,\B}}\sum_{p\in\mathcal{P}, p\vert z}\chi(z)h(z)=\sum_{p\in\mathcal{P}}\sum_{z\in p^{-1}R_{N,\B}}\chi(pz)h(pz).$$
	In this sum, the term $\chi(pz)h(pz)$ is equal to $\chi(p)\chi(z)h(pz)$ unless $N_{K}(p)$ is not coprime with $N_{K}(z)$ in $\Z$. By the $C$-regularity of $p$ and Lemma \ref{reg2}, if we set
	$$S''(N):=\sum_{p\in\mathcal{P}}\sum_{z\in p^{-1}R_{N,\B}}\chi(p)\chi(z)h(pz),$$
	then there exists $C_{1}:=C_{1}(C,\B)>C$ such that
	\begin{equation}\nonumber
		\begin{split}
			&\quad\vert S'(N)-S''(N)\vert\leq 2\sum_{p\in\mathcal{P}}\Bigl\vert\Bigl\{z\in p^{-1}R_{N,\B}\colon N_{K}(p) \text{ is not coprime with } N_{K}(z) \text{ in $\Z$}\Bigr\}\Bigr\vert
			\\&\leq 2C_{1}\sum_{p\in\mathcal{P}}\frac{N^{D}}{\vert N_{K}(p)\vert^{1+\frac{1}{D}}}+o_{\B,\P}(N^{D})=2C_{1}C_{2}N^{D}+o_{\B,\P}(N^{D}),
		\end{split}	
	\end{equation}	
	where $$C_{2}=\sum_{a\in \O_{K} \text{ is a prime element}}\frac{1}{\vert N_{K}(a)\vert^{1+\frac{1}{D}}}<\infty$$ is a constant depending only on $K$. 
	Let $R'_{N,\B}:=\bigcup_{p\in\mathcal{P}}p^{-1}R_{N,\B}$. By the $C$-regularity of $p\in\mathcal{P}$, $R'_{N,\B}\subseteq R_{CN,\B}$.
	By Cauchy-Swartz inequality, 
	\begin{equation}\nonumber
		\begin{split}
			&\quad	\vert S''(N)\vert^{2}=\Bigl\vert\sum_{z\in R'_{N,\B}}\chi(z)\sum_{p\in\mathcal{P}\colon z\in p^{-1}R_{N,\B}}\chi(p)h(pz)\Bigr\vert^{2}
			\\&\leq (2CN+1)^{D}\sum_{z\in R'_{N,\B}}\Bigl\vert\sum_{p\in\mathcal{P}\colon z\in p^{-1}R_{N,\B}}\chi(p)h(pz)\Bigr\vert^{2} 				
			\\&=(2CN+1)^{D}\sum_{p,q\in\mathcal{P}}\sum_{z\in p^{-1}R_{N,\B}\cap  q^{-1}R_{N,\B}}\chi(p)h(pz)\overline{\chi}(q)\overline{h}(qz)
			\\&\leq (2CN+1)^{D}\sum_{p,q\in\mathcal{P}}\Bigl\vert\sum_{z\in p^{-1}R_{N,\B}\cap  q^{-1}R_{N,\B}}h(pz)\overline{h}(qz)\Bigr\vert.
		\end{split}	
	\end{equation}
	Again by the $C$-regularity of $p\in\mathcal{P}$,
		\begin{equation}\nonumber
			\begin{split}
				&\quad	(2CN+1)^{D}\sum_{p\in\mathcal{P}}\Bigl\vert\sum_{z\in  p^{-1}R_{N,\B}}h(pz)\overline{h}(pz)\Bigr\vert\leq (2CN+1)^{D}\sum_{p\in\mathcal{P}}\Bigl\vert R_{C\vert N_{K}(p)\vert^{-\frac{1}{D}} N,\B}\Bigr\vert
				\\&=(2CN+1)^{D}\sum_{p\in\mathcal{P}}(2C\vert N_{K}(p)\vert^{-\frac{1}{D}} N+1)^{D}
				\ll_{C,\B}  \mathcal{A}_{\mathcal{P}}\cdot N^{2D}+\vert\mathcal{P}\vert\cdot N^{D}.
			\end{split}	
		\end{equation}
	Combining all the previous estimates, we have that
	\begin{equation}\nonumber
		\begin{split}
			&\quad	\vert \mathcal{A}_{\mathcal{P}}S(N)\vert^{2}\leq \vert S'(N)-\mathcal{A}_{\mathcal{P}}S(N)\vert^{2}+\vert S''(N)-S'(N)\vert^{2}+\vert S''(N)\vert^{2}
			\\&\ll_{C,\B} \mathcal{A}_{\mathcal{P}}\cdot N^{2D}+ o_{\B,\P}(N^{2D})+ N^{2D}+\vert\mathcal{P}\vert^{2}+N^{D}\cdot (C_{\mathcal{P}}(N)+\vert\mathcal{P}\vert).
		\end{split}	
	\end{equation}
	This finishes the proof by dividing both sides by $(\mathcal{A}_{\mathcal{P}}\cdot N^{D})^{2}$.
\end{proof}	

By using (\ref{iota}), we have the following additive version of Lemma \ref{katai}:

\begin{lem}[K\'atai's Lemma on algebraic number fields (additive version)]\label{katai}
	Let $\bold{K}=(K,\mathcal{O}_{K},D,$ $\B=\{b_{1},\dots,b_{D}\})$ be an integral tuple and $C>0$.
	Let $\chi\in\mathcal{M}_{K}$, and $h\colon\Z^{D}\to\mathbb{C}$ be a function with modulus at most 1.
	Let  $\mathcal{P}$ be a finite collection of  $C$-regular prime elements of $\O_{K}$ whose $K$-norms are pairwise coprime in $\Z$.
	For $N\in\N$, let
	$$S(N):=\sum_{\n\in R_{N,D}}\chi(\iota_{\B}(\n))h(\n)$$
	and
	$$C_{\mathcal{P}}(N):=\sum_{p,q\in\mathcal{P}, p\neq q}\Bigl\vert\sum_{\n\in\Z^{D}, \n A_{\B}(p),\n A_{\B}(q)\in  R_{N,D}}h(\n A_{\B}(p))\overline{h}(\n A_{\B}(q))\Bigr\vert.$$
	Then	
	$$\Bigl\vert\frac{S(N)}{(2N+1)^{D}}\Bigr\vert^{2}\ll_{C,\B} \frac{1}{\mathcal{A}^{2}_{\mathcal{P}}}\cdot\frac{C_{\mathcal{P}}(N)}{(2N+1)^{D}}+(\frac{1}{\mathcal{A}_{\mathcal{P}}}+\frac{1}{\mathcal{A}^{2}_{\mathcal{P}}})+o_{C,\B,\P}(1).$$
\end{lem}

\section{Nilmanifolds}\label{s:3}

We provide the background material and the notations we use for nilmanifolds in this section. Some of the notions we use follow from \cite{FH,GT,S}.

\subsection{Nilmanifolds and nil-structures}

Let $G$ be a connected and simply connected Lie group with the identity element $e_{G}$.\footnote{In this paper, we only concern connected and simply connected Lie groups as we will eventually reduce  all the results to this special case.} For $a,b\in G$, denote $[a,b]:=aba^{-1}b^{-1}$. For subgroups $H_{1}$ and $H_{2}$ of $G$, let $[H_{1},H_{2}]$ denote the smallest subgroup of $G$ generated by $[a,b], a\in H_{1}, b\in H_{2}$.


\begin{defn}[Nilpotent groups and Filtrations]
		Let $G$ be a connected and simply connected Lie group with the identity element $e_{G}$. The \emph{natural filtration} (or the \emph{lower central series}) $G_{c,\bullet}:=(G_{i})_{i\in\N}$ is the sequence of subgroups of $G$ defined by $G_{0}:=G_{1}:=G$, $G_{i+1}:=[G,G_{i}]$ for all $i\in\N_{+}$. We say that $G$ is  \emph{nilpotent} if there exists $d\in\N_{+}$ such that $G_{d+1}=\{e_{X}\}$. The smallest such $d\in\N_{+}$ is called the \emph{natural step} of $G$.
		
		A \emph{pre-filtration} $G_{\bullet}:=(G^{(i)})_{0\leq i\leq k+1}$ of  a nilpotent Lie group  $G$ is a sequence of subgroups $G^{(i)}$ of $G$ and some $k\in\N$ such that 
		$$G=G^{(0)}=G^{(1)}\supseteq G^{(2)}\supseteq\dots\supseteq G^{(k+1)}=\{e_{G}\}$$ 
		and $[G^{(i)},G^{(j)}]\subseteq G^{(i+j)}$ for all $i,j\in \N$, where we denote $G^{(i)}=\{e_{G}\}$ for all $i\geq d+1$ for convenience.
		We say that $G_{\bullet}$ is a \emph{filtration} if in addition $G^{(i)}\in \{G_{1},\dots,G_{d},G_{d+1}=\{e_{G}\}\}$ for all $i\in\N$.
		The smallest $k\in\N$ such that $G^{(k+1)}=\{e_{G}\}$ is called the \emph{degree} of $G_{\bullet}$. 
		It is easy to see that $k\geq d$. 
%
%
\end{defn}	


\begin{rem}
	Note that what we define as a ``pre-filtration" is called a ``filtration" in literature.
	In this paper, we only work with filtrations instead of the more general pre-filtrations, since the Mal'cev basis adapted to a filtration (see Definition \ref{Mal}) is compatible with the natural filtration. 
\end{rem}

\begin{defn}[Nilmanifold]
Let $G$ be a connected and simply connected  nilpotent Lie group and $\Gamma$ be a discrete, cocompact subgroup of $G$. Denote $X=G/\Gamma$, and let $\mathcal{B}$ and $m_{X}$ be the Borel $\sigma$-algebra and Haar measure of $X$, respectively. The probability space $(X,\mathcal{B},m_{X})$ is called a \emph{nilmanifold}. When there is no confusion, we also say that $(X,m_{X})$ or simply $X$ is a nilmanifold.	
\end{defn}

	\begin{conv}
		For convenience, in this paper, when we say that ``$X=G/\Gamma$ is a  nilmanifold", we implicitly assume that $G$ is a nilpotent connected and simply connected Lie group, and $\Gamma$ is a  discrete and cocompact subgroup of $G$.
		
		If $X=G/\Gamma$ is a nilmanifold, then we use $m_{X}$ to denote the Haar measure on $X$, and $e_{X}:=e_{G}\Gamma=\Gamma$ the identity element in $X$.
	\end{conv}

Let $X=G/\Gamma$ be a nilmanifold and $G'$ be a subgroup of $G$. 
We say that $G'$ is \emph{rational} for $\Gamma$ if $G'$ is connected, simply connected, closed, and $\Gamma':=G'\cap \Gamma$ is cocompact in $G'$.
We say that a filtration $G_{\bullet}:=(G^{(i)})_{0\leq i\leq k+1}$ of $G$ is \emph{rational} for $\Gamma$ if $G^{(i)}$ is rational for $\Gamma$ for all $\leq i\leq k+1$. It was shown in \cite{10} that the natural filtration of $G$ is rational for $\Gamma$.


 We say that $X'$ is a \emph{sub nilmanifold} of $X=G/\Gamma$ if $X'=G'/\Gamma':=G'/(G'\cap \Gamma)$ for some $G'<G$ rational for $\Gamma$. 
 
Every nilmanifold has an explicit algebraic description by using the Mal'cev basis:

\begin{defn} [Mal'cev basis]\label{Mal}
 Let $X=G/\Gamma$ be a  nilmanifold and  $G_{\bullet}:=(G^{(i)})_{0\leq i\leq k+1}$ be a filtration of $G$ for some $k\in\N$. Let $\dim(G)=m$ and $\dim(G^{(i)})=m_{i}$ for all $0\leq i\leq k+1$. A basis $\mathcal{X}:=\{\xi_{1},\dots,\xi_{m}\}$ for the Lie algebra $\g$ of $G$ (over $\mathbb{R}$) is a \emph{Mal'cev basis} for $X$ adapted to the filtration $G_{\bullet}$ if
 	\begin{itemize}
 		\item for all $0\leq j\leq m-1$, $\h_{j}:=\text{Span}_{\mathbb{R}}\{\xi_{j+1},\dots,\xi_{m}\}$ is a Lie algebra ideal of $\g$ and so $H_{j}:=\exp(\h_{j})$\footnote{$\exp\colon\g\to G$ is the exponential map.} is a normal Lie subgroup of $G;$
 		\item $G^{(i)}=H_{m-m_{i}}$ for all $0\leq i\leq k$;
 		\item the map $\psi^{-1}\colon \mathbb{R}^{m}\to G$ given by
 		\begin{equation}\nonumber
 			\psi^{-1}(t_{1},\dots,t_{m})=\exp(t_{1}\xi_{1})\dots\exp(t_{m}\xi_{m})
 		\end{equation}	
 	  is a bijection;
 	  \item $\Gamma=\psi^{-1}(\Z^{m})$.
 	\end{itemize}	
 	We call $\psi$ the  \emph{Mal'cev coordinate map} with respect to the Mal'cev basis $\mathcal{X}$.
   If $g=\psi^{-1}(t_{1},\dots,t_{m})$, we say that $(t_{1},\dots,t_{m})$ are the \emph{Mal'cev coordinates} of $g$ with respect to $\mathcal{X}$. 
   \end{defn}
   
   It is known that for every filtration $G_{\bullet}$ which is rational for $\Gamma$, there exists a Mal'cev basis adapted to it. See for example the discussion on pages 11--12 of \cite{GT}.
   
   Let $\g$ be endowed with an Euclidean structure such that the Mal'cev basis $\mathcal{X}$ is an orthogonal basis. This induces a Riemann structure on $G$ which is invariant under the right translations. We use $d_{G}$ to denote the distance on the group $G$ endowed with the corresponding geodesic distance (which is again invariant under the right translations).  
   
   Let $X=G/\Gamma$ be a nilmanifold and $p\colon G\to X$ be the projection. Let $d_{X}$ denote the metric on $X$ given by
   $$d_{X}(x,y):=\inf_{g,h\in G}\{d_{G}(g,h)\colon p(g)=x, p(h)=y\}.$$
   By the right invariance of $d_{G}$, it is not hard to show that $d_{X}$ is indeed a metric on $X$. Note that the infimum in the definition of $d_{X}$ can always be obtained since $\Gamma$ is discrete. We say that $d_{X}$ and $d_{G}$ are metrics \emph{induced} by $G_{\bullet}$ (or $\mathcal{X}$).
   
   In order to simplify the notations of all the structures imposed above on a nilmanifold, we introduce the following notation:
   
   \begin{defn}[Nil-structure]
   Let $X=G/\Gamma$ be a nilmanifold. If $G_{\bullet}$ is a filtration of $X$ rational for $\Gamma$, $\mathcal{X}$  is a Mal'cev basis adapted to $G_{\bullet}$, $\psi\colon\mathbb{R}^{m}\to G$ is the Mal'cev coordinate map with respect to $\mathcal{X}$, and $d_{G},d_{X}$ are the metrics induced by $G_{\bullet}$, we then say that the tuple
   $\G$ is a \emph{nil-structure} of $X$. 
   We say that $X$ is a \emph{$k$-step} nilmanifold with respect to $\X$ if the degree of $G_{\bullet}$ is $k$. 
   
   We say that $\mathfrak{X}=(G_{\bullet},\mathcal{X},\psi,d_{G},d_{X})$ is a \emph{natural nil-structure} of $X$ if $G_{\bullet}=G_{c,\bullet}$ is the natural filtration of $G$. 
   
   \end{defn}
   
   We define some special nil-structures which are used in later sections:
   
   \begin{defn}[Variations of nil-structures]\label{vn}
      Let $X=G/\Gamma$ be a nilmanifold with a nil-structure $\X=(G_{\bullet}=(G^{(i)})_{0\leq i\leq k+1},\mathcal{X},\psi,d_{G},d_{X})$ and suppose that $\dim (G)=m$. 
      
      \emph{Quotient nilmanifold.} Let $G'$ be a normal subgroup of $G$ rational for $\Gamma$. Let $\pi\colon G\to G'$ be the quotient map. Denote $G_{\pi}:=G/G'$ and $\Gamma_{\pi}:=\Gamma/(G'\cap\Gamma)$. Then $X_{\pi}:=G_{\pi}/\Gamma_{\pi}$ is a nilmanifold. Then we use $G_{\pi,\bullet}:=({G_{\pi}}^{(i)})_{i\in\N}$\footnote{When we do not wish to specify the number of subgroups contained in a filtration $G_{\bullet}$, we simply write $G_{\bullet}:=(G^{(i)})_{i\in\N}$, meaning that there exists $k\in\N$ such that $G_{\bullet}:=(G^{(i)})_{0\leq i\leq k+1}$ and $G^{(s)}=\{e_{G}\}$ for all $s>k$.} to denote the filtration of $G_{\pi}$ given by ${G_{\pi}}^{(i)}:=G^{(i)}/G', i\in\N$. We say that any nil-structure of $X_{\pi}$ of the form $\X_{\pi}=(G_{\pi,\bullet},\mathcal{X}_{\pi},\psi_{\pi},d_{G_{\pi}},d_{X_{\pi}})$ (i.e. the filtration of $\X_{\pi}$ is $G_{\pi,\bullet}$) is a nil-structure \emph{induced} by the quotient $\pi$ from $\X$.\footnote{We remark that in general there is no natural method to induced $\mathcal{X}',\psi',d_{G'},d_{X'}$ from $\X$.}
     
      \emph{Sub nilmanifold.} Let $G'$ be a subgroup of $G$ rational for $\Gamma$, and $X':=G'/(G'\cap\Gamma)$ be a sub nilmanifold of $X$. Then we use $G'_{\bullet}:=({G'}^{(i)})_{i\in\N}$ to denote the filtration of $G'$ given by ${G'}^{(i)}:=G^{(i)}\cap G', i\in\N$. We say that any nil-structure of $X'$ of the form $\X'=(G'_{\bullet},\mathcal{X}',\psi',d_{G'},d_{X'})$ (i.e. the filtration of $\X'$ is $G'_{\bullet}$) is a nil-structure \emph{induced} by $\X$ (or by $G_{\bullet}$).
      	
      	 \emph{Conjugated sub nilmanifold.} 
      	Let $X':=G'/(G'\cap\Gamma)$ be a sub nilmanifold of $X$ with a nil-structure $\X'=(G'_{\bullet},\mathcal{X}',\psi',d_{G'},d_{X'})$ induced by $\X$ and suppose that $\dim(G')=m'$.
      	Let $a\in G$ be \emph{rational} for $\Gamma$, meaning that $a^{m}\in\Gamma$ for some $m\in\Z\backslash\{0\}$. Denote $G'_{a}:=a^{-1}G'a$, $X'_{a}:=G'_{a}/(G'_{a}\cap \Gamma)$,\footnote{Lemma B.4 of \cite{FH} implies that $G_{a}$ is a subgroup of $G$ rational for $\Gamma$, and so $X_{a}$ is a sub nilmanifold of $X$.} and let $G'_{a,\bullet}:=({G'}_{a}^{(i)})_{i\in\N}$ be the filtration given by ${G'}_{a}^{(i)}:=a^{-1} {G'}^{(i)}a, i\in\N$. We say that any nil-structure of $X'_{a}$ of the form $\X'_{a}=(G'_{a,\bullet},\mathcal{X}_{a}',\psi_{a}',d_{G_{a}'},d_{X_{a}'})$ (i.e. the filtration of $\X'_{a}$ is $(G'_{a})_{\bullet}$) is a nil-structure \emph{induced} by $a$-conjugate from $\X'$.
      	
       \emph{Product nilmanifold.} Let $X\times X:=G\times G/(\Gamma\times\Gamma)$ be the \emph{product nilmanifold} of $X$. Then we use $(G\times G)_{\bullet}:=((G\times G)^{(i)})_{0\leq i\leq d+1}$ to denote the filtration of $G\times G$ given by $(G\times G)^{(i)}:=G^{(i)}\times G^{(i)}, i\in\N$, $\psi\times\psi\colon G\times G\to\mathbb{R}^{2s}$ the 
      	Mal'cev coordinate map such that for $\psi(g)=(x_{1},\dots,x_{m})$ and $\psi(g')=(x'_{1},\dots,x'_{m})$, $\psi\times\psi (g,g'):=(x_{1},\dots,x_{m};x'_{1},\dots,x'_{m})$,\footnote{Strictly speaking, we should define $\psi\times\psi (g,g')$ as $(x_{1},x'_{1},\dots,x_{m},x'_{m})$ instead of $(x_{1},\dots,x_{m};x'_{1},\dots,x'_{m})$ in order to comply with the definition of the Mal'cev basis. Nevertheless, with a slight abuse of the word  ``Mal'cev basis", we still use the latter one as the definition  since it is much more convenient.} $d_{G\times G}:=d_{G}\times d_{G}$ and $d_{X\times X}:=d_{X}\times d_{X}$. We use $\X\times \X$ to denote the nil-structure $((G\times G)_{\bullet},\psi\times\psi,d_{G\times G},d_{X\times X})$. 	
   \end{defn}	
   

Let $X$ be a nilmanifold with a nil-structure $\G$.
    For $s\in\N$ and $f\colon X\to\mathbb{C}$, let $\Vert f\Vert_{C^{s}(\X)}$ denote the usual $C^{s}$-norm and  $\Vert f\Vert_{\Lip(\X)}$ the Lipschitz norm 
   of $f$ (with respect to the metric $d_{X}$).
     Let $C^{s}(\X)$ and $\Lip(\X)$ denote the collection of all $f\colon X\to\mathbb{C}$ for which the corresponding norm is finite. It is easy to see that $\Vert f\Vert_{\Lip(\X)}\leq \Vert f\Vert_{C^{1}(\X)}$ for all $f\in C^{1}(\X)$.
    

    We summarize some facts regarding the metrics on nilmanifolds (see Section 4.2 of \cite{FH} for the proof):
    \begin{lem}\label{dist}
    	Let $X=G/\Gamma$ be a $k$-step nilmanifold with a nil-structure $\G$ for some $k\in\N_{+}$. Let $X'=G'/(G'\cap\Gamma)$ be a sub nilmanifold of $X$ and $\mathfrak{X}'=(G'_{\bullet},\mathcal{X}',\psi',d_{G'},d_{X'})$  be a nil-structure of $X'$ induced by $\X$. Then
    	\begin{enumerate}[(i)]
    		\item For every  bounded subset $F$ of $G$, there exists $C>0$ such that for all $g,h,h'\in F$, $d_{G}(gh,gh')\leq C d_{G}(h,h')$;
    		\item  For every  bounded subset $F$ of $G$, there exists $C>0$ such that for all $x,y\in X$ and $g\in F$, $d_{X}(g\cdot x,g\cdot y)\leq C d_{X}(x,y)$;
    		\item  For every  bounded subset $F$ of $G$, there exists $C_{s}>0$ for every $s\in\N$ such that for every $f\in C^{s}(\X)$ and $g\in F$, writing $f_{g}(x):=f(g\cdot x)$, we have that $\Vert f_{g}\Vert_{C^{s}(\X)}\leq C_{s}\Vert f\Vert_{C^{s}(\X)}$;
    		\item There exists $\delta>0$ such that for all $1\leq i\leq k$, $\gamma\in\Gamma$ and $g\in G^{(i)}$, $d_{G}(\gamma,g)<\delta$ implies that $\gamma\in G^{(i)}$;\footnote{Property (iv) is stated for the natural filtration in \cite{FH}, but its proof applies easily to any filtration rational for $\Gamma$ (i.e. the filtration $G'_{\bullet}$).}
    		\item   There exists $C\geq 1$ such that for all $x,y\in X'$, $C^{-1}d_{X}(x,y)\leq d_{X'}(x,y)\leq Cd_{X}(x,y)$.
    	\end{enumerate}	
    \end{lem}	
    
    \subsection{Properties on the Lie bracket}\label{lb}
    \begin{defn}[Iterated Lie bracket]
    	Let $G$ be a Lie group. For $d\in\N_{+}$ and $g_{1},\dots,g_{d}\in G$, denote
    	$$[g_{1},\dots,g_{d}]_{d}:=[[\dots[[g_{1},g_{2}],g_{3}]\dots],g_{d}].$$
    \end{defn}	
    When $d=1$, we denote $[g_{1}]_{1}:=g_{1}$. When $d=2$, we have that $[g_{1},g_{2}]_{2}=[g_{1},g_{2}]$.
    
    We provide a lemma regarding to the map $[\cdot,\dots,\cdot]_{d}$ for later uses.
    
    \begin{lem}\label{lie}
    	Let $d\in\N_{+}$ and $G$ be a nilpotent Lie group of natural step $d$ with the natural filtration $(G_{i})_{0\leq i\leq d+1}$.\footnote{This lemma also holds for any filtration $G_{\bullet}$ of $G$. But we do not need it.}
    	\begin{enumerate}[(i)]
    		\item Let $n\in\N$ and $a_{1},\dots,a_{n}\in\N_{+}$. For all $g_{i}\in G_{a_{i}}, 1\leq i\leq n$, $[g_{1},\dots,g_{i}]_{i}\in G_{a_{1}+\dots+a_{n}}$.
    		\item For all $g'_{1},g_{1},\dots,g_{d}\in G$, 
    		$$[g_{1},\dots,g_{d}]_{d}\cdot [g'_{1},\dots,g_{d}]_{d}=[g_{1}\cdot g'_{1},\dots,g_{d}]_{d}.$$
    		\item For all $g_{1},\dots,g_{d}\in G$, 
    		$$[g_{1},\dots,g_{d}]_{d}^{-1}=[g^{-1}_{1},\dots,g_{d}]_{d};$$
    	\end{enumerate}	
    \end{lem}	
    \begin{proof}
    	(i) is straightforward by induction. 
    	(iii) is a corollary of (ii) by setting $g'_{1}=g^{-1}_{1}$. 
    	
    	We now prove (ii).  By (i), $[g_{2},\dots,g_{d}]_{d-1}\in G_{d-1}$. So it suffices to show that for all $a,b,c\in G$, we have that
    	$$[ab,c]=[a,[b,c]]\cdot [b,c]\cdot [a,c],$$
    	which can be verified by a direct computation.
    	%
    \end{proof}

 \subsection{Special factors of a nilmanifold}
 We introduce three special factors of a nilmanifold in this section. The first one
 is the lower horizontal torus, which plays an important role in equidistribution properties:
 
 \begin{defn}[horizontal torus and characters]
 	Let  $X=G/\Gamma$ be a nilmanifold with a natural nil-structure $\Gc$ with $d$ being the natural step of $X$. Suppose that $\dim(G)=m$ and $\dim(G_{2})=m_{2}$. 
 	Then $\psi_{c}\colon G\to\mathbb{R}^{m}$ induces an isometric identification between the \emph{horizontal torus} $G/(G_{2}\Gamma)$ (endowed with the quotient metric) and $\T^{m-m_{2}}$ (endowed with the canonical metric). 
 	A \emph{horizontal character} is a continuous group homomorphism $\eta\colon G\to\T$ such that $\eta(\Gamma)=\{0\}$. Then every horizontal character $\eta$ vanishes on $G_{2}$ and induces a  continuous group homomorphism between $G/G_{2}$  and $\mathbb{R}^{m-m_{2}}$.

 	Let $\G$ be any nil-structure of $X$. Under the Mal'cev basis $\mathcal{X}$, we can write $$\eta\circ\psi^{-1} (x_{1},\dots,x_{m})=\ell_{1}x_{1}+\dots+\ell_{m}x_{m} \mod \mathbb{Z}$$
 	for some $\ell_{1},\dots,\ell_{m}\in\Z$ (called the \emph{coordinates} of $\eta$ with respect to $\X$)
 	for all $(x_{1},\dots,x_{m})\in\mathbb{R}^{m}$ in a unique way. Denote the \emph{$\X$-norm} of $\eta$ by
 	$$\Vert\eta\Vert_{\X}:=\vert\ell_{1}\vert+\dots+\vert\ell_{m-m_{2}}\vert.$$
 \end{defn}	
 
  The second special factor is a sub torus of the horizontal torus $G/G_{2}\Gamma$ which we call the upper horizontal torus. This concept is uncommon in literature, but is essential in understanding certain sub nilmanifolds of the product space $X\times X$. 
 
  \begin{defn}[Upper horizontal torus]
  	Let  $X=G/\Gamma$ be a nilmanifold of natural step $d\in\N_{+}$ with the natural filtration $G_{c,\bullet}=(G_{i})_{0\leq i\leq d+1}$. 
  	Let $G_{\ker}$ be the collection of all $g\in G$ such that for all $g_{2},\dots,g_{d}\in G$, $[g,g_{2},\dots,g_{d}]_{d}=e_{G}$. By Lemma \ref{lie}, it is easy to see that
  	$G_{\ker}$ is a normal subgroup of $G$ and contains $G_{2}$. We say that $G/G_{\ker}\Gamma$ is the \emph{upper horizontal torus} of $X$ (endowed with the quotient metric). 
  \end{defn}
  	 
  	We remark that if $(Z_{0})_{0\leq i\leq d+1}$ is the \emph{upper central series} of $G$, meaning that $Z_{0}=\{e_{G}\}$ and $Z_{i+1}=\{x\in G\colon [x,y]\in Z_{i} \text{ for all } y\in G\}$, then $G_{\ker}$ is equal to $Z_{d}$.
 The upper horizontal torus $G/G_{\ker}\Gamma$ is a sub torus of the horizontal torus $G/G_{2}\Gamma$, but the converse may not be true. 
  
  \begin{ex}\label{exHei}
  	Let $H=\mathbb{R}^{3}$ be endowed with a group structure given by
  	$$(x,y;z)\cdot (x',y';z'):=(x+x',y+y';z+z'+xy')$$
  	for all $(x,y;z), (x',y';z')\in\mathbb{R}^{3}$. It is easy to see that $(H,\cdot)$ is a group, and $H_{2}=\{0\}\times\{0\}\times \R$, $H_{3}=\{(0,0;0)\}$. This group is called the \emph{Heisenberg group}. 
  	
  	Let $G=\mathbb{R}\times H$, $\Gamma=\Z^{4}$ and $X=G/\Gamma$. Then $G_{2}=\{0\}\times\{0\}\times\{0\}\times \R$ and the horizontal torus $G/G_{2}\Gamma$ is $\T^{3}\times\{0\}$. On the other hand, $G_{\ker}=\R\times\{0\}\times\{0\}\times \R$, and so the upper horizontal torus $G/G_{\ker}\Gamma$ is $\{0\}\times\T^{2}\times\{0\}$.      	
  \end{ex}	
  
  We postpone further properties of the upper horizontal torus to Section \ref{s:d}.
   Given a filtration $G_{\bullet}$ of a nilmanifold $X=G/\Gamma$ of natural step $d$, it is convenient for us to work on a Mal'cev adapted to $G_{\bullet}$ where the subgroups $G_{\ker}$ and $G_{d}$ of $G$ can be expressed in a nice way. 
   
   \begin{defn}[Standard Mal'cev basis]
   	Let $X=G/\Gamma$ be a  nilmanifold of natural step $d$ for some $d\in\N_{+}$ with $G_{c,\bullet}=(G_{i})_{0\leq i\leq d+1}$ being its natural filtration. Let  $G_{\bullet}=(G^{(i)})_{0\leq i\leq k+1}$ be another filtration of $G$ for some $k\in\N_{+}$. Suppose that $\dim(G)=m$, $\dim(G_{\ker})=m'_{2}$ and $\dim(G_{d})=r$. Let $\mathcal{X}:=\{\xi_{1},\dots,\xi_{m}\}$ be a Mal'cev basis for $X$ adapted to the filtration $G_{\bullet}$ with $\psi\colon G\to\mathbb{R}^{m}$ being the Mal'cev coordinate map. We say that $\mathcal{X}$ is \emph{standard} if $G_{\ker}=\psi^{-1}(\{0\}^{m-m'_{2}}\times\mathbb{R}^{m'_{2}})$ and $G_{d}=\psi^{-1}(\{0\}^{m-r}\times\mathbb{R}^{r})$.
   	
   	We say that a nil-structure $\G$ is \emph{standard} if $\mathcal{X}$ is standard.
   \end{defn}

   It is easy to see that every filtration $G_{\bullet}$ admits one (but not necessarily unique) standard Mal'cev basis, as $G^{(i)}\in \{G_{1},\dots,G_{d},G_{d+1}=\{e_{G}\}\}$ for all $i\in\N$.
   
   The last special factor is the vertical torus, a concept which allows us to conduct Fourier analysis on nilmanifolds.

     \begin{defn}[Vertical torus and nilcharacters (or vertical characters)]
     	Let $X=G/\Gamma$ be a  nilmanifold of natural step $d$ for some $d\in\N_{+}$ with $G_{c,\bullet}=(G_{i})_{0\leq i\leq d+1}$ being its natural filtration. 
     	Suppose that $\dim(G_{d})=r$. Then $G_{d}$ lies in the center of $G$. We call $G_{d}/(G_{d}\cap\Gamma)$ the \emph{vertical torus} on $X$.
     	For a standard nil-structure $\G$ of $X$,
     	we say that $\Phi\colon X\to\mathbb{C}$ is a \emph{nilcharacter} (or \emph{vertical character}) with \emph{frequency} $(h_{1},\dots,h_{r})\in\mathbb{Z}^{r}$ with respect to $\X$ if $$\Phi(g\cdot x)=e(h_{1}y_{1}+\dots+h_{r}y_{r})\Phi(x)$$ for all $g=\psi^{-1}(0,\dots,0;y_{1},\dots,y_{r})\in G_{d}$\footnote{$g\in G_{d}$ because $\X$ is standard.} and $x\in X$.
     \end{defn}
     

     The following are some basic properties of nilcharacters, which will be used in later sections: 
     
     \begin{lem}[Translation invariance of nilcharacters]\label{good2}
     	Let $X=G/\Gamma$ be a  nilmanifold of natural step $d$ for some $d\in\N_{+}$ with $G_{c,\bullet}=(G_{i})_{0\leq i\leq d+1}$ being its natural filtration. 
     	 Let $\G$ be a standard nil-structure of $X$ and $\Phi$ be a nilcharacter of $X$ with respect to $\X$. For $g_{0}\in G$, let $\Phi_{g_{0}}(x):=\Phi(g_{0}\cdot x)$ for all $x\in X$. Then $\Phi_{g_{0}}$ is also a nilcharacter of $X$ with the same frequency as $\Phi$ with respect to $\X$.
     	\end{lem}
     \begin{proof}
     	Suppose that $\dim(G_{d})=r$ and   $\Phi$ is with frequency $(h_{1},\dots,h_{r})\in\mathbb{Z}^{r}$ with respect to $\X$. Since $\X$ is standard,  $$\Phi(g\cdot x)=e(h_{1}y_{1}+\dots+h_{r}y_{r})\Phi(x)$$ for all $g=\psi^{-1}(0,\dots,0;y_{1},\dots,y_{r})\in G_{d}$ and $x\in X$. Since $g\in G_{d}$ is in the center of $G$, 
     	$$\Phi_{g_{0}}(g\cdot x)=\Phi(g_{0}g\cdot x)=\Phi(gg_{0}\cdot x)=e(h_{1}y_{1}+\dots+h_{r}y_{r})\Phi(g_{0}x)=e(h_{1}y_{1}+\dots+h_{r}y_{r})\Phi_{g_{0}}(x).$$
     	This implies that $\Phi_{g_{0}}$ is also a nilcharacter of $X$ with frequency $(h_{1},\dots,h_{r})$ with respect to $\X$.
     \end{proof}		
   
  
     \begin{lem}[Nilcharacters on $X\times X$]\label{good}
     		Let $X=G/\Gamma$ be a  nilmanifold of natural step $d$ for some $d\in\N_{+}$ with $G_{c,\bullet}=(G_{i})_{0\leq i\leq d+1}$ being its natural filtration. 
     		Suppose that $\dim(G_{d})=1$. Let $\G$ be a standard nil-structure of $X$ and $\Phi$ be a nilcharacter of $X$ with frequency $\ell$ with respect to $\X$ for some $\ell\in\Z\backslash\{0\}$. Let $H$ be a subgroup of $G\times G$ rational for $\Gamma\times \Gamma$ and $Y:=H/(H\cap(\Gamma\times\Gamma))$ be a sub nilmanifold of $X\times X$ with a standard nil-structure $\mathfrak{Y}$ induced by $\X\times\X$. Then
     	\begin{enumerate}[(i)]
     		\item  $\Phi\otimes\overline{\Phi}$ is a  nilcharacter of $X\times X$ with frequency $(\ell,-\ell)$ with respect to $\X\times\X$.
     		\item If $\dim(H_{d})=2$ (i.e. $H_{d}=G_{d}\times G_{d}$), then  $\Phi\otimes\overline{\Phi}\Big\vert_{Y}$ is a  nilcharacter of $Y$ with frequency $(\ell,-\ell)$ with respect to $\mathfrak{Y}$.
     			\item If  $\dim(H_{d})=1$, and suppose that $$H_{d}=\{(\psi^{-1}(0,\dots, 0;\ell_{1}t),\psi^{-1}(0,\dots, 0;\ell_{2}t))\in G_{d}\times G_{d}\colon t\in\mathbb{R}\}$$ for some $\ell_{1},\ell_{2}\in\mathbb{Z}$ not all equal to 0, then $\Phi\otimes\overline{\Phi}\Big\vert_{Y}$ is a  nilcharacter of $Y$ with respect to $\mathfrak{Y}$. Moreover, its frequency is non-zero if and only if $\ell_{1}\neq \ell_{2}$.
     	\end{enumerate}	
     \end{lem}
     \begin{proof}
     	By assumption, $\Phi(g\cdot x)=e(\ell y)\Phi(x)$ for all $g=\psi^{-1}(0,\dots,0;y)\in G_{d}$ and $x\in X$. 
     	
     	(i) For all $g=\psi^{-1}(0,\dots,0;y),g'=\psi^{-1}(0,\dots,0;y')\in G_{d}$ and $(x,x')\in X\times X$, 
     	\begin{equation}\label{46}
     		\begin{split}
     		\Phi\otimes\overline{\Phi}((g,g')\cdot (x,x'))=\Phi(g x)\overline{\Phi}(g' x')=e(\ell y-\ell y')\Phi(x)\overline{\Phi}(x')
     		=e((\ell,-\ell)\cdot (y,y'))\Phi\otimes\overline{\Phi}(x,x').
     		\end{split}
     	\end{equation}	
     	So $\Phi\otimes\overline{\Phi}$ is a  nilcharacter of $X\times X$ with frequency $(\ell,-\ell)$ with respect to $\X\times\X$.
     	
     	(ii) If $H_{d}=G_{d}\times G_{d}$, then (\ref{46}) holds for all $(g,g')\in H_{d}$. So $\Phi\otimes\overline{\Phi}\Big\vert_{Y}$ is a  nilcharacter of $Y$ with frequency $(\ell,-\ell)$ with respect to $\mathfrak{Y}$.
     	
     	(iii) Let $h=(\psi^{-1}(0,\dots, 0;\ell_{1}t),(\psi^{-1}(0,\dots, 0;\ell_{2}t))\in H_{d}$ for some
     	$t\in\mathbb{R}$. Then for all $(x,x')\in X\times X$, by (\ref{46}),
     	\begin{equation}\nonumber
     		\begin{split}
     			&\quad\Phi\otimes\overline{\Phi}(h\cdot (x,x'))=
     			e(\ell(\ell_{1}-\ell_{2})t)\Phi\otimes\overline{\Phi}(x,x').
     		\end{split}
     	\end{equation}	
     	So $\Phi\otimes\overline{\Phi}\vert_{Y}$ is a  nilcharacter of $Y$ with respect to $\mathfrak{Y}$, and its frequency is zero if and only if $\ell_{1}-\ell_{2}=0$ (since $\ell\neq 0$).
     \end{proof}

    \section{Equidistribution properties for polynomial sequences on nilmanifolds}\label{s:4}
    In this section, we collect all the equidistribution results we need in this paper.
    \subsection{Polynomial sequences and smooth norms}
     We start with the definition of polynomial
     sequences.
     \begin{defn}[Polynomial sequences]
     	Let $G$ be a group endowed with a pre-filtration $G_{\bullet}=(G^{(i)})_{0\leq i\leq k+1}$ for some $k\in\N$. Let $D\in\N_{+}$ and $g\colon\Z^{D}\to G$ be a map. For $\bold{h}\in\Z^{D}$, define $\partial_{\bold{h}}g\colon\Z^{D}\to G$ by $\partial_{\bold{h}}g(\n):=g(\n+\bold{h})g^{-1}(\n)$ for all $\n\in\Z^{D}$. Let $\poly_{D}(G_{\bullet})$ denote the collection of all $g\colon\Z^{D}\to G$ such that for all $i\in\N$, and $\n,\bold{h}_{1},\dots,\bold{h}_{i}\in\Z^{D}$, we have that $\partial_{\bold{h}_{i}}\dots\partial_{\bold{h}_{1}}g(\n)\in G^{(i)}$. We call functions in $\poly_{D}(G_{\bullet})$ \emph{polynomial sequences} with respect to  $G_{\bullet}$. 
     	
        We say that $g\colon \Z^{D}\to G$ is a \emph{polynomial sequence} on $G$  (written as $g\in\poly_{D}(G)$ without specifying the pre-filtration) if $g\in\poly_{D}(G_{\bullet})$ for some pre-filtration $G_{\bullet}$ of $G$. The \emph{degree} of $g$ is the smallest degree of all the filtrations $G_{\bullet}$ of $G$ such that $g\in\poly_{D}(G_{\bullet})$.
     \end{defn}

     \begin{rem}
     Clearly, if $g\in\poly_{D}(G_{\bullet})$ for some pre-filtration $G_{\bullet}$, then $g\in\poly_{D}(G'_{\bullet})$ for some filtration $G'_{\bullet}$. So the definition of polynomial sequences	in this paper coincides with the one used in \cite{FH} and \cite{GT}.
     
       Note that there is an implicitly upper bound for the ``degree" of every polynomial sequence in $\poly_{D}(G_{\bullet})$, namely the degree of the pre-filtration $G_{\bullet}$. 
     \end{rem}	
     
     \begin{rem}\label{2m}
     	As we shall see later in this paper, in many theorems, we endow two filtrations (and two nil-structures adapted to them) on a nilmanifold simultaneously: a natural filtration $G_{c,\bullet}$ through which the horizontal, upper horizontal and vertical toruses are defined, and a filtration $G_{\bullet}$ through which the polynomial sequence is defined.
     \end{rem}

     For $D\in\N_{+}$, $\n=(n_{1},\dots,n_{D})\in\Z^{D}$, and $\j=(j_{1},\dots,j_{D})\in\N^{D}$, recall that
     $\vert \j\vert:=j_{1}+\dots+j_{D}$. Denote $\n^{\j}:=n^{j_{1}}_{1}\dots n^{j_{D}}_{D}$ and
     $$\binom{\n}{\j}:=\prod_{i=1}^{D}\binom{n_{i}}{j_{i}}.$$
     The following description of polynomial sequences is Lemma 6.7 of \cite{GT} (or Section 4 of \cite{G24}):
     
     \begin{lem}[Polynomials in Mal'cev basis]\label{6.7}
     	Let $X=G/\Gamma$ be a nilmanifold with a nil-structure $\mathfrak{X}=(G_{\bullet}=(G^{(i)})_{0\leq i\leq k+1},\mathcal{X},\psi,d_{G},d_{X})$. Suppose that $\dim(G)=m$ and $\dim(G^{(i)})=m_{i}$ for all $0\leq i\leq k+1$.  Then $g\in\poly_{D}(G_{\bullet})$ if and only if 
     	$$\psi\circ g(\bold{n})=\sum_{\j\in\N^{D}}\alpha_{\j}\binom{\n}{\j}$$
     	for some $\alpha_{\j}=(\alpha_{\j,1},\dots,\alpha_{\j,m})\in\mathbb{R}^{m}$ for all $\j\in\N^{D}$ such that $\alpha_{\j,i}=0$ for all $i\leq m-m_{\vert \j\vert}$. 
     \end{lem}	 
     
     Let $D,k,r\in\N_{+}$ and $g\in\poly_{D}(\mathbb{R}_{\bullet}^{r})$, where the filtration $\mathbb{R}_{\bullet}^{r}:=((\R^{r})^{(i)})_{0\leq i\leq k+1}$ of $\R$ is given by $(\R^{r})^{(i)}:=\R^{r}$ for all $0\leq i\leq k$ and $(\R^{r})^{(k+1)}:=\{0\}$. By Lemma \ref{6.7}, $g$ can be expressed alternatively in two different ways (in each way there is a unique expression):
     $$g(\n)=\sum_{\j\in\N^{D}}\alpha_{\j}\binom{\n}{\j} \text{ or } g(\n)=\sum_{\j\in\N^{D}}\alpha'_{\j}\n^{\j}$$
     for some $\alpha_{\j},\alpha'_{\j}\in\R^{r}$ for all $\j\in\N^{D}$ such that $\alpha_{\j}=\alpha'_{\j}=\bold{0}$ whenever $\vert\j\vert>k$. 
     
     \begin{defn}[Smooth norms]\label{sn}
     	Let the notations be as above.
     	 For all $N\in\N$, we define the \emph{smooth norms} of $g\in\poly_{D}(\mathbb{R}_{\bullet}^{r})$ as 
     	 $$\Vert g\Vert_{C_{r}^{\infty}(R_{N,D})}:=\max_{\j\neq \bold{0}}(2N+1)^{\vert\j\vert}\cdot\Vert\alpha_{\j}\Vert_{\T^{r}} \text{ and } \Vert g\Vert'_{C_{r}^{\infty}(R_{N,D})}:=\max_{\j\neq \bold{0}}(2N+1)^{\vert\j\vert}\cdot\Vert\alpha'_{\j}\Vert_{\T^{r}}.$$
     \end{defn}	
    
     It is easy to check that there exists $C:=C(k,D)>1$ such that
     $$C^{-1}\Vert g\Vert_{C_{r}^{\infty}(R_{N,D})}\leq \Vert g\Vert'_{C_{r}^{\infty}(R_{N,D})}\leq C\Vert g\Vert_{C_{r}^{\infty}(R_{N,D})}$$
     for all $r\in\N_{+}$ and $g\in\poly_{D}(\mathbb{R}_{\bullet}^{r})$. So we can use both norms alternatively without affecting our proofs. Roughly speaking, it was shown in \cite{FH,GT} that the smallness of the smooth norms of $g$ indicates that $g$ is a slow-varying function.
     
     Obviously, the smallness of the $\Vert\cdot\Vert_{\T^{r}}$-norms of the coefficients $\alpha_{\j}$ (or $\alpha'_{\j}$) implies the smallness of the smooth norm of $g$. Conversely, we have the following lemma:
    
     \begin{lem}\label{popo}
     	Let $D,m\in\N_{+}$ and $g\colon \Z^{D}\to\R$ be a homogeneous polynomial of the form
     	$$g(\n)=\sum_{\vert\j\vert=m}a'_{\j}\n^{\j}$$
     	for some $a'_{\j}\in\R$ for all $\n\in \Z^{D}$. There exist $C:=C(D,m)>0$ and $Q:=Q(D,m)\in\N_{+}$ such that if $\Vert g(\n)\Vert_{\T}\leq C_{0}$ for all $\n\in\Z^{D}, \vert\n\vert\leq m$, then $\Vert Qa'_{\j}\Vert_{\T}\leq C_{0}C$ for all $\vert\j\vert=m$.
     \end{lem}
     \begin{proof}
     	Recall that $\partial_{\m}g(\n):=g(\n+\m)-g(\n)$. Let $\j=(j_{1},\dots,j_{D})\in\N^{D}$ be any vector with $\vert\j\vert=m$. Then it is easy to check that 
     	$$\partial^{j_{1}}_{\bold{e}_{1}}\dots\partial^{j_{D}}_{\bold{e}_{D}}g(\bold{0})= (j_{1}!\cdot\ldots\cdot j_{D}!)a'_{\j}.$$
     	Since  $\Vert g(\n)\Vert_{\T}\leq C_{0}$ for all $\vert\n\vert\leq m$,
     	$$\Vert (j_{1}!\cdot\ldots\cdot j_{D}!)a'_{\j}\Vert_{\T}\leq 2^{D}C_{0}.$$
     	Let $Q=(m!)^{D}$, which divides $j_{1}!\cdot\ldots\cdot j_{D}!$. We have that 
     	$$\Vert Qa'_{\j}\Vert_{\T}\leq 2^{D}C_{0}Q/(j_{1}!\cdot\ldots\cdot j_{D}!)\leq 2^{D}C_{0}Q$$
     	for all $\vert\j\vert=m$. This finishes the proof by setting $C=2^{D}Q$.   
     \end{proof}	
     
%

\subsection{Smooth norms on the upper horizontal torus}

Let  $X=G/\Gamma$ be a nilmanifold with a standard nil-structure $\G$. Suppose that $\dim(G)=m$, $\dim(G_{\ker})=m'_{2}$ and let $s'=m-m'_{2}$.
For convenience, we use the same notation $\pi_{\ker}$ to denote the following two  different maps, the meaning of which will always be clear from the context: (i) $\pi_{\ker}\colon\mathbb{R}^{m}\to\mathbb{R}^{s'}$, the projection from $\mathbb{R}^{m}$ to its first $s'$ coordinates; (ii) $\pi_{\ker}\colon G\to G/G_{\ker}$, the quotient map of $G$ by $G_{\ker}$.

Clearly, the Mal'cev coordinate map $\psi$ induces an isometric identification $\psi_{\ker}\colon G/G_{\ker}\to\mathbb{R}^{s'}$ between $G/G_{\ker}$ and $\mathbb{R}^{s'}$ such that $\psi_{\ker}\circ\pi_{\ker}=\pi_{\ker}\circ\psi\colon G\to\mathbb{R}^{s'}$. $\psi_{\ker}$ also    induces an isometric identification between the upper horizontal torus $G/(G_{\ker}\Gamma)$ and $\T^{s'}$ (endowed with the canonical metric). We define the smooth norm on the upper horizontal torus as follows, which will be used in later sections. 

\begin{defn}[Smooth norm on the upper horizontal torus]
	Let  $X=G/\Gamma$ be a nilmanifold with a standard nil-structure $\G$. Suppose that $\dim(G)=m$, $\dim(G_{\ker})=m'_{2}$ and let $s'=m-m'_{2}$. Let $D,N\in\N_{+}$ and $g\in \poly_{D}(G_{\bullet})$. Then $\pi_{\ker}\circ\psi\circ g\colon \Z^{D}\to \mathbb{R}^{s'}$ can be written as 
	$$\pi_{\ker}\circ\psi\circ g(\n)=\sum_{\j\in\N^{D},\vert\j\vert\leq k}\alpha_{\j}\binom{\n}{\j} \text{ or } \pi_{\ker}\circ\psi\circ g(\n)=\sum_{\j\in\N^{D},\vert\j\vert\leq k}\alpha'_{\j}\n^{\j}$$
	for some $d\in\N,\alpha_{\j},\alpha'_{\j}\in\mathbb{R}^{r}$ for all $\j\in\N^{D},\vert\j\vert\leq k$. We define the \emph{smooth norm} of $g$ on the upper horizontal torus by
	$$\Vert g\Vert_{C_{\ker,\X}^{\infty}(R_{N,D})}:=\Vert \pi_{\ker}\circ\psi\circ g\Vert_{C_{s'}^{\infty}(R_{N,D})} \text{ and } \Vert g\Vert'_{C_{\ker,\X}^{\infty}(R_{N,D})}:=\Vert \pi_{\ker}\circ\psi\circ g\Vert'_{C_{s'}^{\infty}(R_{N,D})},$$
	where $\Vert\cdot\Vert_{C_{s'}^{\infty}(R_{N,D})}$ and $\Vert\cdot\Vert'_{C_{s'}^{\infty}(R_{N,D})}$ are the norms defined in Definition \ref{sn}.
	%
\end{defn}	

\subsection{Leibman's Theorem and total equidistribution}
By the quantitative nature of the results in this paper, we need to use
the concept of total $\e$-equidistribution first introduced in \cite{GT}, which can be viewed as a quantitative version of (\ref{ed}).

     \begin{defn}[Total $\e$-equidistribution]
     Let $(X=G/\Gamma,m_{X})$ be a nilmanifold with a nil-structure $\X$. Let $D,N\in\N_{+}$, $\epsilon>0$ and $g\colon \Z^{D}\to G$. We say that the sequence $(g(\n)\cdot e_{X})_{\n\in R_{N,D}}$ is \emph{totally $\epsilon$-equidistributed} on $X$ with respect to $\X$ 
     if for every $D$-dimensional arithmetic progression $P$, every function $f$ on $X$ with $\Vert f\Vert_{\Lip(\X)}\leq 1$ and $\int_{X} f\,d m_{X}=0$, we have that
     \begin{equation}\label{33}
     	\Bigl\vert\E_{\n\in R_{N,D}}\bold{1}_{P}(\n)f(g(\n)\cdot e_{X})\Bigr\vert\leq \epsilon.
     \end{equation}	
    \end{defn}
     
     
     The next result is a variation of Theorem 8.6 in ~\cite{GT}, which provides a convenient criteria for establishing equidistribution properties
     of polynomial sequences on nilmanifolds (see also Theorem 7.3 of \cite{S}):
     
     \begin{thm}[A variation of the quantitative Leibman's Theorem]\label{Lei}
     	Let $\epsilon>0, D\in\N_{+}$ and $X=G/\Gamma$ be a nilmanifold with a nil-structure $\G$. There exists $C:=C(\X,\epsilon,D)>0$\footnote{If a quantity  depends on $\X$ (such as $C$), then it also implicitly depends on the nilmanifold $X$.} such that for every $N\in\mathbb{N}$ and polynomial sequence $g\in\poly_{D}(G_{\bullet})$, if $(g(\n)\cdot e_{X})_{\n\in R_{N,D}}$ is not totally $\epsilon$-equidistributed on $X$ with respect to $\X$, then there exists a horizontal character $\eta$ such that $$0<\Vert\eta\Vert_{\X}\leq C \text{ and }\Vert \eta\circ g\Vert_{C_{1}^{\infty}(R_{N,D})}\leq C.$$
     \end{thm}

     	This theorem is stated in \cite{er} and \cite{GT} under the stronger hypothesis that the sequence is not ``$\e$-equidistributed on $X$", meaning that (\ref{33}) fails for $P=R_{N,D}$.
     	The stronger result Theorem \ref{Lei} can be obtained by using Theorem 5.2 of \cite{FH} combined with a similar argument in Lemma 3.1 in \cite{er}. 
     	We omit the proof.

     The following is a partial converse of the above result (see also Lemma 5.3 of \cite{FH} and Theorem 7.5 of \cite{S}):
     
     \begin{thm}[Inverse Leibman's Theorem]\label{inv}
     	Let $D\in\mathbb{N}_{+}, C_{0}>0$ and $X=G/\Gamma$ be a nilmanifold with a nil-structure $\G$. There exist $C:=C(\X,D), N_{0}:=N_{0}(\X,C_{0},D)>0$ such that for every $C_{0}>0$, every $N\geq N_{0}$, and every polynomial sequence $g\in\poly_{D}(G_{\bullet})$, if there exists a non-trivial horizontal character $\eta$ of $X$ with $\Vert\eta\Vert_{\X}\leq C_{0}$ and $\Vert \eta\circ g\Vert_{C_{1}^{\infty}(R_{N,D})}\leq C_{0}$, then the sequence $(g(\n)\cdot e_{X})_{\n\in R_{N,D}}$ is not totally $CC_{0}^{-(D+1)}$-equidistributed on $X$ with respect to $\X$.
     \end{thm}
     \begin{proof}
     	Since $\Vert \eta\circ g\Vert_{C_{1}^{\infty}(R_{N,D})}\leq C_{0}$, we have that
     	\begin{equation}\nonumber
     		\begin{split}
     			\eta\circ g(\n)=\sum_{\j\in\N^{D}}\alpha_{\j}\binom{\n}{\j},
     		\end{split}
     	\end{equation}
     	where $\Vert a_{\j}\Vert_{\T}\leq\frac{C_{0}}{(2N+1)^{\vert\j\vert}}$ for all $0<\vert\j\vert\leq k$ with $k$ being the degree of $G_{\bullet}$ which depends only on $\X$. Thus
     	$\vert e(\eta\circ g(\n))-e(\eta\circ g(\bold{0}))\vert\leq 1/2$ for all $\n\in R_{\frac{C_{1}N}{C_{0}},D}$ for some $C_{1}:=C_{1}(k,D)>0$. 
     	Then for all $N\in\N$,
     	\begin{equation}\nonumber
     		\begin{split}
     			\Bigl\vert\mathbb{E}_{\n\in R_{\frac{C_{1}N}{C_{0}},D}}e(\eta\circ g(\n))\Bigr\vert\geq\frac{1}{2},
     		\end{split}
     	\end{equation}
     	which implies that
     	\begin{equation}\label{tempp}
     		\begin{split}
     			\Bigl\vert\mathbb{E}_{\n\in R_{N,D}}\bold{1}_{R_{\frac{C_{1}N}{C_{0}},D}}(\n)
     			e(\eta\circ g(\n))\Bigr\vert\geq \frac{C_{1}^{D}}{2C_{0}^{D}}-\frac{C_{2}}{C_{0}^{D-1}N}
     		\end{split}
     	\end{equation}
     	for some $C_{2}:=C_{2}(\X,D)>0$. So if $N>4C_{2}C_{0}/C_{1}^{D}$, then the left hand side of (\ref{tempp}) is at least $\frac{C_{1}^{D}}{4C_{0}^{D}}$.
     	
     	Since $\Vert\eta\Vert_{\X}\leq C_{0}$, the function $x\rightarrow e(\eta(x))$ defined on $X$ is Lipschitz with respect to $\X$ with Lipschitz constant at most $C_{0}C_{3}$ for some $C_{3}:=C_{3}(\X,D)>0$, and has integral 0 since $\eta$ is non-trivial. Therefore, the sequence $(g(\n)\cdot e_{X})_{\n\in R_{N,D}}$ is not totally $C_{4}C_{0}^{-(D+1)}$-equidistributed with $C_{4}:=C_{1}^{D}/4C_{3}$ for all $N>4C_{2}C_{0}/C_{1}^{D}$.
     \end{proof}

We also need the following alternative description of  total equidistribution:     
    
\begin{prop}[Total equidistribution on general subsets]\label{comp}
	Let $X=G/\Gamma$ be a nilmanifold with a nil-structure $\G$. Let 
	$D\in\mathbb{N}_{+}$ and $\epsilon>0$. There exist $\delta:=\delta(\X,\epsilon)>0$ and $N_{0}:=N_{0}(\X,\epsilon)>0$ 
	such that for every $g\in\poly_{D}(G_{\bullet})$, if there exist  $N\in\mathbb{N}, N\geq N_{0}$, 
	a set $P\subseteq R_{N,D}$ such that for any line $\ell\subseteq \R^{D}$, $P\cap \ell$ is a 1-dimensional arithmetic progression (possibly an empty set), 
	and a function $\Phi\colon X\to\mathbb{C}$ with $\int_{X}\Phi\,d\mu=0, \Vert \Phi\Vert_{\Lip(\X)},\vert \Phi\vert\leq 1$ such that
	\begin{equation}\label{91}
		\Bigl\vert\E_{\n\in R_{N,D}}\bold{1}_{P}(\n)\Phi(g(\n)\cdot e_{X})\Bigr\vert>\epsilon,
	\end{equation}	
	then $(g(\n)\cdot e_{X})_{\n\in R_{N,D}}$ is not totally $\delta$-equidistributed on $X$ with respect to $\X$.
\end{prop}

To prove this proposition, we need the following technical lemma, whose proof is the argument on pages 6--9 of \cite{GT}.
\begin{lem}\label{GTer}
	Let 
	$D,N\in\mathbb{N}_{+}$ and $\epsilon>0$. 
	Let $X=G/\Gamma$ be a nilmanifold with a nil-structure $\G$ and
	$g\in\poly_{D}(G_{\bullet})$ be a polynomial sequence.
	 Let $N,L>0$ be such that $N>L^{2}$ and $L>C/\e$ for some $C$ sufficiently large depending only on $\X$, $D$ and $\e$. Suppose that for all $\v\in [L]^{D}$, there exist $J_{\v}\subseteq R_{N,D}$ with $\vert J_{\v}\vert>\frac{1}{4}\e N^{D}$ such that $(g(\m+n\v)\cdot e_{X})_{n\in[-N/L^{2},N/L^{2}]}$ is not totally $\e/2$-equidistributed on $X$ with respect to $\X$ for all $\m\in J_{\v}$. Then there exists  $W:=W(\X,\e)>0$ (independent of $\v$) and a horizontal character $\eta$ such that 
		 $$0<\Vert\eta\Vert_{\X}\leq W \text{ and }\Vert \eta\circ g\Vert_{C_{1}^{\infty}(R_{N,D})}\leq W.\footnote{The conclusion of Lemma \ref{GTer} is exactly the same as that of Lemma 3.1 in \cite{er}.}$$
\end{lem}	

\begin{proof}[Proof of Proposition \ref{comp}]
	Let $C:=C(\X,\e)>0$ be sufficiently large to be chosen latter. Let $N>(C\e^{-1})^{2}$ and pick $C\e^{-1}<L<N^{2}$. 
	Since $\Vert \Phi\Vert_{\Lip(\X)}\leq 1$,  for all $\v\in [L]^{D}$, we have that 
	\begin{equation}\label{92}
		\E_{\n\in R_{N,D}}\bold{1}_{P}(\n)\Phi(g(\n)\cdot e_{X})=\E_{\m\in R_{N,D}}\E_{-N/L^{2}\leq n\leq N/L^{2}}\bold{1}_{P}(\m+n\v)\Phi(g(\m+n\v)\cdot e_{X})+O(\frac{1}{L}).
	\end{equation}
	So if $C\e^{-1}>4$, then (\ref{91}) and (\ref{92}) imply that there exists a set $J_{\v}\subseteq R_{N,D}$ with $\vert J_{\v}\vert>\e(2N+1)^{D}/4$ such that for all $\m\in J_{\v}$,
	\begin{equation}\label{93}
		\Bigl\vert\E_{-N/L^{2}\leq n\leq N/L^{2}}\bold{1}_{P}(\m+n\v)\Phi(g(\m+n\v)\cdot e_{X})\Bigr\vert>\epsilon/2.
	\end{equation}
	By assumption, the set  $\{n\in\Z\colon \m_{\v}+n\v\in P\}$  is a 1-dimensional arithmetic progression. So (\ref{93}) implies that
	 the sequence 
	$$(g_{\m,\v}(n)\cdot e_{X})_{n\in [-N/L^{2},N/L^{2}]}:=(g(\m+n\v)\cdot e_{X})_{n\in [-N/L^{2},N/L^{2}]}$$
	is not  totally $\e/2$-equidistributed on $X$ with respect to $\X$ for all $\v\in [L]^{D}$ and $\m\in J_{\v}$. We may then use Lemma \ref{GTer} to conclude that 
	there exist  $W:=W(\X,\e)>0$ and a horizontal character $\eta$ such that 
	$$0<\Vert\eta\Vert_{\X}\leq W \text{ and }\Vert \eta\circ g\Vert_{C_{1}^{\infty}(R_{N,D})}\leq W.$$
	By Theorem \ref{inv}, 
	$(g(\n)\cdot e_{X})_{\n\in R_{N,D}}$ is not totally $\delta$-equidistributed on $X$ with respect to $\X$ for some $\delta:=\delta(\X,\e)>0$ and $N$ sufficiently large depending only on $\X,\e$.
\end{proof}	

 Let $X=G/\Gamma$ be a nilmanifold.
 Recall that $a\in G$ is rational for $\Gamma$ if $a^{m}\in\Gamma$ for some $m\in\N_{+}$.
 The following is an application of Theorems \ref{Lei} and \ref{inv} which is used in later sections. 
 
 \begin{cor}[Changing the base point]\label{5.5}
 	Let $X=G/\Gamma$ be a nilmanifold with a nil-structure $\G$. Let $G'$ be a subgroup of $G$ rational for $\Gamma$ and $X'=G'/(G'\cap\Gamma)$ be a nilmanifold with a nil-structure  $\X'=(G'_{\bullet},\mathcal{X}',\psi',d_{G'},d_{X'})$ induced by $\X$.
 	Let $a\in G$ be rational for $\Gamma$ and denote $G'_{a}:=a^{-1}G'a$. Let $X'_{a}:=G'_{a}/(G'_{a}\cap\Gamma)$ be with a nil-structure $\X'_{a}=((G'_{a})_{\bullet},\mathcal{X}_{a}',\psi_{a}',d_{G_{a}'},d_{X_{a}'})$ induced by $a$-conjugate from $\X'$. Let $D\in\N_{+}$. Then there exist a function $\rho:=\rho_{\X,\X',\X'_{a},D}\colon\mathbb{R}_{+}\to\mathbb{R}_{+}$ with $\lim_{t\to 0^{+}}\rho(t)=0$ and $N_{0}:=N_{0}(\X,\X',\X'_{a},D)\in\N$ such that for all $g\in\poly_{D}(G'_{\bullet})$ and $N\geq N_{0}$, if $(g(\n)\cdot e_{X'})_{\n\in R_{N,D}}$ is totally $t$-equidistributed on $X'$ with respect to $\X'$, then  $(a^{-1}g(\n)a\cdot e_{X'_{a}})_{\n\in R_{N,D}}$ is totally $\rho(t)$-equidistributed on $X'_{a}$ with respect to $\X'_{a}$.   
 \end{cor}
 \begin{rem}
 	Let the notations be as in Corollary \ref{5.5}.
 	It was proved in Lemma B.4 of \cite{FH} that $G'_{a}$ is a subgroup of $G$ rational for $\Gamma$. So $G'_{a}\cap\Gamma$ is cocompact in $G'_{a}$  and $X'_{a}=G'_{a}\cdot e_{X}$. Moreover, writing $g_{a}(n):=a^{-1}g(n)a$ and $(G'_{a})_{\bullet}:=a^{-1}G_{\bullet}a$, we have that $g_{a}\in\poly_{D}((G'_{a})_{\bullet})$.
 \end{rem}	
 
 Corollary \ref{5.5} was proved in Corollary 5.5 of \cite{FH} for the case $D=1$, but the general case can be proved by a similar method (by using Theorems \ref{Lei} and \ref{inv}, higher dimensional versions of Leibman's theorems), and so we omit the proofs.
 
     \subsection{Factorization theorem}
     \begin{defn}[Smooth and rational sequences]
     	Let $D,N\in\N_{+}$, $M\in\N$ and $X=G/\Gamma$ be a nilmanifold with a nil-structure $\G$. Suppose that $\dim(G)=m$.
     	\begin{itemize}
     		\item A sequence $\epsilon\in \poly_{D}(G_{\bullet})$ is \emph{(M,N)-smooth} with respect to $\X$ if $\Vert\psi\circ \epsilon\Vert_{C^{\infty}_{m}(R_{N,D})}\leq M$, $d_{G}(\epsilon(\n),e_{G})\leq M$, and $d_{G}(\epsilon(\n),\epsilon(\n+\bold{e}_{i}))\leq M/N$ for all $\n\in R_{N,D}$ and $1\leq i\leq D$.
     		\item $g\in G$ is \emph{$M$-rational} for $\Gamma$ if $g^{m}\in\Gamma$ for some $m\in\N, 1\leq m\leq M$. A sequence $\gamma\colon R_{N,D}\to G$ is \emph{$M$-rational} for $\Gamma$ if $\gamma(\n)$ is $M$-rational for $\Gamma$ for all $\n\in R_{N,D}$.
     	\end{itemize}	
     \end{defn}	
    
     The following result says that every polynomial sequence is concentrated near a finite collection of sub nilmanifolds:
     \begin{thm}[Factorization theorem]\label{5.6}     
     	Let $D,M\in\N_{+}$ and $X=G/\Gamma$ be a nilmanifold with a nil-structure $\G$.
     	 There exists a finite family $\mathcal{F}(M):=\mathcal{F}_{\X,D}(M)$ of subnilmanifolds of $X$, which increases with $M$, each of the form $X'=G'/\Gamma'$ for some subgroup $G'$ of $G$ rational for $\Gamma$ and $\Gamma':=G'\cap\Gamma'$,
     	 endowed with a nil-structure $\X'$ induced by $\X$,
     	 such that the following holds: for every function $\omega\colon\N\to\mathbb{R}_{+}$, there exists $M_{1}:=M_{1}(\X,\omega,D)\in\N_{+}$, and for every $N\in\N_{+}$ and $g\in\poly_{D}(G_{\bullet})$, there exist $M\in\N$ with $M\leq M_{1}$, a sub nilmanifold $X'\in\mathcal{F}(M)$, and a factorization
     	 $g(\n)=\epsilon(\n)g'(\n)\gamma(\n), \n\in R_{N,D}$ with $\epsilon, g', \gamma\in\poly_{D}(G_{\bullet})$ such that
     	 \begin{enumerate}[(i)]
     	 	\item $\epsilon\colon R_{N,D}\to G$ is $(M,N)$-smooth with respect to $\X$;
     	 	\item $g'\in\poly_{D}(G'_{\bullet})$ and $(g'(\n)\cdot e_{X'})_{\n\in R_{N,D}}$ is totally $\omega(M)$-equidistributed on $X'$ with respect to $\X'$;
     	 	\item $\gamma\colon R_{N,D}\to G$ is $M$-rational for $\Gamma$, and $\gamma(\n)\cdot e_{X'}=\gamma(\n+M\bold{e}_{i})\cdot e_{X'}$ for all $1\leq i\leq D$, $\n,\n+M\bold{e}_{i}\in R_{N,D}$. 
     	 \end{enumerate}	
     \end{thm}	
    
     \begin{rem}
     	 The proof of Theorem \ref{5.6} is essentially the same as Theorem 10.2 of \cite{GT}. So
     	 we omit its proof, but only pointing out the 
     	 differences: 
     	 \begin{itemize}
     	 	\item The definition of $(M,N)$-smoothness in this paper is stronger than the one used by Green and Tao \cite{GT} as we require that $\Vert\psi\circ \epsilon\Vert_{C^{\infty}_{m}(R_{N,D})}\leq M$ in addition. This stronger conclusion was in fact proved implicitly by using the construction of $\epsilon$ on pages 49--50 in the proof of Proposition 9.2 of \cite{GT}, the proof of Theorem 10.2 of \cite{GT}, and the fact that
     	 	$$\Vert\psi\circ \epsilon_{1}\epsilon_{2}\Vert_{C^{\infty}_{m}(R_{N,D})}\leq C(\Vert\psi\circ \epsilon_{1}\Vert_{C^{\infty}_{m}(R_{N,D})}+\Vert\psi\circ \epsilon_{2}\Vert_{C^{\infty}_{m}(R_{N,D})})$$ 
     	 	for some $C>0$ depending only on $\X$ and $D$.
     	 	\item The part ``$\gamma(\n)\cdot e_{X'}=\gamma(\n+M\bold{e}_{i})\cdot e_{X'}$ for all $1\leq i\leq D$, $\n,\n+M\bold{e}_{i}\in R_{N,D}$" is not mentioned in Theorem 10.2 of \cite{GT}, but it follows immediately from Lemma A.12 of \cite{GT}.
     	 	\item Theorem 10.2 of \cite{GT} is stated only for $\omega(M)=M^{A}$ for some $A>0$, but the same method can be used to prove it for a general function $\omega$ (see also the remark on page 29 of \cite{FH}).
     	 	\item Theorem 10.2 of \cite{GT} provides a more explicit description of the family $\mathcal{F}(M)$, but we do not need it in this paper (see also the remark on page 29 of \cite{FH}).
     	 \end{itemize}	
     \end{rem}

       \section{Description of certain sub nilmanifold of $X\times X$}\label{s:d}
       	 The purpose of this section is to study a special type of sub nilmanifolds of $X\times X$ for some nilmanifold $X$. 
       	 Though short in length, Section \ref{s:d} is the most important piece of ingredient in the proof of Theorem \ref{key0}. 
       	 The main result in this section is Theorem \ref{auto}, but we need some definitions before stating it.

        \subsection{$d$-automorphisms on nilmanifolds}
        Let $X=G/\Gamma$ be a nilmanifold of natural step $d$. For all $g_{1},\dots,g_{d},h_{1},\dots,h_{d}\in G$ such that $g_{i}h^{-1}_{i}\in G_{\ker}$, $1\leq i\leq d$, we have that $[g_{1},\dots,g_{d}]_{d}=[h_{1},\dots,h_{d}]_{d}$. So the map $[\cdot,\ldots,\cdot]_{d}\colon G^{d}\to G_{d}$ factors through $G^{d}_{2}$, and it induces a map $$[\cdot,\ldots,\cdot]_{d}\colon (G/G_{\ker})^{d}\to G_{d}$$  in the natural way (which is still denoted as $[\cdot,\ldots,\cdot]_{d}$ for convenience).

        Let $X=G/\Gamma$ be a nilmanifold. We say that a map $\sigma\colon G\to G$ is an \emph{automorphism} of $X$ if $\sigma$ is a continuous bijection such that $\sigma(\Gamma)\subseteq \Gamma$ and $\sigma(gh)=\sigma(g)\sigma(h)$ for all $g,h\in G$. In this paper, we need to study a special type of automorphisms. 
        
        \begin{defn}[$d$-automorphisms]
        	Let $X=G/\Gamma$ be a nilmanifold and $d\in\N_{+}$ be its natural step. Let $\pi_{\ker}\colon G\to G/G_{\ker}$ be the quotient map. We say that a map $\sigma\colon G/G_{\ker}\to G/G_{\ker}$ is a \emph{$d$-automorphism} of $X$ if $\sigma$ is an automorphism of $G/G_{\ker}\Gamma$, and for all $g_{1},\dots,g_{d}\in G/G_{\ker}$, we have that
        	$$[g_{1},\dots,g_{d}]_{d}=[\sigma(g_{1}),\dots,\sigma(g_{d})]_{d}.$$    
        	 Let $\Aut_{d}(X)$ denote the collection of all  $d$-automorphisms of $X$. 	     
        \end{defn}	
        
        \begin{rem}
        	In the degenerate case $d=1$, a 1-automorphism of $X$ is just an automorphism of $G/G_{\ker}\Gamma$.
        \end{rem}

        	Let $d\in\N_{+}$ and $G$ be a nilpotent group with a standard nil-structure $\G$.  
        	Suppose that $\dim(G)=m$, $\dim(G_{2})=m_{2}$, $s'=m-m_{2}$.
         Recall that $\psi_{\ker}\colon G/G_{\ker}\to\mathbb{R}^{s'}$ is the isometric identification  between $G/G_{\ker}$ and $\mathbb{R}^{s'}$ induced by $\psi$. Then $\sigma\colon G/G_{\ker}\to G/G_{\ker}$ is an automorphism of $G/G_{\ker}\Gamma$ if and only if there exists  $A\in M_{s'\times s'}(\mathbb{Z})$ such that $$\sigma(g)=\psi_{\ker}^{-1}\circ A\circ\psi_{\ker}(g):=\psi_{\ker}^{-1}(\psi_{\ker}(g)\cdot A)$$ for all $g\in G/G_{\ker}$ (recall that $A\colon\mathbb{R}^{s'}\to\mathbb{R}^{s'}$ denotes the map given by $A(\x):=\x\cdot A$, $\x\in\mathbb{R}^{s'}$, i.e. the \textbf{right} multiplication of $A$). For convenience we denote $\sigma$ by $\sigma_{\psi,A}$, and write $$\Aut_{\psi,d}(X):=\{A\in M_{s'\times s'}(\mathbb{Z})\colon \sigma_{\psi,A}\in \Aut_{d}(X)\}.$$

        
        \begin{conv}\label{ker}
        	In order to lighten the notation, we make the following convention. Let $\x$ be a vector in $\mathbb{R}^{s'}$ (then $\psi^{-1}_{\ker}(\x)\in G/G_{\ker}$). We use $\psi^{-1}(\x)$ to denote any element in $G$ of the form  $\psi^{-1}(\x')$ for some $\x'\in\mathbb{R}^{m}$ whose first $s'$ coordinates is the vector $\x$ (i.e. $\pi_{\ker}(\x')=\x$). For every $g\in G$  and every $\sigma\in\Aut_{d}(X)$, we use $\sigma(g)$ to denote any element in $G$ whose projection $\pi_{\ker}(\sigma(g))$ on $G/G_{\ker}$ is $\sigma(\pi_{\ker}(g))$.
        	
        	Although $\psi^{-1}(\x)$ and $\sigma(g)$ are not uniquely defined as elements of $G$, they are well defined modulo $G_{\ker}$. 
        	Since
        	the map $[\cdot,\dots,\cdot]_{d}\colon G^{d}\to G_{d}$ factors through $G_{\ker}^{d}$, expressions such as $[\psi^{-1}(\x_{1}),\dots,\psi^{-1}(\x_{d})]_{d}$ and  $[\sigma(g_{1}),\dots,\sigma(g_{d})]_{d}$ are well defined even though $\psi^{-1}(\x_{i})$ and $\sigma(g_{i})$ are not (and we will use this convention in such expressions only). 
        \end{conv}

        Under this convention, we have that $\sigma\in\Aut_{d}(G)$ if for all $g_{1},\dots,g_{d}\in G$,        $$[\sigma(g_{1}),\dots,\sigma(g_{d})]_{d}=[g_{1},\dots,g_{d}]_{d}.$$

        \subsection{Main result of this section}
       
        We need some quantitative definitions before stating the main result.
        \begin{defn}[Height]
        	The \emph{height} of a rational number $\frac{p}{q}, p,q\in\mathbb{Z},(p,q)=1$ is $\max\{\vert p\vert,\vert q\vert\}$. We denote the height of an irrational number by $\infty$.
        	
        	The \emph{height} of a matrix is the maximum of the heights of entries of this matrix. The \emph{height} of a vector is the maximum of the heights of coordinates of this vector.
        	
        	For a subspace $A$  of $\mathbb{R}^{n}$ with dimension $r$, the \emph{height} of $A$ is the minimum of the heights of matrices $B\in M_{r\times s}(\Z)$ such that $A=\{\x B\in\mathbb{R}^{n}\colon\x\in\mathbb{R}^{r}\}$ (denote the height of $A$ to be $\infty$ if such a $B$ does not exist).
        	 Since natural numbers are well-ordered, the height of $A$ is always well-defined.
        	
        	Let $X=G/\Gamma$ be a nilmanifold with a nil-structure $\G$.
        	 The \emph{height} of $\sigma\in Aut_{d}(G)$ with respect to $\X$ is the smallest height of the matrix $A$ such that $\sigma$ can be written as $\sigma=\sigma_{\psi,A}$.

        	Let $X=G/\Gamma$ be a nilmanifold with a standard nil-structure $\G$.
        	Suppose that $\dim(G)=m$, $\dim(G_{\ker})=m'_{2}$, and $s'=m-m'_{2}$. Recall that $\psi$ induces an identification $\psi_{\ker}\colon G/G_{\ker}\to\mathbb{R}^{s'}$. For every subgroup $H$ of $G$  rational for $\Gamma$, $H/(H\cap G_{\ker})$ is a subgroup  of $G/G_{\ker}$  rational for $G_{\ker}\cap\Gamma$. Assume that $\dim(H/(H\cap G_{\ker}))=r$. The \emph{height} of $H$ with respect to $\X$ is the minimum of the heights of matrices $A\in M_{r\times s'}(\mathbb{Z})$ such that
        	$${\psi}_{\ker}(H/(H\cap G_{\ker}))=\{\x A\in \mathbb{R}^{s'}\colon \x\in\mathbb{R}^{r}\}.$$
        	If $Y=H/(H\cap \Gamma)$ is a sub nilmanifold of $X$, then the \emph{height} of $Y$ with respect to $\X$ is that of $H$ with respect to $\X$.
        \end{defn}
        
        We are now ready to state the main result of this section, which 
        is the heart of this paper:
        
        \begin{thm}[Description of a special sub nilmanifold of $X\times X$]\label{auto}
        	Let $d\in\N_{+}$ and $X=G/\Gamma$ be a nilmanifold with a standard nil-structure $\G$ and the natural filtration $G_{c,\bullet}=(G_{i})_{0\leq i\leq d+1}$ of natural step $d$.\footnote{Note that there are two different filtrations in the statement of this theorem. See also Remark \ref{2m}.
        		} Suppose that $\dim(G_{d})=1$.
        	Then for all $C>0$, there exists $C':=C'(X,C)>0$\footnote{The constant $C'$ and thus the whole theorem is independent of the choice of the nil-structure $\X$. But we do not need it in this paper.} such that  for every sub nilmanifold
        	$Y=H/(H\cap(\Gamma\times\Gamma))$  of $X\times X$  of height at most $C$ with respect to $\X\times\X$ (where $H$ is a subgroup of $G\times G$ rational for $\Gamma\times \Gamma$) satisfying 
        	\begin{itemize}
        		\item the projection of $Y$ to both coordinates equals to $X$;
        		\item $H_{d}$ can be written as  $$H_{d}=\{(\psi^{-1}(0,\dots, 0;t),\psi^{-1}(0,\dots, 0;t))\in G^{(d)}\times G^{(d)}\colon t\in\mathbb{R}\};$$
        	\end{itemize}	
        	there exists $\sigma\in \Aut_{d}(G)$ of height at most $C'$ with respect to $\X$ such that $h_{1}=\sigma(h_{2}) \mod G_{\ker}$\footnote{Meaning that $h_{1}\sigma^{-1}(h_{2})\in G_{\ker}$, or equivalently, $\pi_{\ker}(h_{1})=\pi_{\ker}\circ\sigma(h_{2})$.} for all $(h_{1},h_{2})\in H$.		
        \end{thm}

        \begin{proof}    	
        	Suppose that $\dim(G)=m$, $\dim(G_{\ker})=m'_{2}$ and $s'=m-m'_{2}$. Denote $\varphi:=\psi\times\psi$.
        	Recall that $\X$ induces the product nil-structure $\X\times\X$ on $X\times X$. Since $\X$ is a standard nil-structure on $X$, $\X\times\X$ is a standard nil-structure on $X\times X$.  In other words, $\varphi(G_{\ker}\times G_{\ker})=\mathbb{Z}^{s'}\times \mathbb{R}^{m'_{2}}\times\mathbb{Z}^{s'}\times \mathbb{R}^{m'_{2}}$. This  naturally induces a Mal'cev coordinate map $\varphi_{\ker}:=\psi_{\ker}\times\psi_{\ker}$ from the abelian group $(G/G_{\ker})\times (G/G_{\ker})$ to $\mathbb{R}^{2s'}$.
        	
        	Since $H/(H\cap(G_{\ker}\times G_{\ker}))$ is a subgroup of $(G/G_{\ker})\times (G/G_{\ker})$ rational for $(G_{\ker}\cap\Gamma)\times (G_{\ker}\cap\Gamma)$ of height at most $C$ with respect to $\X\times\X$,
        	denoting $r:=\dim(H/(G_{\ker}\times G_{\ker}))\leq 2s'$, there exists $A=(A_{1},A_{2})\in M_{r\times 2s'}(\mathbb{Z})$ with $\rank(A)=r$ and height at most $C$ such that 
        	$$ \varphi_{\ker}(H/(H\cap(G_{\ker}\times G_{\ker})))=\{\x A=(\x A_{1},\x A_{2})\in \mathbb{R}^{2s'}\colon \x=(x_{1},\dots,x_{r})\in\mathbb{R}^{r}\}.$$
        	

        	By the description of $H_{d}$, we have that $[g_{1,1},\dots,g_{1,d}]_{d}=[g_{2,1},\dots,g_{2,d}]_{d}$ for all $(g_{1,i},g_{2,i})\in H, 1\leq i\leq d$. So
        	\begin{equation}\label{18}
        		[\psi^{-1}_{\ker}(\x_{1}A_{1}),\dots, \psi^{-1}_{\ker}(\x_{d}A_{1})]_{d}=[\psi^{-1}_{\ker}(\x_{1}A_{2}),\dots, \psi^{-1}_{\ker}(\x_{d}A_{2})]_{d}
        	\end{equation}
        	for all $\x_{1},\dots,\x_{d}\in\mathbb{R}^{r}$.

        	Since the projection of $Y$ to both coordinates equals to $X$, we have that $r\geq s'$ and $\rank(A_{1})=\rank(A_{2})=s'$. Suppose that $A_{2}=Y\begin{bmatrix} 
        	I_{s'\times s'}\\
        	0_{(r-s')\times s'}
        	\end{bmatrix}$ for some invertible $r\times r$ matrix $Y$ of height at most $C$. Denote $A_{1}=Y\begin{bmatrix} 
        	B_{1}\\
        	B_{2}
        	\end{bmatrix}$, where $B_{1}$ and $B_{2}$ are $s'\times s'$ and $(r-s')\times s'$ matrices of heights at most $C':=C'(s',C)=s'C^{2}$, respectively. 
        	Setting $\x_{i}=(\y_{i},\z_{i})Y^{-1}, \y_{i}\in\mathbb{R}^{s'}, \z_{i}\in\mathbb{R}^{r-s'}, 1\leq i\leq d$ in (\ref{18}), we have that
        	\begin{equation}\label{18.1}
        		[\psi_{\ker}^{-1}(\y_{1}B_{1}+\z_{1}B_{2}),\dots, {\psi}_{\ker}^{-1}(\y_{d}B_{1}+\z_{d}B_{2})]_{d}=[\psi_{\ker}^{-1}(\y_{1}),\dots, \psi_{\ker}^{-1}(\y_{d})]_{d}
        	\end{equation}
        	for all $\y_{i}\in\mathbb{R}^{s'}, \z_{i}\in\mathbb{R}^{r-s'}, 1\leq i\leq d$. 
        	
        	Let $\sigma=\psi^{-1}_{\ker}\circ B_{1}\circ\psi_{\ker}$. Then $\sigma$ is of height at most $C'$ with respect to $\X\times\X$.
        	By (\ref{18.1}), we have that
        	\begin{equation}\nonumber
        		\begin{split}
        			&\quad 
        			[\sigma(g_{1}),\dots,\sigma(g_{d})]_{d}
        			=[\psi_{\ker}^{-1}\circ B_{1}\circ\psi_{\ker}(g_{1}),\dots,\psi_{\ker}^{-1}\circ B_{1}\circ\psi_{\ker}(g_{d})]_{d}
        			\\&=[\psi_{\ker}^{-1}(\psi_{\ker}(g_{1})B_{1}),\dots,\psi_{\ker}^{-1}(\psi_{\ker}(g_{k})B_{1})]_{d}
        			=[\psi_{\ker}^{-1}(\psi_{\ker}(g_{1})),\dots,\psi_{\ker}^{-1}(\psi_{\ker}(g_{d}))]_{d}=[g_{1},\dots,g_{d}]_{d}
        		\end{split}
        	\end{equation}
        	for all $g_{1},\dots,g_{d}\in G$. So $\sigma\in \Aut_{d}(G)$.
        	
        	

        	Pick any point $(h_{1},h_{2})\in H$. Since $Y$ is invertible, there exists $\x=(\y,\z), \y\in\mathbb{R}^{s'}, \z\in\mathbb{R}^{r-s'}$ such that $(h_{1},h_{2})=\varphi_{\ker}^{-1}(\x A_{1},\x A_{2})=\varphi_{\ker}^{-1}(\y B_{1}+\z B_{2},\y)$. 
        	Then $h_{1}=\psi_{\ker}^{-1}(\y B_{1}+\z B_{2})$, $h_{2}=\psi_{\ker}^{-1}(\y)$, and 
        $\sigma(h_{2})=\psi_{\ker}^{-1}(\y B_{1})$. By (\ref{18.1}), for all $\y_{i}\in\mathbb{R}^{s'}, \z_{i}\in\mathbb{R}^{r-s'}, 2\leq i\leq d$,
        	\begin{equation}\nonumber
        		\begin{split}
        			&\quad 
        			[h_{1},\psi_{\ker}^{-1}(\y_{2}B_{1}+\z_{2}B_{2}),\dots,\psi_{\ker}^{-1}(\y_{d}B_{1}+\z_{d}B_{2})]_{d}
        			\\&=[\psi_{\ker}^{-1}(\y B_{1}+\z B_{2}),\psi_{\ker}^{-1}(\y_{2}B_{1}+\z_{2}B_{2}),\dots, \psi_{\ker}^{-1}(\y_{d}B_{1}+\z_{d}B_{2})]_{\ker}
        			\\&=[\psi_{\ker}^{-1}(\y),\psi_{\ker}^{-1}(\y_{2}),\dots, \psi_{\ker}^{-1}(\y_{d})]_{d}
        			\\&=[\psi_{\ker}^{-1}(\y B_{1}),\psi_{\ker}^{-1}(\y_{2}B_{1}+\z_{2}B_{2}),\dots, \psi_{\ker}^{-1}(\y_{d}B_{1}+\z_{d}B_{2})]_{d}
        			\\&=[\sigma(h_{2}),\psi_{\ker}^{-1}(\y_{2}B_{1}+\z_{2}B_{2}),\dots, \psi_{\ker}^{-1}(\y_{d}B_{1}+\z_{d}B_{2})]_{d}.
        		\end{split}
        	\end{equation}
        	 Since $\rank(A_{1})=s'$, we have that 
        	$$[h_{1},g_{2},\dots,g_{d}]_{d}=[\sigma(h_{2}),g_{2},\dots,g_{d}]_{d}$$
        	for all $g_{2},\dots,g_{d}\in G$. 
        	By Lemma \ref{lie}, $[\sigma(h_{2})h_{1}^{-1},g_{2},\dots,g_{d}]_{d}=e_{G}$ 	for all $g_{2},\dots,g_{d}\in G$.
        	By the definition of $G_{\ker}$, we have that $h_{1}=\sigma(h_{2}) \mod G_{\ker}$.
        \end{proof}
       
       We conclude this section by illustrating how Theorem \ref{auto} can be used to study Theorem \ref{key0}, the main results of the paper. For convenience we will explain the ideas qualitatively instead of quantitatively.
       
        Let $H=\mathbb{R}^{3}$ be the Heisenberg group (recall Example \ref{exHei} for the definition).
      Note that $H_{2}=H_{\ker}=\{0\}\times\{0\}\times \R$ and $H_{3}=\{(0,0;0)\}$. Let $\Gamma=\Z^{3}$ and $X_{Hei}=H/\Gamma$.
        We consider some special case of  Theorem \ref{key0} when $X=X_{Hei}$.
        Let $\K$ be an integral tuple, $\chi\in\mathcal{M}_{K}$, $\Phi\in C(X_{Hei})$ with $\int_{X_{Hei}}\Phi\,d\mu=0$, and $g\colon\Z^{D}\to H$ be a polynomial sequence. We wish to show that 
        $$\lim_{N\to\infty}\frac{1}{(2N+1)^{D}}\sum_{-N\leq n_{1},\dots,n_{D}\leq N}\chi(n_{1}b_{1}+\dots+n_{D}b_{D})\Phi(g(n_{1},\dots,n_{D})\cdot e_{X})=0$$
        (we ignore the arithmetic progression $P$ and the supremium over $\mathcal{M}_{K}$ to simply the discussion). 
       
        \begin{ex}\label{equu3}
       	Consider the case $K=\mathbb{Q}$. Assume without loss of generality that $\Phi$ is a nilcharacter of frequency $\ell$ (see Section 7 for the reason why we can make such a simplification).
       	 By Lemma \ref{katai}, it suffices to show that (say) for many distinct primes $p,q\in\N$, the average
       	 \begin{equation}\label{eeqq}
       	 \frac{1}{N}\sum_{0\leq pn,qn\leq N-1}\vert \Phi\otimes\overline{\Phi}(g(pn)\Gamma,g(qn)\Gamma)\vert
       	 \end{equation}   
       	 is small as $N\to\infty$. By the factorization theorem (Theorem \ref{5.6}), we may assume without loss of generality that   $(g(pn)\Gamma,g(qn)\Gamma)$ is sufficiently equidistributed on a submanifold $Y$ of $X_{Hei}$ (to simplify the explanation, here we assume that the $\e$ and $\gamma$ terms in Theorem \ref{5.6} disappear). 
       	 It suffices to show that $\Phi\otimes\overline{\Phi}\vert_{Y}$ is a non-trivial nilcharacter.
       	 
       	  We may use some standard approach using the  factorization theorem to further assume that the projection of $Y$ to both coordinates equals to $X_{Hei}$ (i.e. the first condition in Theorem \ref{auto} holds).
       	 Consider the group $H_{2}=[H,H]$. Clearly, it is a subgroup of $G_{2}\times G_{2}$, which is of dimension 2. If $H_{2}$ is of dimension 0, then this contradicts to the assumption that the projection of $Y$ to both coordinates equals to $X_{Hei}$. If $H_{2}$ is of dimension 2, then by Lemma \ref{good}, $\Phi\otimes\overline{\Phi}\vert_{Y}$ is of frequency $(\ell,-\ell)$ and is non-trivial. 
       	 If $\dim(H_{2})=1$, then we may write $$H_{2}=\{(\psi^{-1}(0,0;\ell_{1}t),\psi^{-1}(0,0;\ell_{2}t))\in G_{2}\times G_{2}\colon t\in\mathbb{R}\}$$ for some $\ell_{1},\ell_{2}\in\mathbb{Z}$ not all equal to 0. If $\ell_{1}\neq \ell_{2}$, then by Lemma \ref{good}, $\Phi\otimes\overline{\Phi}\vert_{Y}$ is of non-trivial frequency and we are done.
       	 
       	 So the remaining case is when $$H_{2}=\{(\psi^{-1}(0,0; t),\psi^{-1}(0,0; t))\in G_{2}\times G_{2}\colon t\in\mathbb{R}\},$$
       	 i.e. the second condition in Theorem \ref{auto} holds. We may now invoke Theorem \ref{auto} to conclude that there exists $\sigma\in\Aut_{2}(X_{Hei})$ such that
       	 \begin{equation}\label{equu1}
       	 g(pn)=\sigma\circ g(qn) \mod H_{\ker}
       	 \end{equation} 
       	 for all $n\in\Z$. In other words, Theorem \ref{auto} allows us to reduce the proof of Theorem \ref{key0} to the study of the algebraic equation (\ref{equu1}). In Proposition \ref{equu2}, we will show that (\ref{equu1}) has no solution if $g$ is equidistributed. This contradiction implies that (\ref{eeqq}) is small as $N\to\infty$.
       	\end{ex}
      
       \begin{ex}
       	    If we consider an example similar to Example \ref{equu3} for the field $K=\mathbb{Q}[i]$, then under a similar argument, we may reduce the proof of Theorem \ref{key0} to the case when the following equation holds
       	     \begin{equation}\label{equu5}
       	     g(\n A_{\B}(p))=\sigma\circ g(\n A_{\B}(q)) \mod H_{\ker}
       	     \end{equation}     	    
       	    	for all $\n\in\Z^{2}$, where $A_{\B}(p)=\begin{bmatrix}
       	    	p_{1} & p_{2} \\
       	    	-p_{2} & p_{1}
       	    	\end{bmatrix}$,  $p=(p_{1},p_{2}), q=(q_{1},q_{2})$ are prime elements in $\Z[i]$.
       	    In Propositions \ref{equu4} and \ref{equu6}, we will use some further examples to show that (\ref{equu5}) has no solution if $g$ is equidistributed (and thus this case can not happen).
       	\end{ex}

     \section{Multi-linear analysis along polynomial sequences}\label{s:m}
     In this section, we prove the following theorem, which is another important ingredient for the proofs of our main results:
     
      \begin{thm}\label{contri}
      	Let $\K$ be an integral tuple and
      	$p,q\in \O_{K}\backslash\{0\}$ with $\vert N_{K}(p)\vert\neq\vert N_{K}(q)\vert$. Let $d\in\N_{+}$ and $X=G/\Gamma$ be a nilmanifold
      	 with a standard nil-structure $\G$ and the natural filtration $G_{c,\bullet}=(G_{i})_{0\leq i\leq d+1}$ of natural step $d$. Suppose that $\dim(G_{d})=1$. For all $C>0$,
      	there exist $\delta:=\delta(\X,\B,p,q,C)>0$ and $N_{0}:=N_{0}(\X,\B,p,q,C)\in\N$ such that for every $N\geq N_{0}$, every $g\in\poly_{D}(G_{\bullet})$ and every $\sigma\in\Aut_{d}(X)$ of height at most $C$ with respect to $\X$, denoting $$h(\n):=g(\n A_{\B}(p))\cdot(\sigma\circ g(\n A_{\B}(q)))^{-1}, \n\in\Z^{D},$$ if $\Vert h\Vert'_{C^{\infty}_{\ker,\X}(R_{N,D})}\leq C$, then 
      	$(g(\n)\cdot e_{X})_{z\in R_{N,D}}$ is not totally $\delta$-equidistributed on $X$ with respect to $\X$. 
      \end{thm}	
      
      \begin{rem}
      	Since $\sigma$ is not well defined as a map on $G$, neither is $h$. However, since $h$ is well defined modulo $G_{\ker}$, the norm $\Vert h\Vert'_{C^{\infty}_{\ker,\X}(R_{N,D})}$ is well-defined.
      \end{rem}	
 
      Theorem \ref{contri} is a main technical innovation of this paper, whose argument is rather different from the ones used in \cite{FH}. Before giving the proof of Theorem \ref{contri}, we provide some examples to illustrate the ideas. 
      
      \subsection{Some examples}
      
      Recall that for the Heisenberg group  $H=\mathbb{R}^{3}$, $H_{2}=H_{\ker}=\{0\}\times\{0\}\times \R$ and $H_{3}=\{(0,0;0)\}$. Let $\Gamma=\Z^{3}$ and $X_{Hei}=H/\Gamma$. Our first example is a special case of Theorem \ref{contri} for $K=\mathbb{Q}$ (translated into the qualitative version for the convenience of explanations):
      
      \begin{prop}[Example for $K=\mathbb{Q}$, qualitative version]\label{equu2}
      	Let $g\colon \Z\to H$ be given by $g(n)=a^{n}$ for some $a=(x_{0},y_{0};z_{0})\in H$ for all $n\in\Z$. Suppose that there exist $\sigma\in\Aut_{2}(X_{Hei})$ and $p,q\in\Z\backslash\{0\}, \vert p\vert\neq \vert q\vert$ such that 
      	$$g(pn)=\sigma\circ g(qn) \mod H_{\ker}$$
      	for all $n\in\Z$, then $(g(n)\cdot e_{X_{Hei}})_{n\in\Z}$ is not equidistributed on $X_{Hei}$.\footnote{A sequence $(g(\n)\cdot e_{X})_{\n\in\Z^{D}}$  is \emph{equidistributed} on $X$ if
      	$	\lim_{N\to\infty}\frac{1}{(2N+1)^{D}}\sum_{\n\in\{-N,N\}^{D}} \Phi(g(\n)\cdot e_{X})=0$	
      		for  every  
      		$\Phi\in C(X)$ such that $\int_{X}\Phi\,d m_{X}=0$ (where $m_{X}$ is the Haar measure on $X$).} 
      \end{prop}	
      \begin{proof}
      	Suppose that on the contrary $(g(n)\cdot e_{X_{Hei}})_{n\in\Z}$ is  equidistributed on $X_{Hei}$. Then $x_{0},y_{0},1$ are linear independent over $\mathbb{Q}$. 
        Suppose that $\sigma(x,y)=(x,y)A$ for some $A\in M_{2\times 2}(\Z)$ for all $(x,y)\in\R^{2}=H/H_{\ker}$. By assumption, we have that $p(x_{0},y_{0})=q(x_{0},y_{0})A$.  Since $x_{0},y_{0},1$ are linear independent over $\mathbb{Q}$, we have that $A=\begin{bmatrix}
        p/q & 0 \\
        0 & p/q
        \end{bmatrix}$. In other words, $\sigma(x,y)=\frac{p}{q}(x,y)$ for all $(x,y)\in\R^{2}$. However, by the multi-linearity of $[\cdot,\cdot]$, for all $h_{1},h_{2}\in H/H_{\ker}$,
        $$[h_{1},h_{2}]=[\sigma(h_{1}),\sigma(h_{2})]=[\frac{p}{q}h_{1},\frac{p}{q}h_{2}]=(\frac{p}{q})^{2}[h_{1},h_{2}],$$
        where we used Convention \ref{ker}.
        Since $\vert p\vert\neq \vert q\vert$, $[h_{1},h_{2}]=(0,0;0)$ for all $h_{1},h_{2}\in H/H_{\ker}$, a contradiction.
      \end{proof}	
      
       \begin{rem}
       	In fact, a similar argument applies to the case where $H$ is replaced by a $d$-step nilpotent group. In this case, one can deduce that
       	$$[h_{1},\dots,h_{d}]_{d}=(\frac{p}{q})^{d}[h_{1},\dots,h_{d}]_{d}$$
       	to get a contradiction. This idea provides an alternative approach to prove Theorem 6.1 of \cite{FH}.
       \end{rem}

      We provide another example for the case $K=\mathbb{Q}[i]$. In this case $\O_{K}=\Z[i]$ and we can choose $\B=\{1,i\}$ as the integral basis. Then for every $p=p_{1}+p_{2}i\in \mathbb{Q}[i], p_{1},p_{2}\in\mathbb{Q}$, we have that  $A_{\B}(p)=\begin{bmatrix}
      p_{1} & p_{2} \\
      -p_{2} & p_{1}
      \end{bmatrix}$.

         \begin{prop}[{Example for $K=\mathbb{Q}[i]$, qualitative version}]\label{equu4}
         	Let $g\colon \Z^{2}\to H$ be given by $g(n_{1},n_{2})=a_{1}^{n_{1}}a_{2}^{n_{2}}$ for some $a_{1}=(x_{1},y_{1};z_{1}),a_{2}=(x_{2},y_{2};z_{2})\in H$ for all $(n_{1},n_{2})\in\Z^{2}$. Suppose that there exist $\sigma\in\Aut_{2}(X_{Hei})$ and $p,q\in\Z[i]\backslash\{0\}, \vert N_{\mathbb{Q}[i]}(p)\vert\neq \vert  N_{\mathbb{Q}[i]}(q)\vert$ such that 
         	$$g(\n A_{\B}(p))=\sigma\circ g(\n A_{\B}(q)) \mod H_{\ker}$$
         	for all $\n\in\Z^{2}$, where $\B=\{1,i\}$. Then $(g(\n)\cdot e_{X_{Hei}})_{\n\in\Z^{2}}$ is not equidistributed on $X_{Hei}$. 
         \end{prop}
         \begin{proof}
         	The idea of the proof is similar to Proposition 7.9 of \cite{S}. Suppose that $\sigma(x,y)=(x,y)A$ for some $A\in M_{2\times 2}(\Z)$ for all $(x,y)\in\R^{2}=H/H_{\ker}$. Let $g'\colon\Z\to\R^{2}=H/H_{\ker}$, $g'(n_{1},n_{2}):=n_{1}(x_{1},y_{1})+n_{2}(x_{2},y_{2})$ be the projection of $g$ onto $H/H_{\ker}$.
         	By assumption, we have that
         	$$g'(\n A_{\B}(p/q))=\sigma\circ g'(\n)=g'(\n)A$$
         	for all $\n\in\Z^{2}$. Let $f\in\mathbb{Q}[x]$ be the minimal polynomial of $A_{\B}(p/q)$. By the linearity of $g'$,
         	\begin{equation}\label{61}
         	(0,0)=g'(\n f(A_{\B}(p/q)))=g'(\n)f(A)
         	\end{equation}	
         	for all $\n\in\Z^{2}$. 
         	
         	Suppose that on the contrary $(g(\n)\cdot e_{X_{Hei}})_{\n\in\Z^{2}}$ is  equidistributed on $X_{Hei}$. Then $(g'(\n)\mod \Z^{2})_{\n\in\Z^{2}}$ is  equidistributed on $\T^{2}$. Therefore, (\ref{61}) implies that $f(A)=\O_{2\times 2}$. By Lemma \ref{jordan},  there exist a $2\times 2$ invertible matrix $S$ and a  diagonal matrix $J=\begin{bmatrix}
         	\mu_{1} & 0\\
         	0 & \mu_{2}
         	\end{bmatrix}$ with $f(\mu_{1})=f(\mu_{2})=0$ such that $A=SJS^{-1}$.\footnote{We clarify that $A_{\B}(p/q)$ is a $2\times 2$ matrix because $[\mathbb{Q}[i]:\mathbb{Q}]=2$, while $A$ is a $2\times 2$ matrix because $\dim(H/H_{\ker})=2$.}
         	By Lemma \ref{minimal}, $N_{\mathbb{Q}[i]}(\mu_{1})=N_{\mathbb{Q}[i]}(\mu_{2})=\det(A_{\B}(p/q))$.
         	
         	Since $\sigma\in\Aut_{2}(X_{Hei})$, for all $i,j\in\{1,2\}$,
         	\begin{equation}\nonumber
         	\begin{split}
         	&\quad [\bold{e}_{i}S^{-1},\bold{e}_{j}S^{-1}]=[\sigma(\bold{e}_{i}S^{-1}),\sigma(\bold{e}_{j}S^{-1})]
         	=[\bold{e}_{i}JS^{-1},\bold{e}_{j}JS^{-1}]
         	\\&=[\mu_{i}\bold{e}_{i}S^{-1},\mu_{j}\bold{e}_{j}JS^{-1}]
         	=\mu_{i}\mu_{j}[\bold{e}_{i}S^{-1},\bold{e}_{j}JS^{-1}],
         	\end{split}	
         	\end{equation}	
         	where $\bold{e}_{1}:=(1,0)$ and $\bold{e}_{2}:=(0,1)$. Since $\vert N_{\mathbb{Q}[i]}(\mu_{i}\mu_{j})\vert=\vert\det(A_{\B}(p/q))\vert^{2}\neq 1$, we have that $[\bold{e}_{i}S^{-1},\bold{e}_{j}S^{-1}]=(0,0;0)$ for all $i,j\in\{1,2\}$. This implies that $[h_{1},h_{2}]=(0,0;0)$ for all $h_{1},h_{2}\in H/H_{\ker}$ since $S$ is invertible, which is impossible. This finishes the proof. 
         \end{proof}

      Finally, we provide one more example. This example is more complicated than the previous two examples, but its proof is closer to that of Theorem \ref{contri} in the general case.
      
      \begin{prop}[{Another example for $K=\mathbb{Q}[i]$, qualitative version}]\label{equu6}
      	Let $g\colon \Z^{2}\to H$ be given by $$g(n_{1},n_{2})=\psi^{-1}(a_{1,1}n^{2}_{1}+a_{1,2}n_{1}n_{2}+a_{1,3}n^{2}_{2},a_{2,1}n^{2}_{1}+a_{2,2}n_{1}n_{2}+a_{2,3}n^{2}_{2},0)$$ for some $a_{i,j}\in \R$ for all $(n_{1},n_{2})\in\Z^{2}$. Suppose that there exist $\sigma\in\Aut_{2}(X_{Hei})$ and $p,q\in\Z[i]\backslash\{0\}, \vert N_{\mathbb{Q}[i]}(p)\vert\neq \vert  N_{\mathbb{Q}[i]}(q)\vert$ such that 
      	\begin{equation}\label{equu75}
      	g(\n A_{\B}(p))=\sigma\circ g(\n A_{\B}(q)) \mod H_{\ker}
      	\end{equation}
      	for all $\n\in\Z^{2}$, where $\B=\{1,i\}$. Then $(g(\n)\cdot e_{X_{Hei}})_{\n\in\Z^{2}}$ is not equidistributed on $X_{Hei}$. 
      \end{prop}
      \begin{proof}
      	\textbf{Step 1.} The first step is to rephrase (\ref{equu75}) as an equation for some multi-linear forms.
      	     	
      	 Suppose that $\sigma(x,y)=(x,y)A$ for some $A\in M_{2\times 2}(\Z)$ for all $(x,y)\in\R^{2}=H/H_{\ker}$. Let $g'\colon\Z\to\R^{2}=H/H_{\ker}$, $$g'(n_{1},n_{2}):=n^{2}_{1}(a_{1,1},a_{2,1})+n_{1}n_{2}(a_{1,2},a_{2,2})+n^{2}_{2}(a_{1,3},a_{2,3})$$ be the projection of $g$ onto $H/H_{\ker}$.
      	By assumption, we have that
      	\begin{equation}\label{equu7}
      	g'(\n A_{\B}(p/q))=\sigma\circ g'(\n)=g'(\n)A
      	\end{equation}
      	for all $\n\in\Z^{2}$. 
        The first step is to write $g'$ as a symmetric form. There exist $2\times 2$ symmetric matrices $B_{1}$ and $B_{2}$ such that writing $L\colon(\Z^{2})^{2}\to \R^{2}$, $L(\n,\n'):=((\n B_{1})\cdot \n, (\n B_{2})\cdot \n)$, we have that 
        $g'(\n)=L(\n,\n).$
      	By (\ref{equu7}), we have that
      		\begin{equation}\label{equu8}
      	L(\n A_{\B}(p/q),\n A_{\B}(p/q))=L(\n,\n)A.
      		\end{equation}
      		
      	\textbf{Step 2.} In order to further study equation (\ref{equu8}), we investigate the eigenvectors of $A_{\B}(p/q)$.
      	
      	Let $\v_{1}=(1,-i)$, $\v_{2}=(1,i)$, $\lambda_{1}=p=p_{1}+p_{2}i$, and $\lambda_{2}=\overline{p}=p_{1}-p_{2}i$. Then we have that $\v_{i}A_{\B}(p/q)=\lambda_{i}\v_{i}$ for $i=1,2$. Since $L$ is multi-linear, by (\ref{equu8}),
      		\begin{equation}\label{equu9}
      	\lambda_{i}\lambda_{j}L(\v_{i},\v_{j})=L(\v_{i} A_{\B}(p/q),\v_{j} A_{\B}(p/q))=L(\v_{i},\v_{j})A
      		\end{equation}
      	for all $1\leq i,j\leq 2$.
      	
      	\textbf{Step 3.} Our next step is to iteratively apply (\ref{equu9}) to annihilates the left side of (\ref{equu9}).
      	
      	By (\ref{equu9}), for any polynomial $f$, we have that 
      			\begin{equation}\nonumber
      			f(\lambda_{i}\lambda_{j})L(\v_{i},\v_{j})=L(\v_{i},\v_{j})f(A).
      			\end{equation}
       Let $f_{0}\in\mathbb{Q}[x]$ denote the monic polynomial of the smallest possible positive degree such that $f_{0}(\lambda_{i}\lambda_{j})=0$ for all $1\leq i,j\leq 2$. Then 
        \begin{equation}\nonumber
        L(\v_{i},\v_{j})f_{0}(A)=\bold{0}
        \end{equation}
        for all $1\leq i,j\leq 2$. Since $L$ is multi-linear, we have that 
        \begin{equation}\label{equu10}
        L(\n,\n')f_{0}(A)=\bold{0}
        \end{equation}
        for all $\n,\n'\in\Z^{2}$.  Equation (\ref{equu10}) tells us nothing if $f_{0}(A)$ vanishes. However, the fact that $\sigma$ is a $2$-automorphism ensures that:  
        
        \textbf{Claim.} We have that
       $f_{0}(A)\neq \mathcal{O}_{2\times 2}$.
       
       It is not hard to see that 
       $$f_{0}(x)=((x-p_{1})^{2}+p_{2}^{2})((x-(p^{2}_{1}-p_{2}^{2}))^{2}+4p_{1}^{2}p_{2}^{2})(x-(p_{1}^{2}+p_{2}^{2}))$$
       and it has no repeated roots.
       By Lemma \ref{jordan},  there exist a  $2\times 2$ invertible matrix $S$ and a  diagonal matrix $J=\begin{bmatrix}
       \mu_{1}\\
       & \mu_{2}\\
       \end{bmatrix}$ with $f_{0}(\mu_{1})=f_{0}(\mu_{2})=0$ such that $A=SJS^{-1}$.
       $N_{\mathbb{Q}[i]}(\mu_{1}), N_{\mathbb{Q}[i]}(\mu_{2})\geq\det(A_{\B}(p/q))>1$.
       
       Since $\sigma\in\Aut_{2}(X_{Hei})$, for all $i,j\in\{1,2\}$,
       \begin{equation}\nonumber
       \begin{split}
       &\quad [\bold{e}_{i}S^{-1},\bold{e}_{j}S^{-1}]=[\sigma(\bold{e}_{i}S^{-1}),\sigma(\bold{e}_{j}S^{-1})]
       =[\bold{e}_{i}JS^{-1},\bold{e}_{j}JS^{-1}]
       \\&=[\mu_{i}\bold{e}_{i}S^{-1},\mu_{j}\bold{e}_{j}JS^{-1}]
       =\mu_{i}\mu_{j}[\bold{e}_{i}S^{-1},\bold{e}_{j}JS^{-1}],
       \end{split}	
       \end{equation}	
       where $\bold{e}_{1}:=(1,0)$ and $\bold{e}_{2}:=(0,1)$. Since $\vert N_{\mathbb{Q}[i]}(\mu_{i}\mu_{j})\vert\geq\vert\det(A_{\B}(p/q))\vert^{2}>1$, we have that $\mu_{i}\mu_{j}\neq 1$ and so $[\bold{e}_{i}S^{-1},\bold{e}_{j}S^{-1}]=(0,0;0)$ for all $i,j\in\{1,2\}$. This implies that $[h_{1},h_{2}]=(0,0;0)$ for all $h_{1},h_{2}\in H/H_{\ker}$ since $S$ is invertible, which is impossible. This finishes the proof of the claim.

        \textbf{Step 4.} 
        We are now ready to complete the proof of Proposition \ref{equu6}.
        Since $f_{0}(A)\neq \mathcal{O}_{2\times 2}$, by (\ref{equu10}), $L(\n,\n')$ and thus $g'(\n)$ only takes values in a subgroup of $\R^{2}$ of dimension at most 1. This means that $(g(\n)\cdot e_{X_{Hei}})_{\n\in\Z^{2}}$ is not equidistributed on $X_{Hei}$ and we are done.
      \end{proof}	
      
      \subsection{$m$-symmetric and $m$-diagonal forms}
      As we have seen in Proposition \ref{equu6}, Theorem \ref{contri} is related to a problem on certain multi-linear functions. So we start with a
     generalization of the quadratic form to higher order cases:
     
     \begin{defn}[$m$-symmetric and $m$-diagonal forms]
     	 Let $D,m,s\in\N_{+}$. We say that a map $L\colon(\Z^{D})^{m}\to\mathbb{R}^{s}$ is a \emph{($D$-dimensional) $m$-symmetric form} if for all $\n_{i}=(n_{i,1},\dots,n_{i,D})\in\Z^{D}, 1\leq i\leq m$, we have that
     	 \begin{equation}\label{L}
     	 	\begin{split}
     	 		L(\n_{1},\dots,\n_{m})=\sum_{i_{1},\dots,i_{m}=1}^{D}u_{i_{1},\dots,i_{m}}\prod_{j=1}^{m}n_{j,i_{j}}
     	 	\end{split}	
     	 \end{equation}	
     	 for some  $u_{i_{1},\dots,i_{m}}\in\mathbb{R}^{s}$ such that for any permutation $\tau\colon \{1,\dots,m\}\to \{1,\dots,m\}$, $u_{i_{1},\dots,i_{m}}=u_{i_{\tau(1)},\dots,i_{\tau(m)}}$. 
     	 
     	  We say that a map $R\colon\Z^{D}\to\mathbb{R}^{s}$ is a \emph{($D$-dimensional) $m$-diagonal form} if 
     	  \begin{equation}\label{R}
     	  	\begin{split}
     	  		R(\n)=\sum_{\j\in\N^{D},\vert\j\vert=m}v_{\j}
     	  		\n^{\j}
     	  	\end{split}	
     	  \end{equation}	
     	  for some $v_{\j}\in\mathbb{R}^{s}$.
     \end{defn}
     \begin{conv}
     In the rest of this section, the dimension $D$ is considered as fixed, and we simply say that a function is an $m$-symmetric or $m$-diagonal form for short. 	
     \end{conv}		
    
     For example,
     a 1-symmetric form is of the form $L\colon\Z^{D}\to\R^{s}$, $L(\n)=\v\cdot\n$ for some $\v\in\R^{D}$ for all $\n\in\Z^{D}$, which is just a linear function. A 2-symmetric form is of the form $L\colon\Z^{2D}\to\R^{s}$, $L(\m,\n)=\m A\n^{T}$ for some $A\in M_{D\times D}(\R)$ such that $A^{T}=A$ for all $\m,\n\in\Z^{D}$, which is a quadratic form.
     
   
   The following lemma says that
   there exists a canonical bijection between $m$-symmetric and $m$-diagonal forms: 

    \begin{lem}[Identification between $m$-symmetric and $m$-diagonal forms]\label{formdual}
    	Let $D,m,s\in\N_{+}$. For every $m$-symmetric form $L\colon(\Z^{D})^{m}\to\mathbb{R}^{s}$, there exists a unique $m$-diagonal form $R\colon\Z^{D}\to\mathbb{R}^{s}$ such that $L(\n,\dots,\n)=R(\n)$ for all $\n\in\Z^{D}$, and vice versa.
    \end{lem}
    \begin{proof}
    	Suppose that $L\colon(\Z^{D})^{m}\to\mathbb{R}^{s}$ is an $m$-symmetric form given by (\ref{L}), and $R\colon\Z^{D}\to\mathbb{R}^{s}$ is an $m$-diagonal form given by (\ref{R}). Then $L(\n,\dots,\n)=R(\n)$ for all $\n\in\Z^{D}$ if and only if for all $\j=(j_{1},\dots,j_{D})\in\N^{D}$,
    	\begin{equation}\label{L1}
    		\begin{split}
    			v_{\j}=\sum_{(i_{1},\dots,i_{m})\in U(\j)}u_{i_{1},\dots,i_{m}},
    		\end{split}	
    	\end{equation}	
    	where the set $U(\j)$ consists of all $(i_{1},\dots,i_{m})\in\{1,\dots,D\}^{m}$ such that the set
    	$\{1\leq k\leq m\colon i_{k}=i\}$ is of cardinality $j_{i}$
    	 for all $1\leq i\leq D$. 
    	Since $L$ is an $m$-symmetric form, for all $(i_{1},\dots,i_{m})\in U(\j)$,
    	\begin{equation}\label{R1}
    		\begin{split}
    			u_{i_{1},\dots,i_{m}}=\frac{1}{\vert U(\j)\vert}v_{\j}=\frac{1}{\binom{m}{\j}}v_{\j},
    		\end{split}	
    	\end{equation}
    	 where $\binom{m}{\j}:=\binom{m}{j_{1}}\binom{m-j_{1}}{j_{2}}\dots\binom{m-j_{1}-\dots-j_{D-1}}{j_{D}}$. 
    	 (\ref{L1}) implies that $L$ uniquely determines $R$, and (\ref{R1}) implies that $R$ uniquely determines $L$.
    \end{proof}		 
    
    If $L(\n,\dots,\n)=R(\n)$ for all $\n\in\Z^{D}$ for some $m$-symmetric form $L\colon(\Z^{D})^{m}\to\mathbb{R}^{s}$ and $m$-diagonal form $R\colon\Z^{D}\to\mathbb{R}^{s}$, then we denote $R=\hat{L}$ and $L=\check{R}$. Clearly, $\check{\hat{L}}=L$ and $\hat{\check{R}}=R$.
    
        From  (\ref{L1}) and (\ref{R1}), the following lemma is straightforward:
        \begin{lem}[Vanishing property]\label{form0}
        	Let  $D,m,s\in\N_{+}$ and $L\colon(\Z^{D})^{m}\to\mathbb{R}^{s}$ be an $m$-symmetric form. Then $L\equiv \bold{0}$ if and only if $\hat{L}\equiv \bold{0}$.
        \end{lem}

    Similar to the quadratic forms, the $m$-symmetric forms enjoy many invariance properties:
    
    \begin{lem}[Invariance properties]\label{formbasic}
    	Let $D,m,s\in\N_{+}$ and $L\colon(\Z^{D})^{m}\to\mathbb{R}^{s}$ be an $m$-symmetric form. Then
    	\begin{enumerate}[(i)]
    		\item for all $m$-symmetric form $L'\colon(\Z^{D})^{m}\to\mathbb{R}^{s}$, $L+L'$ is an $m$-symmetric form;
    		\item for all $c\in\mathbb{R}$, $cL$ is an $m$-symmetric form;
    		\item for all $A\in M_{s\times s}(\Z)$,  denoting $A\circ L(\n_{1},\dots,\n_{m}):=L(\n_{1},\dots,\n_{m})\cdot A$, then $A\circ L$ is an $m$-symmetric form;
    		\item for all $A\in M_{D\times D}(\Z)$, denoting $L\circ A(\n_{1},\dots,\n_{m}):=L(\n_{1}A,\dots,\n_{m}A)$, then $L\circ A$ is an $m$-symmetric form.
    	\end{enumerate}	 
    \end{lem}

\begin{proof}
	(i), (ii) and (iii) are straightforward by definition, and so we only prove (iv). 
	
	Denote $\n_{i}=(n_{i,1},\dots,n_{i,D})\in\Z^{D}$ for $1\leq i\leq m$.  Suppose that
	$$L(\n_{1},\dots,\n_{m})=\sum_{i_{1},\dots,i_{m}=1}^{D}u_{i_{1},\dots,i_{m}}\prod_{j=1}^{m}n_{j,i_{j}}.$$
	Let $\tau\colon\{1,\dots,m\}\to\{1,\dots,m\}$ be a permutation. Then $u_{i_{1},\dots,i_{m}}=u_{i_{\tau(1)},\dots,i_{\tau(m)}}$ for all $1\leq i_{1},\dots,i_{m}\leq D$.
	
	Suppose that $A=(a_{k,i})_{1\leq i, k\leq D}$. Then $(\n_{j}A)_{i}=\sum_{k=1}^{D}n_{j,k}a_{k,i}$ for all $1\leq j\leq m, 1\leq i\leq D$.
	So
	\begin{equation}\label{141}
		\begin{split}
		L\circ A(\n_{1},\dots,\n_{m})
			=\sum_{i_{1},\dots,i_{m}=1}^{D}u_{i_{1},\dots,i_{m}}\prod_{j=1}^{m}(\sum_{k=1}^{D}n_{j,k}a_{k,i_{j}})
			=\sum_{i_{1},\dots,i_{m}=1}^{D}u'_{i_{1},\dots,i_{m}}\prod_{j=1}^{m}n_{j,i_{j}},
		\end{split}	
	\end{equation}
	where
	\begin{equation}\label{414}
		\begin{split}
		u'_{i_{1},\dots,i_{m}}=\sum_{i'_{1},\dots,i'_{m}=1}^{D}u_{i'_{1},\dots,i'_{m}}a_{i_{1},i'_{1}}\cdot\ldots\cdot a_{i_{m},i'_{m}}.
		\end{split}	
	\end{equation}
	So
	\begin{equation}\nonumber
		\begin{split}
		&\quad	u'_{i_{\tau(1)},\dots,i_{\tau(m)}}=\sum_{i'_{1},\dots,i'_{m}=1}^{D}u_{i'_{1},\dots,i'_{m}}a_{i_{\tau(1)},i'_{1}}\cdot\ldots\cdot a_{i_{\tau(m)},i'_{m}}
			=\sum_{i'_{1},\dots,i'_{m}=1}^{D}u_{i'_{1},\dots,i'_{m}}a_{i_{1},i'_{\tau^{-1}(1)}}\cdot\ldots\cdot a_{i_{m},i'_{\tau^{-1}(m)}}
			\\&=\sum_{j_{1},\dots,j_{m}=1}^{D}u_{j_{\tau(1)},\dots,j_{\tau(m)}}a_{i_{1},j_{1}}\cdot\ldots\cdot a_{i_{m},j_{m}}=\sum_{j_{1},\dots,j_{m}=1}^{D}u_{j_{1},\dots,j_{m}}a_{i_{1},j_{1}}\cdot\ldots\cdot a_{i_{m},j_{m}}=u'_{i_{1},\dots,i_{m}}
		\end{split}	
	\end{equation}
	for all $1\leq i_{1},\dots, i_{m}\leq D$. This implies that $L\circ A$ is an $m$-symmetric form.
\end{proof}

\subsection{Proof of Theorem \ref{contri}}
 We are now ready to prove Theorem \ref{contri} in this section. Although heavy in notations, the proof of Theorem \ref{contri} is similar to that of Proposition \ref{equu6}. The only major difference is that the quantitative feature of Theorem \ref{contri} requires us to bear with an error term throughout the proof. 
 
 \textbf{Step 1: converting $g(\n)$ into $m$-symmetric forms.}
    	Let $s'=\dim(G)-\dim(G_{\ker})$. Recall that $\pi_{\ker}\colon G\to G/G_{\ker}$ is the quotient map and $\psi_{\ker}\colon G/G_{\ker}\to\R^{s'}$ is the map induced by $\psi$. 
    	Suppose that 
    	$g_{\ker}:=\psi_{\ker}\circ\pi_{\ker}\circ g\colon\Z^{D}\to\mathbb{R}^{s'}$ is given by
    	$$g_{\ker}(\n):=\sum_{\j\in\N^{D},\vert\j\vert\leq k}a'_{\j}\n^{\j}$$
    	for some $k\in\N, a'_{\j}\in\mathbb{R}^{s'}$ for all $\vert\j\vert\leq k$ (where $k$ depends only on $\X$). 
    Suppose that $\sigma=\sigma_{\psi,A}$ for some  $A\in M_{s'\times s'}(\mathbb{Z})$ of height at most $C$. Then
    	letting $h_{\ker}:=\psi_{\ker}\circ\pi_{\ker}\circ h\colon\Z^{D}\to\mathbb{R}^{s'}$, we have that
    	$$h_{\ker}(\n)=\sum_{\j\in\N^{D},\vert\j\vert\leq k}\Bigl(a'_{\j}(\n A_{\B}(p))^{\j}-(a'_{\j}\cdot A)(\n A_{\B}(q))^{\j}\Bigr).$$
    	Since $\Vert h\Vert'_{C^{\infty}_{\ker,\X}(R_{N,D})}\leq C$, by definition, $\Vert h_{\ker}\Vert'_{C_{s'}^{\infty}(R_{N,D})}\leq C$. We may assume without loss of generality that $\vert N_{K}(p)\vert>\vert N_{K}(q)\vert$.
    	Denote 
    	$R\colon \Z^{D}\to\mathbb{R}^{s'}$ by
    	$$R(\n):=h_{\ker}(\n A_{\B}(q)^{-1})=\sum_{\j\in\N^{D},\vert\j\vert\leq k}\Bigl(a'_{\j}(\n A_{\B}(p/q))^{\j}-(a'_{\j}\cdot A)\n^{\j}\Bigr).$$ 
    	For $0\leq m\leq k$, let $R_{m}$ be the $m$-diagonal form given by
    	\begin{equation}\label{145}
    		R_{m}(\n):=\sum_{\j\in\N^{D},\vert\j\vert=m}\Bigl(a'_{\j}(\n A_{\B}(p/q))^{\j}-(a'_{\j}\cdot A)\n^{\j}\Bigr):=\sum_{\j\in\N^{D},\vert\j\vert=m}\epsilon_{m,\j}\n^{\j}
    	\end{equation}	
    	for some $\epsilon_{m,\j}\in\mathbb{R}^{s'}$ for all $\vert\j\vert=m$. Then $R(\n)=\sum_{m=0}^{k}R_{m}(\n)$.
    		Since $\Vert h_{\ker}\Vert'_{C_{s'}^{\infty}(R_{N,D})}\leq C$, by (\ref{141}) and (\ref{414}),  $\Vert R_{m}\Vert'_{C_{s'}^{\infty}(R_{N,D})}\leq CD^{m}H(q^{-1})^{m}$ for all $1\leq m\leq k$, where $H(q^{-1})$ is the height of $A_{\B}(q^{-1})$ which is finite. So
		\begin{equation}\label{56}
    	 	\begin{split}
    	 		\Vert \epsilon_{m,\j}\Vert_{\T^{s'}}\leq C_{1,m}/N^{m}
    	 	\end{split}	
    	 \end{equation}
		for some $C_{1,m}:=C_{1,m}(\X,p,q,\B,C)>0$. 
    	
    	Fix $1\leq m\leq k$, and
    	let $R'_{m}\colon\Z^{D}\to\mathbb{R}^{s'}$ be the $m$-diagonal form given by    
    		\begin{equation}\label{41}
    	 	\begin{split}
    	 		R'_{m}(\n):=\sum_{\j\in\N^{D},\vert\j\vert=m}a'_{\j}\n^{\j}.
			   \end{split}	
    	 \end{equation}
    	 
    	 Then $$L_{m}:=\check{R_{m}} \text{ and } L'_{m}:=\check{R'_{m}}$$ are $m$-symmetric forms. 
    	   By Lemma \ref{formbasic}, 
    	   \begin{equation}\label{40}
    	   	\begin{split}
    	   	L''_{m}:=L'_{m}\circ A_{\B}(p/q)-A\circ L'_{m}-L_{m}
    	   	\end{split}	
    	   \end{equation}	 
    	    is also an $m$-symmetric form (where one should consider $L_{m}$ as the ``error term"). By (\ref{145}), (\ref{41}) and (\ref{40}),
    	 \begin{equation}\nonumber
    	 	\begin{split}
    	 		&\quad\hat{L''_{m}}(\n)=L''_{m}(\n,\dots,\n)=L'_{m}(\n A_{\B}(p/q),\dots,\n A_{\B}(p/q))-A\circ L'_{m}(\n,\dots,\n)-L_{m}(\n,\dots,\n)
    	 		\\&=R'_{m}(\n A_{\B}(p/q))-A\circ R'_{m}(\n)-R_{m}(\n)=\bold{0}
    	 	\end{split}	
    	 \end{equation}	 
    for all $\n\in\Z^{D}$.
    By Lemma \ref{form0}, for all $1\leq m\leq k$,
     \begin{equation}\label{401}
     	\begin{split}
     		L''_{m}\equiv \bold{0}.
     	\end{split}	
     \end{equation}

    \textbf{Step 2: using eigenvectors to express $L'_{m}$ and $L_{m}$.}	
    For all $\x_{1},\dots,\x_{d}\in\R^{s'}$, $[\psi^{-1}\x_{1},\dots,\psi^{-1}\x_{d}]_{d}\in G_{d}$ (where we use Convention \ref{ker} to define $\psi^{-1}\x_{i}$) and so $\psi([\psi^{-1}\x_{1},\dots,\psi^{-1}\x_{d}]_{d})=(0,\dots,0;t)$ for some $t\in\R$. 
    Denote $F(\x_{1},\dots,\x_{d}):=t$. Then $F\colon(\R^{s'})^{d}\to\R$ is a multi-linear function on $(\R^{s'})^{d}$. So
    	$F$ can be extended to a multi-linear function from $(\mathbb{C}^{s'})^{d}$ to $\mathbb{C}$ in the natural way, which for convenience is still denoted by $F$. 
    	Since $\sigma=\sigma_{\psi,A}\in\Aut_{d}(X)$, we have that 
    	$$[\psi^{-1}(\x_{1}A),\dots,\psi^{-1}( \x_{d}A)]_{d}=[\psi^{-1}\x_{1},\dots,\psi^{-1}\x_{d}]_{d}$$
    	for all $\x_{1},\dots,\x_{d}\in\R^{s'}$. So
    	\begin{equation}\label{47}
    		\begin{split}
    			F(\x_{1}A,\dots,\x_{d}A)=F(\x_{1},\dots,\x_{d})
    		\end{split}	
    		\end{equation}
    		for all $\x_{1},\dots,\x_{d}\in\R^{s'}$ and so for all $\x_{1},\dots,\x_{d}\in\mathbb{C}^{s'}$.
    	
    	Since $L_{m},L'_{m},L''_{m}\colon(\Z^{D})^{m}\to\R^{s'}$ are multi-linear functions, they can also be extended to  multi-linear functions from $(\mathbb{C}^{D})^{m}$ to $\mathbb{C}^{s'}$ in the natural way, which for convenience are still denoted by $L_{m},L'_{m}$ and $L''_{m}$, respectively. Since $L''_{m}(\n_{1},\dots,\n_{m})\equiv \bold{0}$ for all $\n_{1},\dots,\n_{m}\in\Z^{D}$ by (\ref{401}),  we have that for all $\u_{1},\dots,\u_{m}\in\mathbb{C}^{D}$, $L''_{m}(\u_{1},\dots,\u_{m})\equiv \bold{0}$. So by (\ref{40}),
    	\begin{equation}\label{42}
    		\begin{split}
    			L'_{m}(\u_{1}A_{\B}(p/q),\dots,\u_{m}A_{\B}(p/q))= A\circ L'_{m}(\u_{1},\dots,\u_{m})+L_{m}(\u_{1},\dots,\u_{m}).
    		\end{split}	
    	\end{equation}


    	 Let $\overline{K}$ denote the normal closure of $K$.
    	 By Lemma \ref{di},
    	there exist a basis $\v_{1},\dots,\v_{D}\in\mathbb{C}^{D}$ of $\mathbb{C}^{D}$ (over $\mathbb{C}$) depending only on $\B$, 
    	and $\lambda_{i}\in\mathbb{C},1\leq i\leq D$ depending on $\B, p$ and $q$, 
    	such that $\v_{i}A_{\B}(p/q)=\lambda_{i}\v_{i}$  for all $1\leq i\leq D$. By Lemma \ref{minimal}, $\lambda_{i}\in \overline{K}$.
    Since $\vert N_{K}(p)\vert>\vert N_{K}(q)\vert$, by Lemma \ref{minimal}, $\vert N_{\overline{K}}(\lambda_{1})\vert,\dots,\vert N_{\overline{K}}(\lambda_{D})\vert>1$. Denote $$\kappa:=\kappa(p,q,\B)=\min_{1\leq i\leq D}\vert N_{\overline{K}}(\lambda_{i})\vert>1.$$

        By (\ref{42}), for all $1\leq i_{1},\dots, i_{m}\leq D$, we have that
         \begin{equation}\label{43}
         \begin{split}
         &\quad \Bigl(\prod_{j=1}^{m}\lambda_{i_{j}}\Bigr)\cdot L'_{m}(\v_{i_{1}},\dots,\v_{i_{m}})=L'_{m}(\v_{i_{1}}A_{\B}(p/q),\dots,\v_{i_{m}}A_{\B}(p/q))
         \\&=A\circ L'_{m}(\v_{i_{1}},\dots,\v_{i_{m}})+L_{m}(\v_{i_{1}},\dots,\v_{i_{m}}).
         \end{split}	
         \end{equation}	
        Denote $$V_{m}=\Bigl\{(\v_{i_{1}},\dots,\v_{i_{m}})\in (\mathbb{C}^{D})^{m}\colon 1\leq i_{1},\dots,i_{m}\leq D\Bigr\}.$$
        Clearly, $\text{span}_{\mathbb{C}}V_{m}=(\mathbb{C}^{D})^{m}$. 
         For $\tilde{\v}_{m}=(\v_{i_{1}},\dots,\v_{i_{m}})\in V_{m}$, let $\lambda_{\tilde{\v}_{m}}=\prod_{j=1}^{m}\lambda_{i_{j}}$. Then (\ref{43}) implies that
        \begin{equation}\label{44}
        	\begin{split}
        		\lambda_{\tilde{\v}_{m}}\cdot L'_{m}(\tilde{\v}_{m})=A\circ L'_{m}(\tilde{\v}_{m})+L_{m}(\tilde{\v}_{m})
        	\end{split}	
        \end{equation}	        
         for all $\tilde{\v}_{m}\in V_{m}$ and $1\leq m\leq k$.

         \textbf{Step 3: iterating (\ref{44}) with polynomials.}       
          By induction, it is not hard to show from (\ref{44}) that for all $n\in\N_{+}$,	
         $$\lambda^{n}_{\tilde{\v}_{m}}\cdot L'_{m}(\tilde{\v}_{m})=A^{n}\circ L'_{m}(\tilde{\v}_{m})+B_{\tilde{\v}_{m},A,n}\circ L_{m}(\tilde{\v}_{m}),$$
         where $$B_{\tilde{\v}_{m},A,n}:=\sum_{i=0}^{n-1}\lambda_{\tilde{\v}_{m}}^{n-1-i}A^{i}\in M_{s'\times s'}(\mathbb{C}).\footnote{Here $\lambda_{\tilde{\v}}^{0}:=1$ and $A^{0}:=I_{s'\times s'}$ is the $s'\times s'$ identity matrix.}$$
         So for all $f(x)=\sum_{i=0}^{r}a_{i}x^{i}\in\mathbb{C}[x]$, we have that
          \begin{equation}\nonumber
          	\begin{split}
          		f(\lambda_{\tilde{\v}_{m}})\cdot L'_{m}(\tilde{\v}_{m})=f(A)\circ L'_{m}(\tilde{\v}_{m})+B_{\tilde{\v}_{m},A,f}\circ L_{m}(\tilde{\v}_{m}),
          	\end{split}	
          \end{equation}	
 where
          $$B_{\tilde{\v}_{m},A,f}:=\sum_{i=0}^{r}a_{i}B_{\tilde{\v}_{m},A,i}.$$
         Let $f_{0}\in\mathbb{Q}[x]$ denote the monic polynomial of the smallest possible positive degree such that $f_{0}(\lambda_{\tilde{\v}_{m}})=0$ for all $1\leq m\leq k$ and  $\lambda_{\tilde{\v}_{m}}\in V_{m}$. Then
         \begin{equation}\label{51}
         	\begin{split}
         		f_{0}(A)\circ L'_{m}(\tilde{\v}_{m})=-B_{\tilde{\v}_{m},A,f_{0}}\circ L_{m}(\tilde{\v}_{m})
         	\end{split}	
         \end{equation}
         for all $1\leq m\leq k$ and   $\lambda_{\tilde{\v}_{m}}\in V_{m}$, where the heights of $f_{0}(A)$ and $B_{\tilde{\v}_{m},A,f_{0}}, \tilde{\v}_{m}\in V_{m}, 1\leq m\leq k$ are bounded above by some constant $C_{2}:=C_{2}(\X,D,p,q,C)>0.$

         \textbf{Claim 1:} $f_{0}$ has no repeated roots, and the absolute value of the $\overline{K}$-norm of all the roots of $f_{0}$ are at least $\kappa$.
         
         Let $f_{\tilde{\v}_{m}}$ denote the minimal polynomial of $\lambda_{\tilde{\v}_{m}}$ for all $1\leq m\leq k$ and  $\tilde{\v}_{m}\in V_{m}$. Then for all $\tilde{\v}_{m}\in V_{m}$ and $\tilde{\v}'_{m'}\in V_{m'}, 1\leq m,m'\leq k$,  either $f_{\tilde{\v}_{m}}=f_{\tilde{\v}'_{m'}}$, or $f_{\tilde{\v}_{m}}$ and $f_{\tilde{\v}'_{m'}}$ have no common roots. So $f_{0}$ is a constant multiple of the products of all the different polynomials appearing in the set $\{f_{\tilde{\v}_{m}}\colon 1\leq m\leq k, \tilde{\v}_{m}\in V_{m}\}$. Since each $f_{\tilde{\v}_{m}}$ has no repeated roots by Lemma \ref{minimal}, so does $f_{0}$.
         
          On the other hand,  by Lemma \ref{minimal}, all the roots of  $f_{\tilde{\v}_{m}}$ have the same absolute value of the $\overline{K}$-norm as $\vert N_{\overline{K}}(\lambda_{\tilde{\v}_{m}})\vert$, which is at least $\kappa^{m}\geq \kappa$. So the  absolute value of the $\overline{K}$-norm of all the roots of $f_{0}$ are at least $\kappa$. This finishes the proof of the claim.

         \textbf{Claim 2:} $f_{0}(A)\neq \mathcal{O}_{s'\times s'}$.
         
         Suppose that $f_{0}(A)=\mathcal{O}_{s'\times s'}$. By Claim 1, $f_{0}$ has no repeated roots. By Lemma \ref{jordan},  there exist an $s'\times s'$ invertible matrix $S$ and a  diagonal matrix $J=\begin{bmatrix}
         \mu_{1}\\
         & \mu_{2}\\
         & & \dots\\
         & & & \mu_{s'}
         \end{bmatrix}$ with $f_{0}(\mu_{1})=\dots=f_{0}(\mu_{s'})=0$ such that $A=SJS^{-1}$. Again by Claim 1, we have that $\vert N_{\overline{K}}(\mu_{1})\vert,\dots,\vert N_{\overline{K}}(\mu_{s'})\vert>\kappa$. By (\ref{47}),
         for all $\x_{1},\dots,\x_{d}\in\mathbb{C}^{s'}$,
         \begin{equation}\nonumber
         	\begin{split}
         		 F(\x_{1}S^{-1},\dots,\x_{d}S^{-1})
         	=F(\x_{1}S^{-1}A,\dots,\x_{d}S^{-1}A)=F(\x_{1}JS^{-1},\dots,\x_{d}JS^{-1}).
         	\end{split}	
         \end{equation}
         Recall that $\bold{e}_{i}\in\mathbb{C}^{s'}$ denotes the vector whose $i$-th coordinate is 1 and all other coordinates are 0. By the definition of $G_{\ker}$, $F$ is not constant $0$. So by the multi-linearity of $F$ and invertibility of $S$, there exist $1\leq i_{1},\dots,i_{d}\leq s'$ such that $F(\bold{e}_{i_{1}}S^{-1},\dots,\bold{e}_{i_{d}}S^{-1})\neq 0$. Then
         $$F(\bold{e}_{i_{1}}S^{-1},\dots,\bold{e}_{i_{d}}S^{-1})=F(\bold{e}_{i_{1}}JS^{-1},\dots,\bold{e}_{i_{d}}JS^{-1})=(\prod_{j=1}^{d}\mu_{i_{j}})\cdot F(\bold{e}_{i_{1}}S^{-1},\dots,\bold{e}_{i_{d}}S^{-1}).$$
         Since $\vert N_{\overline{K}}(\prod_{j=1}^{d}\mu_{i_{j}})\vert\geq\kappa^{d}>1$, this is impossible. This contradiction implies that $f_{0}(A)\neq \mathcal{O}_{s'\times s'}$.
         
         \textbf{Step 4: finishing the proof.}         
         By Claim 2, there exists a row $\c=(c_{1},\dots,c_{s})\in\mathbb{Q}^{s'}$ of the matrix $f_{0}(A)$ which is  non-zero. Moreover, the height of $\c$ is at most $C_{2}$. 
By (\ref{51}), for all $1\leq m\leq k$  and $\tilde{\v}_{m}\in V_{m}$, 
         $$\c\cdot L'_{m}(\tilde{\v}_{m})=\c_{\tilde{\v}_{m}}\cdot L_{m}(\tilde{\v}_{m}),$$
where $\c_{\tilde{\v}_{m}}\in\mathbb{C}^{s'}$ is a row of the matrix $-B_{\tilde{\v}_{m},A,f_{0}}$. By (\ref{56}),
$$\Vert\c\cdot L'_{m}(\tilde{\v}_{m})\Vert_{\T}\leq \frac{C_{3,m}}{N^{m}}\footnote{Recall that for $z=a+bi\in\mathbb{C}$ for some $a,b\in\R$, $\Vert z\Vert_{\T}$ denotes the quantity $\Vert a\Vert_{\T}+\Vert b\Vert_{\T}$.}$$
for all $1\leq m\leq k$, $\tilde{\v}_{m}\in V_{m}$ for some $C_{3,m}:=C_{3,m}(\X,D,p,q,C)>0.$
Since $\text{span}_{\mathbb{C}}V_{m}=(\mathbb{C}^{D})^{m}$, by the multi-linearity of $L'_{m}$, we have that for all $\n_{1},\dots,\n_{m}\in \Z^{D}$,        
$$\Vert\c\cdot L'_{m}(\n_{1},\dots,\n_{m})\Vert_{\T}\leq \frac{C_{4,m}\prod_{i=1}^{m}\vert\n_{i}\vert}{N^{m}}$$
for  some $C_{4,m}:=C_{4,m}(\X,D,p,q,C)>0.$ 
So
$$\Vert\c\cdot R'_{m}(\n)\Vert_{\T}\leq \frac{C_{4,m}\vert\n\vert^{m}}{N^{m}}$$
for all $\n\in\Z$ and $1\leq m\leq k$. By Lemma \ref{popo}, 
$$\Vert Q_{m}\c\cdot a'_{\j}\Vert_{\T}\leq \frac{C_{5,m}}{N^{m}}$$
for all $1\leq m\leq k$, $\vert\j\vert=m$ for  some $C_{5,m}:=C_{5,m}(\X,D,p,q,C)>0$ and $ Q_{m}:= Q_{m}(\X,D,p,q,C)\in\N_{+}$. Letting $Q=\prod_{m=1}^{k}Q_{m}$, we have that 
$$\Vert Q\c\cdot a'_{\j}\Vert_{\T}\leq \max_{1\leq m\leq k}\frac{C_{5,m}}{N^{m}}\cdot\frac{Q}{Q_{m}}.$$
for all $1\leq \vert\j\vert\leq k$. Note that $Q\c\cdot a'_{\j}\in\R$ and
 $\c$ is independent of the choice of $1\leq m\leq k$. Since $\X$ is a standard nil-structure,
 the map $\eta\colon G\to \mathbb{T}$ defined by $$\eta(g_{0}):=(Q\c,0,\dots,0)\cdot\psi(g_{0}) \mod\Z, g_{0}\in G$$
 is a
 horizontal character of $X$ with $0<\Vert\eta\Vert_{\X}\leq \dim(G)C_{2}Q$. 
 Since
 $g_{\ker}=\sum_{m=1}^{k}R'_{m}$, we have that
 $\Vert\eta\circ g\Vert'_{C_{1}^{\infty}(R_{N,D})}\leq \max_{1\leq m\leq k}C_{5,m}Q/Q_{m}$.
         By Theorem \ref{inv}, there exist $\delta:=\delta(\X,D,p,q,C)>0$ and $N_{0}:=N_{0}(\X,D,p,q,C)\in\N$ such that for all $N\geq N_{0}$,
         $(g(\n)\cdot e_{X})_{\n\in R_{N,D}}$ is not totally $\delta$-equidistributed on $X$ with respect to $\X$. This finishes the proof of Theorem \ref{contri}.

     \section{Orthogonality of multiplicative functions and nilsequences}\label{s:o}
     In this section, we prove the following central quantitative correlation result of this paper.
     \begin{thm}[Main quantitative correlation result]\label{key}
     	Let $\K$ be an integral tuple and $X=G/\Gamma$ be a nilmanifold with a nil-structure $\X$. For all $w,\epsilon>0$, there exist $\delta:=\delta(\X,w,\B,\epsilon)>0$ and $N_{0}:=N_{0}(\X,w,\B,\epsilon)\in\N$ such that for all $N\geq N_{0}$, the following holds: if there exist $g\in\poly_{D}(G)$ of degree at most $w$, $\m\in\Z^{D}$, $\chi\in\mathcal{M}_{K}$, $\Phi\colon X\to\mathbb{C}$ such that $\Vert \Phi\Vert_{\Lip(\X)}\leq 1$ and $\int_{X}\Phi\,d m_{X}=0$, and a $D$-dimensional arithmetic progression $P$ such that 
     	\begin{equation}\label{6.1}
     		\begin{split}
     			\Bigl\vert\E_{\n\in R_{N,D}}\bold{1}_{P}(\n)\chi(\iota_{\B}( \n))\Phi(g(\n+\m)\cdot e_{X})\Bigr\vert\geq \epsilon,
     		\end{split}
     	\end{equation}	
     	then the sequence $(g(\n)\cdot e_{X})_{\n\in R_{N,D}}$ is not totally $\delta$-equidistributed on $X$ with respect to $\X$.
     \end{thm}		
     
     \subsection{Preliminary reductions}\label{sb}
     Suppose that $X=G/\Gamma$ is of natural step $d$ for some $d\in\N_{+}$. By induction, we may assume the following.
     
     \
     
     \textbf{Assumption 1:} either (i) $d=1$; or (ii) $d\geq 2$ and the conclusion holds for $d-1$.
     
     \ 
     
     Assume that $g\in\poly_{D}(G_{\bullet})$ for some filtration $G_{\bullet}$ of $G$ (which depends only on the degree $w$ of $g$), and let $\X'$ be a standard nil-structure of $X$ adapted to $G_{\bullet}$. Since the metrics generated by all nil-structures of $X$ generate the same topology of $X$, all such metrics are equivalent. So there exists $C_{0}:=C_{0}(\X,w)>1$ such that for all $\Phi\colon X\to\mathbb{C}$,
     \begin{equation}\label{ini}
     	C_{0}^{-1}\Vert \Phi\Vert_{\Lip(\X)}\leq \Vert \Phi\Vert_{\Lip(\X')}\leq C_{0}\Vert \Phi\Vert_{\Lip(\X)}.
     \end{equation}

    Therefore, we can make the following assumption:
     
     \
     
     \textbf{Assumption 2:} 
     $\G$ is a standard nil-structure of $X$, and
     $g\in\poly_{D}(G_{\bullet})$.
   
   \
     
%
%
%
     We need some further reductions similar to the ones used in Theorem 6.1 of \cite{FH}, Lemma 3.7 of \cite{GT} and Proposition 7.9 of \cite{S}.

     Denote $m=\dim(G)$, $m'_{2}=\dim(G_{\ker})$, $s'=m-m'_{2}$ and $r=\dim(G_{d})$.
     By approximating $\Phi$ with a smooth function, there exist $C_{1}:=C_{1}(\X,\epsilon)>0$ and $\Phi'\colon X\to\mathbb{C}$ such that
     $$\Vert\Phi-\Phi'\Vert_{L^{\infty}(m_{X})}\leq \epsilon/2, \int_{X}\Phi'\,dm_{X}=0, \text{ and } \Vert\Phi'\Vert_{C^{2m}(\X)}\leq C_{1}.$$
     So (\ref{6.1}) implies that
     		\begin{equation}\label{6.11}
     			\begin{split}
     				\Bigl\vert\E_{\n\in R_{N,D}}\bold{1}_{P}(\n)\chi(\iota_{\B}( \n))\Phi'(g(\n+\m)\cdot e_{X})\Bigr\vert\geq \epsilon/2.
     			\end{split}
     		\end{equation}
     		Recall that $\psi\colon G\to\mathbb{R}^{m}$ is the Mal'cev coordinate map with respect to $G_{\bullet}$. Define $\tilde{\psi}\colon G_{d}\to\T^{r}$ by
     		$$\tilde{\psi}(\psi^{-1}(0,\dots,0;y_{1},\dots,y_{r})):=(y_{1},\dots,y_{r})  \mod \T^{r}$$
     		for all $(y_{1},\dots,y_{r})\in\mathbb{R}^{r}$. Since $\tilde{\psi}$ factors through $\Gamma$, $\tilde{\psi}$ induces an identification between $G_{d}/(G_{d}\cap\Gamma)$ with $\T^{r}$, as well as an identification between
the dual group of $G_{d}/(G_{d}\cap\Gamma)$ with $\Z^{r}$. For $\y\in\Z^{r}$, let $\Phi'_{\y},\Phi''_{\y}\colon X\to\mathbb{C}$ be the functions
     $$\Phi'_{\y}(x):=\int_{\T^{r}}e(-\y\cdot\v)\Phi'(\tilde{\psi}^{-1}(\v)\cdot x)\,d_{m_{\T^{r}}}(\v) \text{ and } \Phi''_{\y}(x):=\frac{\Phi'_{\y}(x)}{\Vert\Phi'_{\y}\Vert_{L^{\infty}(m_{X})}}\footnote{Denote $\Phi''_{\y}:=0$ if $\Vert\Phi'_{\y}\Vert_{L^{\infty}(m_{X})}=0$.}$$
     for all $x\in X$, where $\tilde{\psi}^{-1}(\v)$ is viewed as an arbitrary pre-image of $\v$ in $G_{d}$. Then for all $\y\in\Z^{r}$,
     we have   $\Vert\Phi'_{\y}\Vert_{\Lip(\X)}\leq C_{1}$, $\int_{X}\Phi'_{\y}\,dm_{X}=0$, and 	$\Phi'_{\y}$ is a nilcharacter of $X$ with frequency $\y$ with respect to $\X$. Since $\Vert\Phi'\Vert_{C^{2m}(\X)}\leq C_{1}$, using integration by parts, we have that $\Vert\Phi'_{\y}\Vert_{L^{\infty}(m_{X})}\leq C_{2}(1+\vert \y\vert)^{-2r}$ for some $C_{2}:=C_{2}(\X,\epsilon)>0$. Since for all $x\in X$, 
     $$\Phi'(x)=\sum_{\y\in\Z^{r}}\Phi'_{\y}(x)=\sum_{\y\in\Z^{r}}\Vert\Phi'_{\y}\Vert_{L^{\infty}(m_{X})}\cdot \Phi''_{\y}(x),$$
     by (\ref{6.11}), there exist $\epsilon_{1}:=\epsilon_{1}(\X,\epsilon)>0$, $C_{3}:=C_{3}(\X,\epsilon)>0$ and $\y\in\Z^{r}$ such that $\vert \y\vert\leq C_{3}$ and 
     	\begin{equation}\label{6.12}
     		\begin{split}
     			\Bigl\vert\E_{\n\in R_{N,D}}\bold{1}_{P}(\n)\chi(\iota_{\B}( \n))\Phi''_{\y}(g(\n+\m)\cdot e_{X})\Bigr\vert\geq \epsilon_{1}.
     		\end{split}
     	\end{equation}

     		For $\y\in\Z^{r}$, let $$G_{d,\y}:=\{g\in G_{d}\colon \y\cdot \tilde{\psi}(g)=0\}.$$ Then $G_{d,\y}$ is a subgroup of $G_{d}$ rational for $G_{d}\cap \Gamma$. Let $G_{\y}:=G/G_{d,\y}$ and $\Gamma_{\y}:=\Gamma/(G_{d,\y}\cap\Gamma)$. Then $X_{\y}:=G_{\y}/\Gamma_{\y}$ is a nilmanifold. Let $\pi_{\y}\colon X\to X_{\y}$ be the quotient map and $\X_{\pi_{\y}}=(G_{\pi_{\y},\bullet}, \mathcal{X}_{\pi_{\y}}, d_{G_{\y}},d_{X_{\y}})$ be any standard nil-structure of $X_{\y}$ induced by the quotient map $\pi_{\y}$.      		
     		Then $\Vert f\vert_{X_{\y}}\Vert_{\Lip(\X_{\pi_{\y}})}\leq C_{4}\Vert f\Vert_{\Lip(\X)}$
     		for some $C_{4}:=C_{4}(\X,\epsilon)>0$ for all $\vert\y\vert\leq C_{3}$ and all $f\colon X\to\mathbb{C}$.

     	We first assume that (\ref{6.12}) holds for $\y=\bold{0}$. If $d=1$, then $G=G_{d}$ and so $\Phi''_{\bold{0}}$ is a constant. Since $\int_{X}\Phi''_{\bold{0}}\,dm_{X}=0$, we have that $\Phi''_{\bold{0}}=0$, a contradiction to (\ref{6.12}).
     	
     	Now suppose that $d\geq 2$. Then $G_{d,\bold{0}}=G_{d}$, $G_{\bold{0}}=G/G_{d}$, $\Gamma_{\bold{0}}= G_{d}\cap\Gamma$. So $X_{\bold{0}}=G_{\bold{0}}/\Gamma_{\bold{0}}$ is of natural step $d-1$. The function $\Phi''_{\bold{0}}$ factors through $G_{d}$ and so can be written as $\Phi''_{\bold{0}}=\tilde{\Phi}\circ \pi_{\bold{0}}$ for some function $\tilde{\Phi}\colon X_{\bold{0}}\to\mathbb{C}$. It is easy to see that $\int_{X_{\bold{0}}}\tilde{\Phi}\,d_{m_{X_{\bold{0}}}}=\int_{X}\Phi''_{\bold{0}}\,dm_{X}=0$.
By (\ref{6.12}), we have that 
     	\begin{equation}\nonumber
     		\begin{split}
     			\Bigl\vert\E_{\n\in R_{N,D}}\bold{1}_{P}(\n)\chi(\iota_{\B}( \n))\tilde{\Phi}(\pi_{\bold{0}}\circ g(\n+\m)\cdot e_{X_{\bold{0}}})\Bigr\vert\geq \epsilon_{1}.
     		\end{split}
     	\end{equation}	
     	Since $g\in\poly_{D}(G_{\bullet})$, $\pi_{\bold{0}}\circ g\in\poly_{D}(G_{\bold{0},\bullet})$. Since $G_{\bold{0}}$ is of natural step $d-1$, by induction hypothesis, if $N\geq N_{0}(\X_{\pi_{\bold{0}}},\epsilon_{1},D)$, then the sequence  $(\pi_{\bold{0}}\circ g(\n)\cdot e_{X_{\bold{0}}})_{\n\in R_{N,D}}$ is not totally $\delta:=\delta(\X,\X_{\pi_{\bold{0}}},\epsilon_{1}C^{-1}_{4},D)$-equidistributed on $X_{\bold{0}}$ with respect to $\X_{\pi_{\bold{0}}}$, which implies that $(g(\n)\cdot e_{X})_{\n\in R_{N,D}}$ is not totally $C^{-1}_{4}\delta$-equidistributed on $X$ with respect to $\X$. This finishes the proof.
     	
     	Now assume that $\y\neq \bold{0}$ and suppose that Theorem \ref{key} holds when $\dim(G_{d})=1$. Note that  $(G_{\y})_{d}=G_{d}/G_{d,\y}$ is of dimension 1. 
     	Since $\w\cdot\y=0$ for all $g=\tilde{\psi}^{-1}(\bold{w})\in G_{d,\y}$.  We have that
     	\begin{equation}\nonumber
     		\begin{split}
     		&\quad \Phi''_{\y}(gx)=\int_{\T^{r}}e(-\y\cdot\v)\Phi''(\tilde{\psi}^{-1}(\v)\cdot gx)\,d_{m_{\T^{r}}}(\v)=\int_{\T^{r}}e(-\y\cdot\v)\Phi''(\tilde{\psi}^{-1}(\v+\bold{w})\cdot x)\,d_{m_{\T^{r}}}(\v)
     		\\&=\int_{\T^{r}}e(-\y\cdot(\v-\bold{w}))\Phi''(\tilde{\psi}^{-1}(\v)\cdot x)\,d_{m_{\T^{r}}}(\v)=\int_{\T^{r}}e(-\y\cdot\v)\Phi''(\tilde{\psi}^{-1}(\v)\cdot x)\,d_{m_{\T^{r}}}(\v)=\Phi''_{\y}(x)
     		\end{split}
     	\end{equation}	
     	for all $x\in X$. So
      there exists $\tilde{\Phi}_{\y}\colon X_{\y}\to\mathbb{C}$ such that $\Phi''_{\y}=\tilde{\Phi}_{\y}\circ \pi_{\y}$. 
     	It is easy to see that $\int_{X_{\y}}\tilde{\Phi}_{\y}\,dm_{X_{\y}}=\int_{X}\Phi''_{\y}\,dm_{X}=0$ and $\Vert\tilde{\Phi}_{\y}\Vert_{\Lip(X_{\y})}\leq C_{4}$.  By (\ref{6.12}), we have that 
     	\begin{equation}\nonumber
     		\begin{split}
     			\Bigl\vert\E_{\n\in R_{N,D}}\bold{1}_{P}(\n)\chi(\iota_{\B}( \n))\tilde{\Phi}_{\y}(\pi_{\y}\circ g(\n+\m)\cdot e_{X_{\y}})\Bigr\vert\geq \epsilon_{1}.
     		\end{split}
     	\end{equation}	
     	Since $g\in\poly_{D}(G_{\bullet})$, $\pi_{\y}\circ g\in\poly_{D}((G_{\y})_{\bullet})$, by assumption, if $N\geq \max_{\Vert\y\Vert\leq C_{3}}N_{0}(\X_{\y},\epsilon_{1},D)$ (which depends only on $\X,\epsilon$ and $D$), then the sequence  $(\pi_{\y}\circ g(\n+\m)\cdot e_{X_{\y}})_{\n\in R_{N,D}}$ is not totally $\max_{\Vert\y\Vert\leq C_{3}}\delta(\X_{\y},C^{-1}_{4}\epsilon_{1},D)$-equidistributed on $X_{\y}$ with respect to $\X_\y$ for some $\vert\y\vert\leq C_{3}$, which implies that $(g(\n+\m)\cdot e_{X})_{\n\in R_{N,D}}$ is not totally $\delta$-equidistributed on $X$ with respect to $\X$ for some $\delta:=\delta(\X,\epsilon,D,C_{4})>0$. This finishes the proof.

%
%

     	Note that $\Phi''_{\y}$ is a nilcharacter of $X$ with non-zero frequency.
     	In conclusion, it now suffices to prove Theorem \ref{key} under the following assumption:
     	
     	\
     	
     	\textbf{Assumption 3:}  $\dim(G_{d})=1$, and $\Phi$ is a nilcharacter of $X$ with frequency $\ell\in\Z\backslash\{0\}$ with respect to $\X$.
     	
     	\
     	
     	By using Theorems \ref{Lei} and \ref{inv}, we may further assume that:
     	
     	\

     	\textbf{Assumption 4:}  $\m=\bold{0}$, and $g(\bold{0})=e_{G}$.
     	
     	\
     	
     	The justification of Assumption 4 is identical to the argument in Section 7.3 of \cite{FH}, and so we omit the proof.
     	
     \subsection{Using Katai's Lemma}
     We now use Katai's Lemma (Lemma \ref{katai}) to get rid of the multiplicative function $\chi$ in the expression of (\ref{6.1}).
     
     Let $C_{\B}>0$ be defined as in Lemma \ref{mink2} and $\mathcal{J}_{p}$ be defined as in Lemma \ref{pi}. We construct a set $\mathcal{P}\subseteq\N$ as follows:
      for every prime integer $p\in\N_{+}$, if $\mathcal{J}_{p}$ consists of principal prime ideals, let $(a)$ be one of them with the smallest $K$-norm $N((a))=\vert N_{K}(a)\vert$ for some $a\in\O_{K}$. By Lemma \ref{mink2}, we may pick some $a'\in\O_{K}$ which is $C_{\B}$-regular such that $(a')=(a)$. We put such an element $a'$ into the set $\mathcal{P}$. Then all the elements in $\mathcal{P}$ are $C_{\B}$-regular, and have pairwise coprime $K$-norms in $\Z$ by Lemma \ref{pi}. For $W\in\N_{+}$, let $\mathcal{P}_{W}$ denote the first $W$ elements in $\mathcal{P}$ (in an arbitrary order).

      By Lemma \ref{pi}, the cardinality of each $\mathcal{J}_{p}$ is at most $D$. By the minimality  of $N((a'))$ and Theorem \ref{dirich},  we have that $$\lim_{W\to\infty}\mathcal{A}_{\mathcal{P}_{W}}\geq \frac{1}{D}\sum_{a\in\O_{K} \text{ is a prime element}}\frac{1}{\vert N_{K}(a)\vert}=\infty,$$ 
     where $\mathcal{A}_{\mathcal{P}_{W}}$ is defined in (\ref{a}).
     So by (\ref{6.1}), the assumption that $\m=\bold{0}$, and Lemma \ref{katai}, there exist $N_{0}:=N_{0}(\epsilon,\B), W:=W(\epsilon,\B)>0$,
      $p,q\in\mathcal{P}_{W}$ with $\vert N_{K}(p)\vert\neq \vert N_{K}(q)\vert$, and $\epsilon_{2}:=\epsilon_{2}(\epsilon,\B)>0$ such that for all $N\geq N_{0}$,
     	\begin{equation}\label{7.4}
     		\begin{split}
     			\Bigl\vert\E_{\n\in R_{N,D}}\bold{1}_{P(p,q)}(\n)\Phi(g(\n A_{\B}(p))\cdot e_{X})\cdot
     			\overline{\Phi}(g(\n A_{\B}(q))\cdot e_{X})\Bigr\vert\geq \epsilon_{2},
     		\end{split}
     	\end{equation}	
     where $$P(p,q):=\{\n\in\Z^{D}\colon \n A_{\B}(p),\n A_{\B}(q)\in P\}.$$ 
     In order to simplify the notations,
     from now on, we assume implicitly that all the quantities are dependent on $p,q$ and so on $W$. Since there are only finitely many pairs of such $p,q$, from now on we may consider $p,q$ as fixed.
     
     \subsection{Factorizing the polynomial sequence}\label{complex}    
     Let $$h_{1}(\n):=g(\n  A_{\B}(p)), h_{2}(\n):=g(\n A_{\B}(q)) \text{ and } h(\n):=(h_{1}(\n),h_{2}(\n)) \text{ for all } \n\in\Z^{D}.$$ Then $h\in \poly_{D}((G\times G)_{\bullet})$. 
     We now use Theorem \ref{5.6} to convert $h(\n)$ into a sequence which is totally equidistributed on a sub nilmanifold  of $X\times X$. This step is again
     similar to the ones used in Theorem 6.1 of \cite{FH} and Proposition 7.9 of \cite{S}.

     Let $\omega\colon\N\to\mathbb{R}_{+}$ be a function to be defined later.
      By Theorem \ref{5.6}, 
      there exists a finite family $\mathcal{F}(M):=\mathcal{F}(\X,D,M)$ of sub nilmanifolds of $X\times X$, which increases with $M$ and independent of $\omega$, a constant  $M_{1}:=M_{1}(\X,\omega,D)\in\N_{+}$, an integer $M^{\ast}\in\N$ with $M^{\ast}\leq M_{1}$, a closed subgroup $H$ of $G\times G$ rational for $\Gamma\times\Gamma$, a nilmanifold $Y:=H/(H\cap(\Gamma\times\Gamma))$ belonging to $\mathcal{F}(M^{\ast})$ with a nil-structure $\mathfrak{Y}=(H_{\bullet},\mathcal{Y},\psi_{Y},d_{H},d_{Y})$ induced by $\X\times \X$, and a factorization
     $h(\n)=\epsilon(\n)h'(\n)\gamma(\n), \n\in R_{N,D}$ with $\epsilon, g', \gamma\in\poly_{D}(G_{\bullet})$ such that
     \begin{itemize}
     	\item $\epsilon\colon R_{N,D}\to G\times G$ is $(M^{\ast},N)$-smooth;
     	\item $h'\in\poly_{D}(H_{\bullet})$ and $(h'(\n)\cdot e_{Y})_{\n\in R_{N,D}}$ is totally $\omega(M^{\ast})$-equidistributed on $Y$ with respect to $\mathfrak{Y}$;
     	\item $\gamma\colon R_{N,D}\to G\times G$ is $M^{\ast}$-rational for $\Gamma\times \Gamma$ and $\gamma(\n)\cdot e_{Y}=\gamma(\n+M^{\ast}\bold{e}_{i})\cdot e_{Y}$ for all $1\leq i\leq D$ and $\n,\n+M^{\ast}\bold{e}_{i}\in R_{N,D}$.
     \end{itemize}	
     
     We may rewrite (\ref{7.4}) as 
     	\begin{equation}\label{7.5}
     		\begin{split}
     			\Bigl\vert\E_{\n\in R_{N,D}}\bold{1}_{P(p,q)}(\n)\Phi\otimes\overline{\Phi}(\epsilon(\n)h'(\n)\gamma(\n)\cdot e_{X\times X})\Bigr\vert\geq \epsilon_{2}.
     		\end{split}
     	\end{equation}
     	
     	Our goal is to remove $\epsilon(\n)$ and $\gamma(\n)$ on the left hand side of (\ref{7.5}).
     	By Corollary B.3 of \cite{FH}, there exists a finite subset $\Sigma(M^{\ast})$ of $G\times G$, which consists of elements $M^{\ast}$-rational for $\Gamma\times \Gamma$ such that every element in $G\times G$ which is $M^{\ast}$-rational for  $\Gamma\times \Gamma$ can be written as $a\gamma_{0}$ for some $a\in \Sigma(M^{\ast})$ and $\gamma_{0}\in\Gamma\times\Gamma$. We may also assume that $e_{G\times G}\in\Sigma(M^{\ast})$. For all $a\in\Sigma(M^{\ast})$, let  $H_{a}:=a^{-1}Ha$, $\Gamma_{a}:=H_{a}\cap(\Gamma\times\Gamma)$ and $Y_{a}:=H_{a}/\Gamma_{a}$. Lemma B.4 of \cite{FH} implies that $H_{a}$ is a subgroup of $G\times G$ rational for $\Gamma\times \Gamma$, and so $Y_{a}$ is a sub nilmanifold of $X\times X$.
     	Let 
     	$\mathfrak{Y}_{a}=((H_{a})_{\bullet},\mathcal{Y}_{a},\psi_{Y_{a}},d_{H_{a}},d_{Y_{a}})$ be a nil-structure of $Y_{a}$ induced by the a-conjugate from $Y$. Then
     	$(H_{a})_{\bullet}$ is the filtration of $H_{a}$ given by $H^{(j)}_{a}=H^{(j)}\cap H_{a}$ for all $j\in\N$. Since $G^{(j)}$ is a normal subgroup of $G$, we have that $H^{(j)}_{a}=a^{-1}H^{(j)}a$. Let 
     	$$\mathcal{F}'(M^{\ast}):=\{Y_{a}\colon Y\in\mathcal{F}(M^{\ast}), a\in\Sigma(M^{\ast})\}.$$
     	
     	By Lemma \ref{dist} and Corollary \ref{5.5}, there exists a function $C_{1}:=C_{1,\X}\colon \N\to\mathbb{R}_{+}$ such that the following properties holds:
     		\begin{enumerate}[label=(\subscript{P}{{\arabic*}})]
     			\item For all $a\in\Sigma(M^{\ast})$ and $g\in G\times G$ with $d_{G}(g,e_{G\times G})\leq M^{\ast}$, we have $d_{G}(a^{-1}ga,e_{G\times G})\leq C_{1}(M^{\ast}) d_{G}(g,e_{G\times G})$;
     			\item For all $a\in\Sigma(M^{\ast})$, $g\in G\times G$ with $d_{G}(g,e_{G\times G})\leq M^{\ast}$, and $x,y\in X\times X$, we have that 
     		 $d_{X}(ga\cdot x,ga\cdot y)\leq C_{1}(M^{\ast}) d_{X}(x,y)$;
     			\item As a result, for all $a\in\Sigma(M^{\ast})$, $g\in G\times G$ with $d_{G}(g,e_{G\times G})\leq M^{\ast}$, and $f\in\Lip(\X\times \X)$, denoting $f_{g}(x):=f(g\cdot x)$ for all $x\in X\times X$, we have that $\Vert f_{g}\Vert_{\Lip(\X\times \X)}\leq C_{1}(M^{\ast})\Vert f\Vert_{\Lip(\X\times \X)}$.
     			\item   For all $Y_{a}\in \mathcal{F}'(M^{\ast})$ and $x,x'\in Y_{a}$, $C_{1}(M^{\ast})^{-1}d_{X\times X}(x,y)\leq d_{Y_{a}}(x,y)\leq C_{1}(M^{\ast})d_{X\times X}(x,y)$;
     			\item There exist a function $\rho:=\rho_{\X,\mathfrak{Y},D}\colon\N\times\mathbb{R}_{+}\to\mathbb{R}_{+}$ with $\lim_{t\to 0^{+}}\rho(M,t)=0$ for all $M\in\N$ ,and 
     			$N_{1}:=N_{1,\X,\mathfrak{Y},D}\colon\mathbb{R}_{+}\to\N$ such that 
     			for every $Y=H/(H\cap(\Gamma\times\Gamma))\in\mathcal{F}(M^{\ast})$, $a\in\Sigma(M^{\ast})$, $t>0$, $N\in\N$ with $N\geq N_{1}(M^{\ast})$ and $f\in\poly_{D}(H_{\bullet})$, if $(f(\n)\cdot e_{Y})_{\n\in R_{N,D}}$ is totally $t$-equidistributed on $Y$ with respect to $\mathfrak{Y}$, then $a^{-1}fa\in\poly_{D}((H_{a})_{\bullet})$ and  $(a^{-1}f(\n)a\cdot e_{Y_{a}})_{\n\in R_{N,D}}$ is totally $\rho(M^{\ast},t)$-equidistributed on $Y_{a}$ with respect to $\Y_{a}$.
     		\end{enumerate}	
     		
     		\
     		
     		We now return to (\ref{7.5}). 
     		For convenience, for every subset $R\subseteq\Z^{D}$, denote $$I(R):=\{\n\in\Z^{D}\colon \n A_{\B}(p),\n A_{\B}(q)\in R\}.$$
     		Then $P(p,q)=I(P)$.
     		Set 
     		\begin{equation}\label{5.7}
     			L:=\Bigl\lfloor\frac{\epsilon_{2}N}{20DC^{3}_{1}(M^{\ast})\vert N_{K}(pq)\vert (M^{\ast})^{2}}\Bigr\rfloor \text{ and } N_{2}(M^{\ast})=N_{2}(\X,\omega,K,\epsilon,M^{\ast}):=\frac{20DC^{3}_{1}(M^{\ast})\vert N_{K}(pq)\vert (M^{\ast})^{2}}{\epsilon_{2}}.
     		\end{equation}	
     From now on we assume that $N\geq N_{0}+N_{1}(M^{\ast})+N_{2}(M^{\ast})$. Then $L\geq 1$ and
     \begin{equation}\label{5.8}
     	\frac{\epsilon_{2}N}{40DC^{3}_{1}(M^{\ast})\vert N_{K}(pq)\vert (M^{\ast})^{2}}\leq L\leq\frac{\epsilon_{2}N}{20DC^{3}_{1}(M^{\ast})\vert N_{K}(pq)\vert (M^{\ast})^{2}}.
     \end{equation}	
     Since $\epsilon_{2}\leq 1$ and $C_{1}(M^{\ast})\geq 1$, $M^{\ast}L\leq N$. 
     
     Let $P_{0}$ be a $D$-dimensional arithmetic progression in $R_{N,D}$ of step $(\vert N_{K}(pq)\vert\cdot M^{\ast},\dots,\vert N_{K}(pq)\vert\cdot M^{\ast})$ and length $(L_{1},\dots,L_{D})$ for some $L\leq L_{i}<2L, 1\leq i\leq D$. Then for all $\n,\n'\in I(P_{0})$, $\n-\n'\in M^{\ast}\cdot\Z^{D}$. So there exist $a\in\Sigma(M^{\ast})$ and $\gamma_{0}\in\Gamma\times\Gamma$ such that for all $\n\in I(P_{0})$, $\gamma(\n)\cdot e_{X\times X}=a\gamma_{0}\cdot e_{X\times X}=a\cdot e_{X\times X}.$
     Denote $h'_{a}(\n):=a^{-1}h(\n)a, \n\in\Z^{D}$ and $(\Phi\otimes\overline{\Phi})_{a}(x):=\Phi\otimes\overline{\Phi}(\epsilon(\n_{0})a\cdot x), x\in X\times X$ for some fixed $\n_{0}\in I(P_{0})$. For all $\n\in P_{0}$, we have
     $$\Phi\otimes\overline{\Phi}(h(\n)\cdot e_{X\times X})=(\Phi\otimes\overline{\Phi})_{a}(a^{-1}\epsilon(\n_{0})^{-1}\epsilon(\n)a h'_{a}(\n)\cdot e_{X\times X}).$$
     Since $\epsilon$ is $(M^{\ast},N)$-smooth, $$d_{G\times G}(\epsilon(\n_{0})^{-1}\epsilon(\n), e_{G\times G})\leq (2DL\vert N_{K}(pq)\vert M^{\ast})\cdot\frac{M^{\ast}}{N}=2DL\vert N_{K}(pq)\vert(M^{\ast})^{2}/N.$$ By (\subscript{P}{1}),
     $$d_{G\times G}(a^{-1}\epsilon(\n_{0})^{-1}\epsilon(\n)a, e_{G\times G})\leq 2C_{1}(M^{\ast})DL\vert N_{K}(pq)\vert(M^{\ast})^{2}/N.$$
     
     Since $\Vert \Phi\Vert_{\Lip(\X)}\leq 1$, we have that $\Vert \Phi\otimes\overline{\Phi}\Vert_{\Lip(\X\times \X)}\leq 2$. 
     By (\subscript{P}{3}), $\Vert(\Phi\otimes\overline{\Phi})_{a}\Vert_{\Lip(\X\times \X)}\leq 2C_{1}(M^{\ast})$. By (\subscript{P}{4}), 
     \begin{equation}\label{1234}
     	\begin{split}
     	\Vert(\Phi\otimes\overline{\Phi})_{a}\Vert_{\Lip(\Y_{a})}\leq 2C^{2}_{1}(M^{\ast}).
     	\end{split}	
     \end{equation}
     
     Since $P_{0}$ is of length at most $2L$, $I(P_{0})$ is of cardinality at most $(2L)^{D}$. So
     \begin{equation}\label{5.13}
     	\begin{split}
     		&\quad \Bigl\vert\E_{\n\in R_{N,D}}\bold{1}_{I(P_{0})}(\n)\bold{1}_{I(P)}(\n)\Phi\otimes\overline{\Phi}(h(\n)\cdot e_{X\times X})-\E_{\n\in R_{N,D}}\bold{1}_{I(P_{0})}(\n)\bold{1}_{I(P)}(\n)(\Phi\otimes\overline{\Phi})_{a}(h'_{a}(\n)\cdot e_{Y_{a}})\Bigr\vert
     		\\&\leq \frac{(2L)^{D}}{(2N+1)^{D}}\cdot 2C^{2}_{1}(M^{\ast})\cdot \frac{2C_{1}(M^{\ast})DL\vert N_{K}(pq)\vert(M^{\ast})^{2}}{N}
     		\\&\leq 4DC^{3}_{1}(M^{\ast})\vert N_{K}(pq)\vert(M^{\ast})^{2}(\frac{L}{N})^{D+1}\leq \frac{\epsilon_{2}}{5}\cdot(\frac{L}{N})^{D}.
     	\end{split}	
     \end{equation}	
 Since $N\geq N_{2}(M^{\ast})$, $L\geq 1$ and (\ref{5.8}) holds. Since $M^{\ast}L\leq N$, we may partition   $R_{N,D}$ into $D$-dimensional  arithmetic progressions $R_{N,D}=\bigcup_{i} P'_{i}$ of step $(\vert N_{K}(pq)\vert\cdot M^{\ast},\dots,\vert N_{K}(pq)\vert\cdot M^{\ast})$ and length between $L$ and $2L$ in each of the $D$ directions. The number of these progressions is bounded above by $(N/L)^{D}$. Note that $$\bold{1}_{I(P)}(\n)=\bold{1}_{I(P)}(\n)\bold{1}_{I(R_{N,D})}(\n)=\sum_{i}\bold{1}_{I(P)}(\n)\bold{1}_{I(P'_{i})}(\n)=\bold{1}_{I(P'_{i}\cap P)}(\n).$$ It follows from (\ref{7.5}) that there exist one of them $P'_{i}$ such that 
     \begin{equation}\nonumber
     	\begin{split}
     		\Bigl\vert\E_{\n\in R_{N,D}}\bold{1}_{I(P'_{i}\cap P)}(\n)\Phi\otimes\overline{\Phi}(h(\n)\cdot e_{X\times X})\Bigr\vert\geq \epsilon_{2}(\frac{L}{N})^{D}.
     	\end{split}
     \end{equation}
     We deduce from (\ref{5.13}) that for some $a\in\Sigma(M^{\ast})$,
    \begin{equation}\label{7.3}
    	\begin{split}
    		\Bigl\vert\E_{\n\in R_{N,D}}\bold{1}_{I(P'_{i}\cap P)}(\n)(\Phi\otimes\overline{\Phi})_{a}(h'_{a}(\n)\cdot e_{Y_{a}})\Bigr\vert\geq \epsilon_{2}(\frac{L}{N})^{D}-\frac{\epsilon_{2}}{5}\cdot(\frac{L}{N})^{D}\geq \frac{\epsilon_{2}}{2}\cdot(\frac{L}{N})^{D}\geq \epsilon_{3}(M^{\ast}),
    	\end{split}
    \end{equation} 
     where $\epsilon_{3}(M^{\ast}):=\epsilon_{3,\X,D,\epsilon}(M^{\ast})=\frac{\epsilon_{2}}{2}\cdot(\frac{\epsilon_{2}}{40DC_{1}^{3}(M^{\ast})\vert N_{K}(pq)\vert(M^{\ast})^{2}})^{D}$ with the last inequality coming from (\ref{5.8}).

     It is easy to see that for every line $\ell\subseteq \R^{D}$, the set $I(P'_{i}\cap P)\cap\ell$ is a 1-dimensional arithmetic progression.
     By
     (\ref{1234}) and (\ref{7.3})
       and Proposition \ref{comp}, there exist $\epsilon_{4}(M^{\ast}):=\epsilon_{4,\epsilon_{3}}(M^{\ast})$ and $N_{3}(\omega,M^{\ast}):=N_{3,\e_{3},\X}(\omega,M^{\ast})>N_{0}+N_{1}(M^{\ast})+N_{2}(M^{\ast})$ such that for all $N>N_{3}(\omega,M^{\ast})$, $\int_{Y_{a}}\Phi\otimes\overline{\Phi}\Big\vert_{Y_{a}}\,d\mu_{Y_{a}}=0$ implies that
     	\begin{equation}\label{7.99}
     		\begin{split}
     			\text{ $(h'_{a}(\n)\cdot e_{Y_{a}})_{\n\in R_{N,D}}$ is not totally $\epsilon_{4}(M^{\ast})$-equidistributed on $Y_{a}$ with respect to $\mathfrak{Y}_{a}$.}
     		\end{split}
     	\end{equation}
     	Moreover, $\e_{4}\colon\N\to\mathbb{R}_{+}$ as a function of $M^{\ast}$ is independent of the choice of the function $\omega$.

On the other hand,
 since $(h'(\n)\cdot e_{Y})_{\n\in R_{N,D}}$ is totally $\omega(M^{\ast})$-equidistributed on $Y$ with respect to $\mathfrak{Y}$,
 by (\subscript{P}{5}),
 $(h'_{a}(\n)\cdot e_{Y_{a}})_{\n\in R_{N,D}}$ is totally $\rho(M^{\ast},\omega(M^{\ast}))$-equidistributed on $Y_{a}$ with respect to $\mathfrak{Y}_{a}$.
 Since $\lim_{t\to\infty}\rho(M,t)=0$, for any function $\zeta\colon\N\to\mathbb{R}_{+}$, there exists $\omega\colon\N\to\mathbb{R}_{+}$ such that $\rho(M,\omega(M))<\zeta(M)$ for all $M\in\N$. 
 In other words, 
 for every function $\zeta\colon\N\to\mathbb{R}_{+}$, there exists $\omega\colon\N\to\mathbb{R}_{+}$ such that for every $a\in\Sigma(M^{\ast})$ and every $N\geq N_{3}(\omega,M^{\ast})$,
 \begin{equation}\label{7.90}
 	\begin{split}
 		\text{ $(h'_{a}(\n)\cdot e_{Y_{a}})_{\n\in R_{N,D}}$ is totally $\zeta(M^{\ast})$-equidistributed on $Y_{a}$ with respect to $\mathfrak{Y}_{a}$.}
 	\end{split}
 \end{equation}

 We are now ready to state the restriction of the function $\omega$: we pick  $\omega$ to be any function such that (\ref{7.90}) holds for $\zeta(M^{\ast})=\e_{4}(M^{\ast})$ (recall that $\e_{4}\colon\N\to\mathbb{R}_{+}$ as a function of $M^{\ast}$ is independent of the choice of $\omega$). Then for every $a\in\Sigma(M^{\ast})$ and $N\geq N_{3}(\omega,M^{\ast})$,
%
%
 \begin{equation}\label{7.9}
 	\begin{split}
 		\text{ $(h'_{a}(\n)\cdot e_{Y_{a}})_{\n\in R_{N,D}}$ is totally $\epsilon_{4}(M^{\ast})$-equidistributed on $Y_{a}$ with respect to $\mathfrak{Y}_{a}$.}
 	\end{split}
 \end{equation}
 Combining (\ref{7.99}) and (\ref{7.9}), we have that 
  \begin{equation}\label{7.2}
  	\begin{split}
  	\int_{Y_{a}}\Phi\otimes\overline{\Phi}\Big\vert_{Y_{a}}\,d\mu_{Y_{a}}\neq 0.
  	\end{split}
  \end{equation}
     
     \subsection{Invoking the key ingredients}
      Denote $h'_{a}(\n)=(h'_{a,1}(\n),h'_{a,2}(\n))$, $h=(h_{1}(\n),h_{2}(\n))$, $\epsilon(\n)=(\epsilon_{1}(\n),\epsilon_{2}(\n))$ and $\gamma(\n)=(\gamma_{1}(\n),\gamma_{2}(\n))$. We are now ready to use the results from Sections \ref{s:d} and \ref{s:m} to finish the proof of Theorem \ref{key}.
     
      Recall that $\Phi$ is a nilcharacter of $X$ with frequency $\ell\neq 0$ with respect to $\X$ by Assumption 2. By Lemma \ref{good2}, $\Phi\otimes\overline{\Phi}$ is a nilcharacter of $X\times X$ with frequency $(\ell,-\ell)$ with respect to $\X\times\X$, and so is $(\Phi\otimes\overline{\Phi})_{a}$.
      
       
       Since $H_{a}<G\times G$, $(H_{a})_{i}<(G\times G)_{i}=G_{i}\times G_{i}$ for all $i\in\N_{+}$. So $(H_{a})_{d}< G_{d}\times G_{d}$. Since $\dim(G_{d})=1$, $\dim((H_{a})_{d})=0,1$ or 2.

         \textbf{Case that $\dim((H_{a})_{d})=0$.} Then the projection of $Y_{a}$ to the first coordinate is not $X$. By the choice of $Y_{a}$, there exist $C_{4}(M^{\ast}):=C_{4,\X,D}(M^{\ast})>0$ and horizontal character $\eta$ of $X$ such that $0<\Vert\eta\Vert_{\X}\leq C_{4}(M^{\ast})$ and $\eta\circ h'_{a,1}=\eta\circ h'_{1}\equiv 0.$    
         Since $\gamma$ takes value in the finite set $\Sigma(M^{\ast})$,
         there exists $Q:=Q(M^{\ast})$ such that $\eta^{Q}\circ\gamma\equiv 0.$
 
         Since $\epsilon_{1}(\n)$ is $(M^{\ast},N)$-smooth, $\epsilon_{1}(Q\n)$ is $(Q^{D}M^{\ast},N)$-smooth. 
         By definition, $$\Vert\eta^{Q}\circ (\epsilon_{1}h'_{a,1})\Vert_{C^{\infty}_{1}(R_{N,D})}\leq C_{5}(M^{\ast})$$ for some $C_{5}(M^{\ast}):=C_{5,\X,D}(M^{\ast})>0$ for all $N\geq N_{4}(M^{\ast}):=N_{4,\X,D}(M^{\ast})>N_{3}(M^{\ast})$.  So
          we have that 
         $$\Vert\eta^{Q}\circ g(\n A_{\B}(p))\Vert_{C^{\infty}_{1}(R_{N,D})}=\Vert\eta^{Q}\circ (\epsilon_{1}h'_{a,1}\gamma_{1})\Vert_{C^{\infty}_{1}(R_{N,D})}\leq C_{5}(M^{\ast}).$$
         By Theorem \ref{inv}, there exist $C_{6}(M^{\ast}):=C_{6,C_{5},\X,D}(M^{\ast})$ and  $N_{5}(M^{\ast}):=N_{5,C_{5},\X,D}(M^{\ast})>N_{4}(M^{\ast})$, such that for all $N\geq N_{5}(M^{\ast})$, $(g(\n A_{\B}(p))\cdot e_{X})_{\n\in R_{N,D}}$ is not totally $C_{6}(M^{\ast})$-equidistributed on $X$ with respect to $\X$. By Proposition \ref{comp}, there exist $\delta:=\max_{M^{\ast}\leq M_{1}}\delta_{C_{6}}(M^{\ast})>0$ and  $N_{6}:=\max_{M^{\ast}\leq M_{1}}N_{6,C_{6}}(M^{\ast})>N_{5}(M^{\ast})$, such that for all $N\geq N_{6}$, $(g(\n)\cdot e_{X})_{\n\in R_{N,D}}$ is not totally $\delta$-equidistributed on $X$ with respect to $\X$. This finishes the proof.
       
       
       \textbf{Case that $\dim((H_{a})_{d})=2$.} 
    In this case, $(H_{a})_{d}=G_{d}\times G_{d}$. Since $\Phi\otimes\overline{\Phi}$ is with frequency $(\ell,-\ell)$ on $X\times X$ with respect to $\X\times\X$, by Lemma \ref{good}, $(\Phi\otimes\overline{\Phi})_{a}\Big\vert_{Y_{a}}$ is also with frequency $(\ell,-\ell)$ on $Y$ with respect to $\Y$.
        So
       $\int_{Y_{a}}\Phi\otimes\overline{\Phi}\Big\vert_{Y_{a}}\,d\mu_{Y_{a}}=0,$ a contradiction to (\ref{7.2}). 
       

        \textbf{Case that $\dim((H_{a})_{d})=1$.} Since $\X$ is standard,
      in this case $$(H_{a})_{d}=\{(\psi^{-1}(0,\dots, 0;\ell_{1}t),\psi^{-1}(0,\dots, 0;\ell_{2}t))\in G^{(d)}\times G^{(d)}\colon t\in\mathbb{R}\}$$ for some $\ell_{1},\ell_{2}\in\mathbb{Z}$ not all equal to 0.
         If $\ell_{1}\neq \ell_{2}$, then by Lemma \ref{good}, $(\Phi\otimes\overline{\Phi})_{a}\Big\vert_{Y_{a}}$ is also with frequency $(\ell,-\ell)$ on $Y$ with respect to $\Y$.
        So
        $\int_{Y_{a}}\Phi\otimes\overline{\Phi}\Big\vert_{Y_{a}}\,d\mu_{Y_{a}}=0,$
        a contradiction to (\ref{7.2}).
       
       So we must have that$$(H_{a})_{d}=\{(\psi^{-1}(0,\dots, 0;t),\psi^{-1}(0,\dots, 0;t))\in G^{(d)}\times G^{(d)}\colon t\in\mathbb{R}\}.$$
       
      If the projection of $Y_{a}$ to one of the two coordinates is not $X$, we are done by the same argument as in the case that $\dim((H_{a})_{d})=0$. So we may assume that the projection of $Y_{a}$ to both coordinates are $X$. Since $\mathcal{F}'(M^{\ast})$ is a finite set,
     by Theorem \ref{auto},
      there exists $\sigma\in \Aut_{d}(G)$ of height at most $C_{7}(M^{\ast}):=C_{7,\X,D,\e}(M^{\ast})>0$ such that $h_{1}=\sigma(h_{2}) \mod G_{\ker}$ for all $(h_{1},h_{2})\in H$.      
     Then
     		$h'_{a,1}(\n)=\sigma\circ h'_{a,2}(\n) \mod G_{\ker}$ and so $h'_{1}(\n)=\sigma\circ h'_{2}(\n) \mod G_{\ker}$ for all $\n\in R_{N,D}$.
     		
     		Since $\gamma$ takes value in the finite set $\Sigma(M^{\ast})$, there exists $Q(M^{\ast})\in\N_{+}$ such that  $$\gamma(Q(M^{\ast})\n)\in\Gamma\times\Gamma$$ for all $\n\in\Z^{D}.$ Let $Q=\prod_{M^{\ast}\leq M_{1}}Q_{M^{\ast}}$.    	
     	Since $\epsilon_{1}(\n)$ and $\epsilon_{2}(\n)$ are $(M^{\ast},N)$-smooth and $\sigma$ is of height at most $C_{7}(M^{\ast})$ with respect to $\X$,
     	$\epsilon_{1}(Q\n)$ and $\sigma\circ\epsilon_{2}(Q\n)$ are $(C_{8}(M^{\ast}),N)$-smooth
     	 for some $C_{8}(M^{\ast}):=C_{8,\X,D,C_{7},Q}(M^{\ast})>0$. By definition, for all $N\geq N_{7}(M^{\ast}):=N_{7,\X,D,C_{7},Q}(M^{\ast})>N_{3}(M^{\ast})$,
     	 $$\Vert\epsilon_{1}(Q\n)\Vert_{C^{\infty}_{\ker}(R_{N,D})}, \Vert\sigma\circ\epsilon_{2}(Q\n)\Vert_{C^{\infty}_{\ker}(R_{N,D})}\leq C_{9}(M^{\ast})$$ 
     	 for some $C_{9}(M^{\ast}):=C_{9,\X,D,\e}(M^{\ast})>0$.
     	Since $h'_{1}=\sigma\circ h'_{2} \mod G_{\ker}$, 
    \begin{equation}\nonumber
    	\begin{split}
    	&\quad	\Bigl\Vert g(Q\n A_{\B}(p))\cdot(\sigma\circ g(Q\n A_{\B}(q)))^{-1}\Bigr\Vert_{C^{\infty}_{\ker}(R_{N,D})}
    	=\Bigl\Vert\epsilon_{1}h'_{1}\gamma_{1}(Q\n)\cdot(\sigma\circ(\epsilon_{2}h'_{2}\gamma_{2}(Q\n)))^{-1}\Bigr\Vert_{C^{\infty}_{\ker}(R_{N,D})}
    	\\&=\Bigl\Vert\epsilon_{1}h'_{1}(Q\n)\cdot(\sigma\circ(\epsilon_{2}h'_{2}(Q\n)))^{-1}\Bigr\Vert_{C^{\infty}_{\ker}(R_{N,D})}=\Bigl\Vert\epsilon_{1}(Q\n)\cdot(\sigma\circ(\epsilon_{2}(Q\n)))^{-1}\Bigr\Vert_{C^{\infty}_{\ker}(R_{N,D})}
    	\leq 2C_{9}(M^{\ast}).
    	\end{split}	
    \end{equation}

%
%
%
    
    Since $\vert N_{K}(p)\vert\neq\vert N_{K}(q)\vert$
     and $M^{\ast}\leq M_{1}$,
    by Theorem \ref{contri}, 	there exist $$\delta':=\delta'(\X,D,\e):=\max_{M^{\ast}\leq M_{1}}\delta_{\X,D,C_{9}}(M^{\ast})>0$$ and $$N_{8}':=N'_{8}(\X,D,\e):=\max_{M^{\ast}\leq M_{1}}N_{8,\X,D,C_{9}}(M^{\ast})>\max_{M^{\ast}\leq M_{1}}N_{7}(M^{\ast})$$ such that for every $N\geq N'_{8}$,
    $(g(Q\n)\cdot e_{X})_{\n\in R_{N,D}}$ and thus $(g(\n)\cdot e_{X})_{\n\in R_{N,D}}$ is not totally $\delta'$-equidistributed on $X$ with respect to $\X$. 
    Since $\X$ depends only on $\X_{0}$ and $w$, and one can verify that all other quantities in the proof depend eventually only on $\X, \epsilon$ and $D$. This (finally!) finishes the proof.
   
      \section{Consequences of Theorem \ref{key}}\label{s:p}
      In this section, we  deduce Theorems \ref{s2}, \ref{key0} and \ref{ap} by using Theorem \ref{key}.

            \subsection{Proofs of Theorems \ref{s2} and \ref{key0}}

            
            \begin{proof}[Proof of Theorem \ref{key0}]
            	Denote $g(n_{1},\dots,n_{d}):=T^{n_{1}}_{1}\cdot\ldots\cdot T_{D}^{n_{d}}$ and let $\G$ be a nil-structure on $X$. 
            	If the conclusion of Theorem \ref{key0} does not hold, then there exist a function $\Phi\in C(X)$ with $\int_{X}\Phi\,dm_{X}=0$, a $D$-dimensional arithmetic progression $P$, $\e>0$ and an infinite set $J\subseteq\N$ such that for all $N\in J$, there exists $\chi\in\mathcal{M}_{K}$ such that 
            	\begin{equation}\nonumber
            		\begin{split}
            			\Bigl\vert\E_{\n\in R_{N,D}}\bold{1}_{P}(\n)\chi(\iota_{\B}(\n))\Phi(g(\n)\cdot e_{X})\Bigr\vert\geq \e.
            		\end{split}
            	\end{equation}			
            	By Theorem \ref{key}, there exist $\delta>0$ and $N_{0}\in\N$ such that the sequence $(g(\n)\cdot e_{X})_{\n\in R_{N,D}}$ is not totally $\delta$-equidistributed on $X$ with respect to $\X$ for all $N>N_{0}, N\in J$. By Theorem \ref{Lei}, there exist $C>0$ independent of $N$ and a horizontal character $\eta$ such that $$0<\Vert\eta\Vert_{\X}\leq C \text{ and }\Vert \eta\circ g\Vert_{C_{1}^{\infty}(R_{N,D})}\leq C.$$ 
            	Since there are only finitely many $\eta$ with $0<\Vert\eta\Vert_{\X}\leq C$, there exist an infinite set $J'\subseteq J$ and a horizontal character $\eta_{0}$ such that for all $N\in J'$, $\Vert \eta_{0}\circ g\Vert_{C_{1}^{\infty}(R_{N,D})}\leq C.$ Letting $N\to\infty$, we have that $\eta_{0}\circ g\equiv 0$. So
            	\begin{equation}\nonumber
            		\begin{split}
            			\lim_{N\to\infty}\E_{\n\in R_{N,D}}e(\eta_{0}(g(\n)\cdot e_{X}))=1\neq 0.
            		\end{split}
            	\end{equation}
            	Since $\Vert\eta_{0}\Vert_{\X}>0$, we have that $\int_{X}e(\eta_{0})\,d\mu=0$. So $(g(\n)\cdot e_{X})_{\n\in\Z^{D}}$ is not equidistributed on $X$, a contradiction. This finishes the proof.
            \end{proof}
            
            %
            %
            For Theorem \ref{s2}, we prove the following stronger version:
            
            \begin{thm}[Quantitative version of Theorem \ref{s2}]\label{s3}
            	Let $\K$ be an integral tuple and $X=G/\Gamma$ be a nilmanifold with a nil-structure $\X$. For every $w,\e>0$, there exist $\e':=\e'(\X,w,\B,\e)>0$ and $N_{0}:=N_{0}(\X,w,\B,\e)\in\N$ such that for every $N\geq N_{0}$, the following holds: if there exist $g\in\poly_{D}(G)$ of degree at most $w$,  $\m\in\Z^{D}$, $\chi\in\mathcal{M}_{K}$, $\Phi\colon X\to\mathbb{C}$ such that $\vert\Phi\vert\leq 1,\Vert\Phi\Vert_{\Lip(\X)}\leq 1$, and a $D$-dimensional arithmetic progression $P$ such that 
            	\begin{equation}\label{s10}
            		\Bigl\vert\E_{\n\in R_{N,D}}\bold{1}_{P}(\n)\chi\circ\iota_{\B}(\n)\Phi(g(\n+\m)\cdot e_{X})\Bigr\vert\geq\e,
            	\end{equation}	
            	then there exists a  $D$-dimensional arithmetic progression $P'$ such that
            	\begin{equation}\nonumber
            		\Bigl\vert\E_{\n\in R_{N,D}}\bold{1}_{P'}(\n)\chi\circ\iota_{\B}(\n)\Bigr\vert\geq\e'.
            	\end{equation}
            \end{thm}
            
            \begin{proof}
            	We assume  implicitly that all the quantities in the proof depends on $\X$, $\B$ and $\e$.
            	Similar to the deduction of Assumption 2 in Section \ref{sb}, we may assume that $\G$ and $g\in\poly_{D}(G_{\bullet})$.

            	Suppose that (\ref{s10}) holds for some choice of the parameters.
            	By using the factorization theorem (Theorem \ref{5.6}),
            	we may deduce from (\ref{s10}) that there exist
            	\begin{itemize}
            	\item A function $\lambda:=\lambda_{\e,\X,\B}\colon\N_{+}\to\R_{+}$;
            	\item For all $M\in\N$ a finite subset $\Sigma(M)\subseteq G$ and a finite collection $\mathcal{F}(M)$ of sub nilmanifolds  $X'=G'/\Gamma'$ of $X$ with nil-structures $\X'$ induced by $\X$;
            	\item For each $a\in\Sigma(M)$, a sub nilmanifold $X'_{a}=G'_{a}/\Gamma'_{a}$  of $X$ with a nil-structure $\X'_{a}$ adapted to the filtration $G'_{a,\bullet}$ induced by the $a$-conjugate of $\X'$, a polynomial sequence $g'_{a}\in\poly_{D}(G'_{a,\bullet})$  and a function $\Phi'_{a}\colon X'_{a}\to \mathbb{C}$ with $\vert\Phi'_{a}\vert\leq 1$ and $\Vert\Phi'_{a}\Vert_{\Lip(\X'_{a})}=1$,
            	\end{itemize}		
            	such that for every function 
            	$\zeta\colon\N\to\R_{+}$, there exist a function $N_{0}\colon\N\to\mathbb{R}_{+}$ and $M_{1}\in\N_{+}$ such that
            	the following holds:

            	\textbf{Property 1.}
            	 There exist $M^{\ast}\in\N$ with $M^{\ast}\leq M_{1}$,
            	 such that if (\ref{s10}) holds for some $N\geq N_{0}(M^{\ast})$, then exist
            	  a $D$-dimensional arithmetic progression $P'$ of $\Z^{D}$ 
            	  an element $a\in\Sigma(M^{\ast})$ such that
            		\begin{equation}\nonumber
            			\Bigl\vert\E_{\n\in R_{N,D}}\bold{1}_{P'}(\n)\chi\circ\iota_{\B}(\n)\cdot\Phi'_{a}(g'_{a}(\n)\cdot e_{X'_{a}})\Bigr\vert>2\lambda(M^{\ast}).
            		\end{equation}
            		
            		\textbf{Property 2.} For all $M^{\ast}\in\N$, $a\in\Sigma(M^{\ast})$ and $N\geq N_{0}(M^{\ast})$,
            		\begin{equation}\nonumber
            			\text{$(g'_{a}(\n)\cdot e_{X'_{a}})_{\n\in R_{N,D}}$ is totally $\zeta(M^{\ast})$-equidistributed on $X'_{a}$ with respect to $\X'_{a}$.}
            		\end{equation}

            	
            	We left the choice of $\zeta$ to the end of the proof (note that the choice of $\zeta$ must be dependent only on $\lambda,\e,w,\X,\B$).	
            	The method we deduce  Properties 1 and 2 from (\ref{s10}) is similar to the one we used to deduce (\ref{7.90}) and (\ref{7.3}) from (\ref{7.5}), or the one used in \cite{FH} to deduce (8.11) and (8.12) from (8.9), or the one used in \cite{S} to deduce (28) and the equidistribution condition right after (28) from the last inequality at the end of page 101 (or at the end page of 41 for the arXiv version).
            	So for conciseness we omit the proof of  Properties 1 and 2 and leave them to the interested readers.
            	
            Assume that $N\geq \max_{M^{\ast}\leq M_{1}}N_{0}(M^{\ast})$.
            	Let $z:=\int_{X'_{a}}\Phi'_{a}\,dm_{X'_{a}}$ and $\Phi'_{0,a}:=\Phi'_{a}-z$.   Then $\int_{X'_{a}}\Phi'_{0,a}\,dm_{X'_{a}}=0$.  By Theorem \ref{key} and Property 2, if we choose $\zeta$ to be the function $\zeta(M):=\delta(\X,w,\B,\lambda(M))$, where $\delta(\X,w,\B,\lambda(M))$ is defined in Theorem \ref{key}, then 
            	\begin{equation}\nonumber
            		\Bigl\vert\E_{\n\in R_{N,D}}\bold{1}_{P'}(\n)\chi\circ\iota_{\B}(\n)\cdot\Phi'_{0,a}(g'_{a}(\n)\cdot e_{X'_{a}})\Bigr\vert<\lambda(M^{\ast}).
            	\end{equation}
            	Since $\vert z\vert\leq 1$, by Property 1, 
            	\begin{equation}\nonumber
            		\Bigl\vert\E_{\n\in R_{N,D}}\bold{1}_{P'}(\n)\chi\circ\iota_{\B}(\n)\Bigr\vert\geq\Bigl\vert\E_{\n\in R_{N,D}}\bold{1}_{P'}(\n)\chi\circ\iota_{\B}(\n)\cdot z\Bigr\vert\geq\lambda(M^{\ast}),
            	\end{equation}
            	which
            	finishes the proof by setting $\e':=\min_{1\leq M^{\ast}\leq M_{1}}\lambda(M^{\ast})$.
            \end{proof}

      \subsection{Properties of the Gowers norms}
      We introduce some basic properties about the Gowers norms
      before proving Theorem \ref{ap}.
      We start with the definitions of the convolution product and the Fourier transformation on $\RR$:
      
      \begin{defn}[Convolution product] Let $N\in\N_{+}$. The convolution product of two functions $f,g\colon\RR\rightarrow\mathbb{C}$ is defined by
      	\begin{equation}\nonumber
      		\begin{split}
      			f*g(\A):=\mathbb{E}_{\Bb\in \RR}f(\A-\Bb)g(\Bb).
      		\end{split}
      	\end{equation}
      \end{defn}
      
      \begin{defn}[Fourier transformation]  For every $\A=(n_{1},\dots,n_{D})$ and $\Bb=(m_{1},\dots,m_{D})\in\RR$, write
      	\begin{equation}\nonumber
      		\begin{split}
      			\A\circ_{N}\Bb:=\frac{1}{N}(\n\cdot \m)=\frac{1}{N}(n_{1}m_{1}+\dots+n_{D}m_{D}).
      		\end{split}
      	\end{equation}
      	For every function $f\colon \RR\rightarrow\mathbb{C}$,  let $\widehat{f}\colon\RR\rightarrow\mathbb{C}$ denote the \emph{Fourier transformation} of $f$ by
      	\begin{equation}\nonumber
      		\begin{split}
      			\widehat{f}(\xi):=\mathbb{E}_{\A\in \RR}f(\A)\cdot e(-\A\circ_{N}\xi)
      		\end{split}
      	\end{equation}
      	for all $\xi\in\RR$ (recall  that $e(x):=\exp(2\pi ix)$ for all $x\in\mathbb{R}$).
      \end{defn}
      
      A direct computation shows that for any function $f$ on $\RR$, we have
      \begin{equation}\label{F}
      	\begin{split}
      		\Vert f\Vert_{U^{2}(\RR)}^{4}=\sum_{\xi\in \RR}\vert\widehat{f}(\xi)\vert^{4}.
      	\end{split}
      \end{equation}

      We provide some basic lemmas on the Gowers norms for later uses.
      The following lemma generalizes Lemma A.6 in ~\cite{FH}:
      
      \begin{lem}\label{1p}
      	Let $D\in\mathbb{N}$ and let $N\in\mathbb{N}$ be a prime integer. For every function $a\colon \RR\rightarrow\mathbb{C}$ and every $D$-dimensional arithmetic progression $P$, we have that
      	\begin{equation}\nonumber
      		\begin{split}
      			\Bigl\vert\mathbb{E}_{\A\in \RR}\bold{1}_{P}(\A)a(\A)\Bigr\vert\leq C\Vert a\Vert_{U^{2}(\RR)}
      		\end{split}
      	\end{equation}
      	for some $C:=C(D)>0$.
      \end{lem}
      \begin{proof}
      	Since $N$ is a prime integer, the norm $\Vert a\Vert_{U^{2}(\RR)}$ is invariant under any change of variables of the form $\n=(n_{1},\dots,n_{D})\to (c_{1}n_{1}+m_{1},\dots,c_{D}n_{D}+m_{D})$ for any $m_{i}, c_{i}\in\Z$ such that $(c_{i},N)=1$ for all $1\leq i\leq D$. So we may assume without loss of generality that $P=[d_{1}]\times\dots [d_{D}]$ for some $1\leq d_{1},\dots,d_{D}\leq N$. A direct computation shows that
      	\begin{equation}\nonumber
      		\begin{split}
      			\vert\widehat{\bold{1}_{P}}(\xi_{1},\dots,\xi_{D})\vert\leq\frac{2^{D}}{N^{D}}\prod_{i=1}^{D}\Bigl\Vert \frac{\xi_{i}}{N}\Bigr\Vert_{\T}=\frac{2^{D}}{\prod_{i=1}^{D}\min\{\xi_{i},N-\xi_{i}\}}
      		\end{split}
      	\end{equation}
      	for all $(\xi_{1},\dots,\xi_{D})\in \RR$. Thus
      	\begin{equation}\nonumber
      		\begin{split}
      			\Bigl\Vert\widehat{\bold{1}_{P}}(\xi_{1},\dots,\xi_{D})\Bigr\Vert_{\ell^{4/3}(\RR)}\leq C
      		\end{split}
      	\end{equation}
      	for some $C:=C(D)>0$. Then by Parseval's identity, H\"{o}lder's inequality, and identity (\ref{F}), we deduce that
      	\begin{equation}\nonumber
      		\begin{split}
      			\Bigl\vert\mathbb{E}_{\A\in \RR}\bold{1}_{P}(\A)a(\A)\Bigr\vert
      			=\Bigl\vert\sum_{\xi\in \RR}\widehat{\bold{1}_{P}}(\xi)\widehat{a}(\xi)\Bigr\vert
      			\leq C\Bigl(\sum_{\xi\in \RR}\vert \widehat{a}(\xi)\vert^{4}\Bigr)^{1/4}
      			\leq C\Vert a\Vert_{U^{2}(\RR)}.
      		\end{split}
      	\end{equation}
      \end{proof}

      The following inverse theorem can be deduced from Theorem 11 of ~\cite{Sz} (or from \cite{web2}) and Lemma A.4 of \cite{FH}:
      
      \begin{thm}[The inverse theorem for $\mathbb{Z}^{D}$ actions]\label{inv2} For every $\e>0$, $d\geq 2$ and $D\in\N_{+}$, there exist $\d:=\d(d,D,\e)>0, N_{0}:=N_{0}(d,D,\e)\in\mathbb{N}$ and a nilmanifold $X:=X(d,D,\e)$ with a nil-structure $\X:=\X(d,D,\e)=(G_{\bullet},\mathcal{X},\psi,d_{G},d_{X})$  such that for every $N\geq N_{0}$ and every $f\colon \RR\rightarrow\mathbb{C}$ with
      	$\vert f\vert\leq 1$, if $$\Vert f\Vert_{U^{d}(\RR)}\geq\e,$$ 
      	then there exist a function $\Phi\colon X\rightarrow\mathbb{C}$ with $\Vert\Phi\Vert_{\Lip(\X)}\leq 1$ and a polynomial sequence $g\in\poly_{D}(G_{\bullet})$ such that $$\Bigl\vert\mathbb{E}_{\n\in[N]^{D}}f(\n)\cdot\Phi(g(\n)\cdot e_{X})\Bigr\vert\geq\d,$$
      	where we regard $f$ as a function from $\Z^{D}$ to $\mathbb{C}$ supported on $[N]^{D}$ in the obvious way.
      \end{thm}

      We are now ready to prove Theorem \ref{ap}.
      
      \begin{proof}[Proof of Theorem \ref{ap}]
      	Suppose first that $\lim_{N\to\infty}\Vert\chi\circ\iota_{\B}\Vert_{U^{d}([N]^{D})}\neq 0$ for some $d\geq 2$. Then by definition, there exist $\e>0$ and an infinite set $J\subseteq \N$ such that $\Vert\bold{1}_{[N]^{D}}\cdot\chi\circ\iota_{\B}\Vert_{U^{d}(\Z^{D}_{3N})}>\e$ for all $N\in J$.  By Theorems \ref{inv2}, 
      	there exist $\e'>0, N_{0}\in\mathbb{N}$ and a nilmanifold $X=G/\Gamma$ such that for every $N\geq N_{0}, N\in J$, there exist a function $\Phi\colon X\rightarrow\mathbb{C}$ with $\Vert\Phi\Vert_{\Lip(\X)}\leq 1$ and a polynomial sequence $g\in\poly_{D}(G)$ such that $$\Bigl\vert\mathbb{E}_{\n\in[3N]^{D}}\bold{1}_{[N]^{D}}(\n)\chi\circ\iota_{\B}(\n)\cdot\Phi(g(\n)\cdot e_{X})\Bigr\vert\geq\e'.$$
      	By  Theorem \ref{s3},  there exist $\delta>0, N_{1}\geq N_{0}$ and a  $D$-dimensional arithmetic progression $P'$ such that for all $N\geq N_{1}, N\in J$,
      	\begin{equation}\nonumber
      		\Bigl\vert\E_{\n\in R_{3N,D}}\bold{1}_{P'}(\n)\chi\circ\iota_{\B}(\n)\Bigr\vert\geq\delta.
      	\end{equation}
      	By definition, $\chi$ is not aperiodic.
      	
      	Conversely, suppose that  $\lim_{N\to\infty}\Vert\chi\circ\iota_{\B}\Vert_{U^{2}([N]^{D})}=0$. Let $N^{\ast}$ denote the smallest prime number greater than $2N$. Then $N^{\ast}\leq 4N$. Similar to Lemma A.3 of \cite{FH}, $\liminf_{N\to\infty}\Vert\bold{1}_{[N]^{D}}\Vert_{U^{2}(\Z^{D}_{N^{\ast}})}$ is bounded below by a positive constant depending only on $D$. This implies that
      	$$\lim_{N\to\infty}\Vert\bold{1}_{[N]^{D}}\cdot\chi\circ\iota_{\B}\Vert_{U^{2}(\Z^{D}_{N^{\ast}})}=0.$$
      	
      	By Lemma \ref{1p}, 
      	\begin{equation}\nonumber
      		\begin{split}
      		&\quad	\limsup_{N\to\infty}\sup_{P}\Bigl\vert\E_{\n\in R_{N,D}}\bold{1}_{P}(\n)\chi\circ\iota_{\B}(\n)\Bigr\vert
      		\leq 	4^{D}\limsup_{N\to\infty}\sup_{P}\Bigl\vert\E_{\n\in R_{N^{\ast},D}}\bold{1}_{P}(\n)\bold{1}_{[N]^{D}}(\n)\chi\circ\iota_{\B}(\n)\Bigr\vert
      		\\&\leq 4^{D}\limsup_{N\to\infty}\Vert\bold{1}_{[N]^{D}}\cdot\chi\circ\iota_{\B}\Vert_{U^{2}(\Z^{D}_{N^{\ast}})}=0,
      		\end{split}	
      	\end{equation}	
      which implies that   $\chi$ is aperiodic.
      \end{proof}	
     
\section{Structure theorem for multiplicative functions}\label{s:s}

We prove Theorem \ref{U3} and its stronger form Theorem \ref{nU3s} in this section. 
The approach we use is similar to the ones used in \cite{FH,S}. 

\subsection{Strong $U^{d}$ structure theorem}


\begin{defn}[Kernel]
	A function $\phi\colon \RR\rightarrow\mathbb{C}$ is a \emph{kernel} of $\RR$ if it is non-negative and $\mathbb{E}_{\A\in\RR}\phi(\A)=1$. The set $\{\xi\in \RR\colon \widehat{\phi}(\xi)\neq 0\}$ is called the \emph{spectrum} of $\phi$.
\end{defn}

In order to show Theorem \ref{U3}, it suffices to show the following stronger theorem, which generalizes the main structure theorems in \cite{FH,S}:

\begin{thm}[Strong $U^{d}$ structure theorem for multiplicative functions]\label{nU3s}
	Let $\Omega\in\N$. For $N\in\N$, let $\tilde{N}$ denote the smallest prime integer greater than $\Omega N$.
	Let $\K$ be an integral tuple and
	$\nu$ be a probability measure  on the group $\mathcal{M}_{K}$.
	Let $F\colon\mathbb{N}\times\mathbb{N}\times\mathbb{R}^{+}\rightarrow\mathbb{R}^{+}$ be a function. For every $\e>0$ and $d\geq 2$, there exist $Q=Q(D,d,F,\e,\Omega),$ $R=R(D,d,F,\e,\Omega),$ $N_{0}=N_{0}(D,d,F,\e,\Omega)\in\N_{+}$ such that for every
	$N\geq N_{0}$ and $\chi\in\mathcal{M}_{K}$, the truncated function $\chi_{N}\colon \Rr\to\mathbb{C}$ can be written as
	\begin{equation}\nonumber
		\begin{split}
			\chi_{N}(\A)=\chi_{N,s}(\A)+\chi_{N,u}(\A)+\chi_{N,e}(\A)
		\end{split}
	\end{equation}
	for all $\A\in\Rr$ such that the following holds:
	\begin{enumerate}[(i)]
		\item   $\vert\chi_{N,s}\vert\leq 1$, $\chi_{N,s}=\chi_{N}*\phi_{N,1}$ and $\chi_{N,s}+\chi_{N,e}=\chi_{N}*\phi_{N,2}$, where $\phi_{N,1}$ and $\phi_{N,2}$ are kernels of $\Rr$ that are independent of $\chi$, and the convolution product is defined on $\Rr$;
		\item $\vert\chi_{N,s}(\A+Q\bold{e}_{i})-\chi_{N,s}(\A)\vert\leq\frac{R}{N}$ for every $\A\in \Rr$ and $1\leq i\leq D$;
		\item For every $\xi=(\xi_{1},\dots,\xi_{D})\in\Rr$ such that $\widehat{f}_{N,s}(\xi)\neq 0$ and every $1\leq i\leq D$, there exists $p_{i}\in\{0,\dots,Q-1\}$ such that $\vert\frac{\xi_{i}}{\tilde{N}}-\frac{p_{i}}{Q}\vert\leq \frac{R}{\tilde{N}}$;
		\item $\Vert\chi_{N,u}\Vert_{U^{d}(\Rr)}\leq\frac{1}{F(Q,R,\e)}$;
		\item $\mathbb{E}_{\A\in \Rr}\int_{\mathcal{M}_{K}}\vert\chi_{N,e}(\A)\vert d\nu(\chi)\leq\e$.
	\end{enumerate}	
\end{thm}

\subsection{Weak $U^{2}$ structure theorem}\label{u2}
Our first step is to prove a weak $U^{2}$ structure theorem:

\begin{thm}[Weak $U^{2}$ structure theorem for multiplicative functions]\label{nU2} 
		Let $\Omega\in\N$. For $N\in\N$, let $\tilde{N}$ denote the smallest prime integer greater than $\Omega N$.
	Let $\K$ be an integral tuple. 
	For every $\e>0$, there exist $Q:=Q(\e,\B,\Omega),R:=R(\e,\B,\Omega),N_{0}:=N_{0}(\e,\B,\Omega)\in\N_{+}$ such that for every $N\geq N_{0}$ and $\chi\in\mathcal{M}_{K}$, the truncated function $\chi_{N}\colon \Rr\to\mathbb{C}$ can be written as
	\begin{equation}\nonumber
		\begin{split}
			\chi_{N}(\A)=\chi_{N,s}(\A)+\chi_{N,u}(\A)
		\end{split}
	\end{equation}
	for all $\A\in\Rr$ such that the following holds:
	\begin{enumerate}[(i)]
		\item\label{w21}  $\vert\chi_{N,s}\vert\leq 1$ and $\chi_{N,s}=\chi_{N}*\phi_{N,\epsilon}$ for some kernel $\phi_{N,\epsilon}$ of $\Rr$ which  is independent of $\chi$, where the convolution product is defined on $\Rr$;
		\item\label{w22} $\vert\chi_{N,s}(\A+Q\bold{e}_{i})-\chi_{N,s}(\A)\vert\leq\frac{R}{N}$ for every $\A\in \Rr$ and $1\leq i\leq D$;
		\item\label{w23} For every $\xi=(\xi_{1},\dots,\xi_{D})\in\Rr$ such that $\widehat{f}_{N,s}(\xi)\neq 0$ and every $1\leq i\leq D$, there exists $p_{i}\in\{0,\dots,Q-1\}$ such that $\vert\frac{\xi_{i}}{\tilde{N}}-\frac{p_{i}}{Q}\vert\leq \frac{R}{\tilde{N}}$;
		\item\label{w24} $\Vert\chi_{N,u}\Vert_{U^{2}(\Rr)}\leq\epsilon$;
		\item\label{w25} For every $0<\e'\leq\e$, $N\geq \max\{N_{0}(\e,\B),N_{0}(\e',\B)\}$ and $\xi\in \Rr$, we have that $$\widehat{\phi_{N,\e'}}(\xi)\geq\widehat{\phi_{N,\e}}(\xi)\geq 0.$$
	\end{enumerate}	
\end{thm}

In the rest of Section \ref{u2}, we consider $\Omega$ as fixed, and all the quantities depend implicitly on $\Omega$. Moreover,  $\tilde{N}$ always denotes the smallest prime integer greater than $\Omega N$.

We first explain what happens when the Fourier coefficient of $\chi$ is away from 0. 

\begin{cor}[A consequence of Theorem \ref{key}]\label{nF2}	
	Let $\K$ be an integral tuple.
	For every $\e>0$, there exist $Q:=Q(\e,\B),V:=V(\e,\B), N_{0}:=N_{0}(\e,\B)\in\mathbb{N}_{+}$ such that for every $N\geq N_{0}$, every $\chi\in\mathcal{M}_{K}$ and every $\xi=(\xi_{1},\dots,\xi_{D})\in \Rr$, if $\vert\widehat{\chi_{N}}(\xi)\vert\geq\e$, then $$\sum_{i=1}^{D}\Bigl\Vert \frac{Q\xi_{i}}{\tilde{N}}\Bigr\Vert_{\T}\leq\frac{QV}{\tilde{N}}.$$
\end{cor}
One way to prove Corollary \ref{nF2} is to follow the method used in Corollary 5.2 of \cite{S}. Here we provide a different proof by using Theorem \ref{key} as a black box:

\begin{proof}
	Let $G=\R$, $\Gamma=\Z$ and $X=G/\Gamma=\T$.
	Let $g_{N,\xi}\in\poly_{D}(\R)$ be the function given by $g_{N,\xi}(\n):=-\n\circ_{\tilde{N}}\xi$ for all $\n\in\Z^{D}$. 
	Since $\vert R_{N,D}\vert\leq (3N)^{D}$, $\vert\widehat{\chi_{N}}(\xi)\vert\geq\e$ implies that 
	\begin{equation}\nonumber
		\begin{split}
			\Bigl\vert\E_{\n\in R_{N,D}}\bold{1}_{[N]^{D}}(\n)\chi(\iota_{\B}( \n))e(g_{N,\xi}(\n)\cdot e_{X})\Bigr\vert\geq 3^{D}\epsilon.
		\end{split}
	\end{equation}
	By Theorem \ref{key}, there exist $\delta:=\delta(\e,\B)>0$ and $N_{0}:=N_{0}(\e,\B)\in\N$ such that if $N\geq N_{0}$, then
	$(g_{N,\xi}(\n)\cdot e_{X})_{\n\in R_{N,D}}$ is not totally $\delta$-equidistributed.
	Since every horizontal character $\eta\colon\T\to\T$ on $\T$ can be written as $\eta(t)=qt \mod \Z$ for some $q\in\Z$, 
	by Theorem \ref{Lei}, there exist $V:=V(\delta,D)=V(\e,\B)>0$ and $0<q_{\xi}\leq V$ such that  $$\Vert\eta_{q_{\xi}}\circ g_{N,\xi}\Vert_{C^{\infty}_{1}[N]}=\sum_{i=1}^{D}\Bigl\Vert \frac{q_{\xi}\cdot\xi_{i}}{\tilde{N}}\Bigr\Vert_{\T}\leq\frac{V}{\tilde{N}},$$
	where $\eta_{q_{\xi}}(t):=q_{\xi}t \mod \Z$ for all $t\in\T$. Let $Q:=V!$. Then $Q$ depends only on $\e$ and $\B$, and for all $\xi\in\Rr$, 
	$$\sum_{i=1}^{D}\Bigl\Vert \frac{Q\cdot\xi_{i}}{\tilde{N}}\Bigr\Vert_{\T}=\sum_{i=1}^{D}\Bigl\Vert \frac{Q}{q_{\xi}}\cdot\frac{q_{\xi}\cdot\xi_{i}}{\tilde{N}}\Bigr\Vert_{\T}\leq \frac{Q}{q_{\xi}}\sum_{i=1}^{D}\Bigl\Vert\frac{q_{\xi}\cdot\xi_{i}}{\tilde{N}}\Bigr\Vert_{\T}\leq \frac{QV}{q_{\xi}\tilde{N}}\leq \frac{QV}{\tilde{N}}.$$
	This finishes the proof.
\end{proof}

Let the integral tuple $\K$ be fixed, and we  assume that all the quantities depend implicitly on $\B$ in the rest of this section.
For $\e>0, N\in\N, q\in\Rr$, define

\begin{equation}\label{79}
	\begin{split}
		&\mathcal{A}(N,\e):=\Bigl\{\xi\in \Rr\colon \sup_{\chi\in\mathcal{M}_{K}}\vert\widehat{\chi_{N}}(\xi)\vert\geq\e^{2}\Bigr\};
		\\& W(N,q,\e):=\max_{\xi=(\xi_{1},\dots,\xi_{D})\in\mathcal{A}(N,\e)}\sum_{i=1}^{D}\tilde{N}\Bigl\Vert \frac{q\xi_{i}}{\tilde{N}}\Bigr\Vert_{\T};
		\\& Q(\e):=\min_{k\in\mathbb{N}}\Bigl\{k!\colon\limsup_{N\rightarrow\infty}W(N,k!,\e)<\infty\Bigr\};
		\\& V(\e):=1+\Bigl\lfloor \frac{1}{Q(\e)}\limsup_{N\rightarrow\infty}W(N,Q(\e),\e)\Bigr\rfloor.
	\end{split}
\end{equation}

It follows from Corollary \ref{nF2} that $Q(\e)$ is well defined. Notice that for all $0<\e'\leq\e, Q(\e')\geq Q(\e)$ and $Q(\e')$ is a multiple of $Q(\e)$. Thus $V(\e')\geq V(\e)$ (it is easy to verify that $V(\e)$ increases as $\e$ decreases). 
By definition, there exists $N_{1}(\e)\in\N$ such that for all $N\geq N_{1}(\e)$, $\chi\in\mathcal{M}_{K}$ and $\xi\in\Rr$,   $$\vert\widehat{\chi_{N}}(\xi)\vert\geq\e^{2}\Rightarrow\sum_{i=1}^{D}\Bigl\Vert \frac{Q(\e)\xi_{i}}{\tilde{N}}\Bigr\Vert_{\T}\leq\frac{Q(\e)V(\e)}{\tilde{N}}.$$

For every $m\geq 1, N>2m$, we define the function $f_{N,m}\colon \Rr\rightarrow\mathbb{C}$ by

\begin{equation}\nonumber
	\begin{split}
		f_{N,m}(\n):=\sum_{-m\leq\xi_{1},\dots,\xi_{D}\leq m}\Bigl(\prod_{i=1}^{D}(1-\frac{\vert\xi_{i}\vert}{m})\Bigr)\cdot e(\n\circ_{\tilde{N}}(\xi_{1},\dots,\xi_{D})).
	\end{split}
\end{equation}
It is easy to verify that $f_{N,m}$ is a kernel of $\Rr$ whose spectrum is $\{-(m-1),\dots, m-1\}^{D}$. Let $Q_{\tilde{N}}(\e)^{*}$ be the unique integer in $\{1,\dots,\tilde{N}-1\}$ such that $Q(\e)Q_{\tilde{N}}(\e)^{*}\equiv 1\mod \tilde{N}$. 
Let 
\begin{equation}\label{76}
	\begin{split}
		N_{0}(\e):=\max\{N_{1}(\e),2DQ(\e)V(\e)\lceil\e^{-4}\rceil\}.
	\end{split}
\end{equation}
For $N>N_{0}$, we define $\phi_{N,\e}\colon \Rr\rightarrow\mathbb{C}$ by
\begin{equation}\label{phi}
	\begin{split}
		\phi_{N,\e}(\xi):=f_{N,DQ(\e)V(\e)\lceil\e^{-4}\rceil}(Q_{\tilde{N}}(\e)^{*}\xi).
	\end{split}
\end{equation}
In other words, $f_{N,DQ(\e)V(\e)\lceil\e^{-4}\rceil}(\xi)=\phi_{N,\e}(Q(\e)\xi)$. Then $\phi_{N,\e}$ is also a kernel of $\Rr$, and the spectrum of $\phi_{N,\e}$ is the set
\begin{equation}\nonumber
	\begin{split}
		\Xi_{N,\e}:=\Bigl\{\xi=(\xi_{1},\dots,\xi_{D})\in \Rr\colon\Bigl\Vert\frac{Q(\e)\xi_{i}}{\tilde{N}}\Bigr\Vert_{\T}<\frac{DQ(\e)V(\e)\lceil\e^{-4}\rceil}{\tilde{N}}, 1\leq i\leq D\Bigr\}.
	\end{split}
\end{equation}
Moreover, $$\widehat{\phi_{N,\e}}(\xi)=\prod_{i=1}^{D}\Bigl(1-\Bigl\Vert\frac{Q(\e)\xi_{i}}{\tilde{N}}\Bigr\Vert_{\T}\cdot\frac{\tilde{N}}{DQ(\e)V(\e)\lceil\e^{-4}\rceil}\Bigr)
$$ if $\xi\in\Xi_{N,\e}$ and $\widehat{\phi_{N,\e}}(\xi)=0$ otherwise.

\begin{proof}[Proof of Theorem \ref{nU2}]
	In this proof, we assume implicitly  that every constant depends on $\B$ and $\Omega$.
	Let the notations be defined as above. 	Fix $\e>0$, and let $Q(\e)$ and $N_{0}(\e)$ be defined as in (\ref{79}) and (\ref{76}), respectively. Let $R(\e)$ be sufficiently large to be chosen later.
	For all $\chi\in\mathcal{M}_{K}$, let $\chi_{N,s}:=\chi_{N}\ast\phi_{N,\e}$ and $\chi_{N,u}=\chi_{N}-\chi_{N,s}$,
	where $\phi_{N,\e}$ is defined in (\ref{phi}). We show that $\chi_{N,s}$ and $\chi_{N,u}$ satisfy all the requirements. 
	
	We now fix $\chi\in\mathcal{M}_{K}$ and $N\geq N_{0}(\e)$.
	Since $\vert\chi\vert\leq 1$, by definition, $\vert\chi_{N,s}\vert\leq 1$. So
 Property (i) holds.

	Using Fourier inversion formula and the estimate $\vert e(x)-1\vert\leq 2\pi\Vert x\Vert_{\T}$, for all $1\leq i\leq D$, we have that
	\begin{equation}\nonumber
		\begin{split}
			\vert\chi_{N,s}(\n+Q(\e)\bold{e}_{i})-\chi_{N,s}(\n)\vert\leq\sum_{\xi=(\xi_{1},\dots,\xi_{D})\in \Rr}\vert\widehat{\phi_{N,\e}}(\xi)\vert\cdot 2\pi\Bigl\Vert\frac{Q(\e)\xi_{i}}{\tilde{N}}\Bigr\Vert_{\T}
			\leq\vert\Xi_{N,\e}\vert\cdot\frac{2\pi D Q(\e)V(\e)\lceil\e^{-4}\rceil}{\tilde{N}}.
		\end{split}
	\end{equation}
	Since $\vert\Xi_{N,\e}\vert$ is finite and depends only on $\e$, Property (ii) follows by taking $R(\e)$ sufficiently large depending only on $\e$, $\B$ and $\Omega$.

	Let $\xi\in\Rr$ be such that $\widehat{\chi}_{N,s}(\xi)\neq 0$. Then $\widehat{\phi}_{N,s}(\xi)\neq 0$ and so $\xi\in\Xi_{N,\e}$. By definition, there exist $p_{1},\dots,p_{D}\in\{0,\dots,Q(\e)-1\}$ such that $\vert\frac{\xi_{i}}{\tilde{N}}-\frac{p_{i}}{Q(\e)}\vert\leq \frac{DV(\e)\lceil\e^{-4}\rceil}{\tilde{N}}$. So Property (iii) holds by taking $R(\e)$ sufficiently large depending only on $\e$ (and $\B$).

	For every $\chi\in\mathcal{M}_{K}$ and $\xi=(\xi_{1},\dots,\xi_{D})\in \Rr$, if $\vert\widehat{\chi_{N}}(\x)\vert\geq\e^{2}$, then by the definition of $Q(\e)$ and $N_{1}(\e)$, $$\Bigl\Vert\frac{Q(\e)\xi_{i}}{\tilde{N}}\Bigr\Vert_{\T}\leq \frac{Q(\e)V(\e)}{\tilde{N}}$$ for all $1\leq i\leq D$. Then $\widehat{\phi_{N,\e}}(\xi)\geq(1-\e^{4}/D)^{D}\geq 1-\e^{4}$. So
	\begin{equation}\label{78}
		\Bigl\vert\widehat{\chi_{N}}(\xi)-\widehat{\phi_{N,\e}*\chi_{N}}(\xi)\Bigr\vert=\Bigl\vert\widehat{\chi_{N}}(\xi)(1-\widehat{\phi_{N,\e}}(\xi))\Bigr\vert\leq\e^{4}\leq\e^{2}.
	\end{equation}	
	Note that (\ref{78}) also holds if $\vert\widehat{\chi_{N}}(\xi)\vert\leq\e^{2}$. Thus by identity (\ref{F}) and Parseval's identity, we have
	\begin{equation}\nonumber
		\begin{split}
			\Vert\chi_{N,u}\Vert_{U^{2}(\Rr)}^{4}=\sum_{\xi\in \Rr}\Bigl\vert\widehat{\chi_{N}}(\xi)-\widehat{\phi_{N,\e}*\chi_{N}}(\xi)\Bigr\vert^{4}\leq\e^{4}\sum_{\xi\in \Rr}\Bigl\vert\widehat{\chi_{N}}(\xi)-\widehat{\phi_{N,\e}*\chi_{N}}(\xi)\Bigr\vert^{2}
			\leq\sum_{\xi\in \Rr}\vert\widehat{\chi_{N}}(\xi)\vert^{2}\leq\e^{4}.
		\end{split}
	\end{equation}
	 This proves Property (iv).

	Suppose that $0<\e'\leq\e$. Since $Q(\e')\geq Q(\e), V(\e')\geq V(\e)$ and $Q(\e')$ is a multiple of $Q(\e)$, we have $\Xi_{N,\e}\subseteq\Xi_{N,\e'}$ and $\widehat{\phi_{N,\e'}}(\xi)\geq\widehat{\phi_{N,\e}}(\xi)$ for every $\xi\in \Rr$.
	This proves Property (v), which finishes the proof of the whole theorem.
\end{proof}

\subsection{Weak $U^{d}$ structure theorem}\label{u3}
 Our section step is to prove a weak  $U^{d}$ structure theorem:

\begin{thm}[Weak $U^{d}$ structure theorem for multiplicative functions]\label{nU3w}
	Let $\Omega\in\N$. For $N\in\N$, let $\tilde{N}$ denote the smallest prime integer greater than $\Omega N$.	
	Let $\K$ be an integral tuple.
	For every $d\in\N, d\geq 3$ and $\e>0$, there exists $\theta_{0}:=\theta_{0}(d,\epsilon,\B,\Omega)>0$ such that for all $0<\theta<\theta_{0}$, there exist $Q:=Q(d,\e,\theta,\B,\Omega),R:=R(d,\e,\theta,\B,\Omega),N_{0}:=N_{0}(d,\e,\theta,\B,\Omega)\in\N_{+}$ such that for every $N\geq N_{0}$ and $\chi\in\mathcal{M}_{K}$, the truncated function $\chi_{N}\colon \Rr\to\mathbb{C}$ can be written as
	\begin{equation}\nonumber
		\begin{split}
			\chi_{N}(\A)=\chi_{N,s}(\A)+\chi_{N,u}(\A)
		\end{split}
	\end{equation}
	for all $\A\in\Rr$ such that the following holds:
	\begin{enumerate}[(i)]
		\item\label{w31} $\vert\chi_{N,s}\vert\leq 1$ and $\chi_{N,s}=\chi_{N}*\phi_{N,\theta}$, where $\phi_{N,\theta}$ is the kernel of $\Rr$ defined in (\ref{phi}) which is independent of $\chi$, and the convolution product is defined on $\Rr$;
		\item\label{w32} $\vert\chi_{N,s}(\A+Q\bold{e}_{i})-\chi_{N,s}(\A)\vert\leq\frac{R}{N}$ for every $\A\in \Rr$ and $1\leq i\leq D$;
		\item\label{w33} For every $\xi=(\xi_{1},\dots,\xi_{D})\in\Rr$ such that $\widehat{f}_{N,s}(\xi)\neq 0$ and every $1\leq i\leq D$, there exists $p_{i}\in\{0,\dots,Q-1\}$ such that $\vert\frac{\xi_{i}}{\tilde{N}}-\frac{p_{i}}{Q}\vert\leq \frac{R}{\tilde{N}}$;
		\item\label{w34} $\Vert\chi_{N,u}\Vert_{U^{d}(\Rr)}\leq\epsilon$.
	\end{enumerate}
\end{thm}

The proof of Theorem \ref{nU3w} is similar to Theorem 8.1 of \cite{FH} and Theorem 8.2 of \cite{S}. We provide the details for completeness.
Again, in the rest of Section \ref{u3}, we consider $\Omega$ as fixed, and all the quantities depend implicitly on $\Omega$. Moreover,  $\tilde{N}$ always denotes the smallest prime integer greater than $\Omega N$.

Let $\bold{K},d,\epsilon$ be as in the statement of Theorem \ref{nU3w} and let $\phi_{N,\e}$ be defined as in (\ref{phi}). For all $\e>0$, $N\in\N$ and $\chi\in\mathcal{M}_{K}$, denote 
\begin{equation}\label{75}
	\chi_{N,s,\e}:=\chi_{N}*\phi_{N,\epsilon} \text{ and } \chi_{N,u,\e}:=\chi_{N}-\chi_{N,s,\e}.
\end{equation}	
By Theorem \ref{nU2}, 
for all $\theta>0$, there exist $Q(\delta,\B),R(\delta,\B),N_{0}(\delta,\B)\in\N_{+}$ such that for every $N\geq N_{0}(\delta,\B)$ and $\chi\in\mathcal{M}_{K}$, Properties (i)--(v) of Theorem \ref{nU2} holds with $\e$ replaced with $\theta$.

Comparing Theorem \ref{nU3w} with  Theorem \ref{nU2}, it is easy to see that we only need to show that for every $\epsilon>0$ and $d\geq 3$, there exists $\theta_{0}:=\theta_{0}(d,\e,\B)>0$ such that for all $0<\theta<\theta_{0}$, there exists $N_{0}:=N_{0}(d,\e,\theta,\B)\in\N$ such that for all $N\geq N_{0}(d,\e,\theta,\B)$ and $\chi\in\mathcal{M}_{K}$, we have that $$\Vert\chi_{N,u,\theta}\Vert_{U^{d}(\Rr)}\leq \e.$$

By Properties (iv) of Theorem \ref{nU2}, $\Vert\chi_{N,u,\theta}\Vert_{U^{2}(\Rr)}\leq \theta$ for all $\theta>0$. Our strategy is to show that for multiplicative functions, the smallness of the $U^{2}$ norm implies the  smallness of the $U^{d}$ norm for $d\geq 3$.
By Theorem \ref{inv2} (the inverse theorem for Gowers norms), in order to prove 
Theorem \ref{nU3w}, it suffices to show the following:

\begin{prop}\label{fin}
	Let $\K$ be an integral tuple, $\d>0$ and $d\in\N_{+}$. Let $X=G/\Gamma$ be nilmanifold  with a nil-structure $\G$. There exists $\theta_{0}:=\theta_{0}(\X,d,\delta,\B)>0$ such that for all $0<\theta<\theta_{0}$, there exists  $N_{0}:=N_{0}(\X,d,\delta,\theta,\B)\in\N$ such that for every $N\geq N_{0}$, every $\chi\in\mathcal{M}_{K}$, every  $g\in\poly_{D}(G_{\bullet})$ and every $\Phi\colon X\to\mathbb{C}$ with modulus at most 1 and $\Vert\Phi\Vert_{\Lip(\X)}\leq 1$, we have that 
	\begin{equation}\label{e81}
		\Bigl\vert\E_{\n\in[\tilde{N}]^{D}}\chi_{N,u,\theta}(\n)\cdot\Phi(g(\n)\cdot e_{X})\Bigr\vert\leq \delta,
	\end{equation}	
	where $\chi_{N,u,\theta}$ is defined in (\ref{75}).
\end{prop}

%
\begin{proof}
We may assume without loss of generality that $\X$ is standard.
To simplify the notations, in the proof, $\d$, $d$, $\X$, $\B$ and $\Omega$ are fixed and all
the quantities depend implicitly on them.

Suppose on the contrary that there exist arbitrarily small $\theta>0$, arbitrarily large $N\in\N$,  function $\chi\in\mathcal{M}_{K}$, polynomial sequence  $g\in\poly_{D}(G_{\bullet})$, and $\Phi\colon X\to\mathbb{C}$ with modulus at most 1 and $\Vert\Phi\Vert_{\Lip(\X)}\leq 1$, such that 
\begin{equation}\label{e89}
	\Bigl\vert\E_{\n\in[\tilde{N}]^{D}}\chi_{N,u,\theta}(\n)\cdot\Phi(g(\n)\cdot e_{X})\Bigr\vert>\delta,
\end{equation}

	By using the factorization theorem (Theorem \ref{5.6}),
	we may deduce from (\ref{e89}) that there exist
	\begin{itemize}
		\item A function $\lambda:=\lambda_{\delta,\X,\B,\Omega}\colon\N_{+}\to\R_{+}$;
		\item For all $M\in\N$ a finite subset $\Sigma(M)\subseteq G$ and a finite collection $\mathcal{F}(M)$ of sub nilmanifolds  $X'=G'/\Gamma'$ of $X$ with nil-structures $\X'$ induced by $\X$;
		\item For each $a\in\Sigma(M)$, a sub nilmanifold $X'_{a}=G'_{a}/\Gamma'_{a}$  of $X$ with a nil-structure $\X'_{a}$ adapted to the filtration $G'_{a,\bullet}$ induced by the $a$-conjugate of $\X'$, a polynomial sequence $g'_{a}\in\poly_{D}(G'_{a,\bullet})$  and a function $\Phi'_{a}\colon X'_{a}\to \mathbb{C}$ with $\vert\Phi'_{a}\vert\leq 1$ and $\Vert\Phi'_{a}\Vert_{\Lip(\X'_{a})}=1$,
	\end{itemize}		
	such that for every function 
	$\zeta\colon\N\to\R_{+}$, there exist a function $N_{0}\colon\N\to\mathbb{R}_{+}$ and $M_{1}\in\N_{+}$ such that
	the following holds:

	\textbf{Property 1.}
	There exist $M^{\ast}\in\N$ with $M^{\ast}\leq M_{1}$,
	such that if (\ref{s10}) holds for some $N\geq N_{0}(M^{\ast})$, then exist
	a $D$-dimensional arithmetic progression $P$ of $\Z^{D}$ 
	an element $a\in\Sigma(M^{\ast})$ such that
	\begin{equation}\nonumber
		\Bigl\vert\E_{\n\in[\tilde{N}]^{D}}\bold{1}_{P}(\n)\chi_{N,u,\theta}(\n)\cdot\Phi'_{a}(g'_{a}(\n)\cdot e_{X'_{a}})\Bigr\vert>3\cdot 2^{D}C\lambda(M^{\ast}),
	\end{equation}
	where $C:=C(D)>0$ is the constant defined in Lemma \ref{1p}.
	
	\textbf{Property 2.} For all $M^{\ast}\in\N$, $a\in\Sigma(M^{\ast})$ and $N\geq N_{0}(M^{\ast})$,
	\begin{equation}\nonumber
		\text{$(g'_{a}(\n)\cdot e_{X'_{a}})_{\n\in R_{N,D}}$ is totally $\zeta(M^{\ast})$-equidistributed on $X'_{a}$ with respect to $\X'_{a}$.}
	\end{equation}

We left the choice of $\zeta$ to the end of the proof (note that the choice of $\zeta$ must be dependent only on $\delta,d,\X,\B$).	
The method we deduce Properties 1 and 2 from (\ref{e89}) is again similar to the one we used to deduce (\ref{7.3}) and (\ref{7.90}) from (\ref{7.5}), 
as well as the one to deduce Properties 1 and 2 from (\ref{s10}) in the proof of Theorem \ref{s3}. 
So we omit it.

Set $\theta_{0}:=\min_{M^{\ast}\leq M_{1}}\lambda(M^{\ast})$. We may assume that (\ref{e89}) holds for some $\theta<\theta_{0}$ and $N>\max_{M^{\ast}\leq M_{1}}N_{0}(M^{\ast})$.
Let $z:=\int_{X'_{a}}\Phi'_{a}\,dm_{X'_{a}}$ and $\Phi'_{0,a}:=\Phi'_{a}-z$.   Then $\int_{X'_{a}}\Phi'_{0,a}\,dm_{X'_{a}}=0$. Applying Lemma \ref{1p}, Property (iv) of Theorem \ref{nU2}, and the definition of $\theta_{0}$ consecutively, we have that 
\begin{equation}\nonumber
	\Bigl\vert\E_{\n\in[\tilde{N}]^{D}}\bold{1}_{P}(\n)\chi_{N,u,\theta}(\n)\cdot z\Bigr\vert\leq  C\Vert \chi_{N,u,\theta}\Vert_{U^{2}(\Rr)}\leq C\theta_{0}\leq C\lambda(M^{\ast}).
\end{equation}
By Property 1, 
\begin{equation}\label{e813}
	\Bigl\vert\E_{\n\in[\tilde{N}]^{D}}\bold{1}_{P}(\n)\chi_{N,u,\theta}(\n)\cdot\Phi'_{0,a}(g'_{a}(\n)\cdot e_{X'_{a}})\Bigr\vert>2\cdot 2^{D}C\lambda(M^{\ast}).
\end{equation}

By (\ref{75}), we may write $\chi_{N,u,\theta}$ as $\chi_{N,u,\theta}=\chi_{N}\ast\psi_{N,\epsilon}$, where $\psi_{N,\epsilon}\colon\Rr\to\mathbb{R}$ is the function given by $\psi_{N,\epsilon}(\bold{0}):=\tilde{N}^{D}-\phi_{N,\e}(\bold{0})$ and $\psi_{N,\epsilon}(\n):=-\phi_{N,\e}(\n)$ for all $\n\in\Rr\backslash\{\bold{0}\}$. Since $\E_{\n\in\Rr}\phi_{N,\e}=1$, we have that $\E_{\n\in\Rr}\vert\psi_{N,\e}\vert\leq 2$. Therefore, by (\ref{e813}), there exists $\m=(m_{1},\dots,m_{D})\in\Rr$ such that 
\begin{equation}\nonumber
	\Bigl\vert\E_{\n\in[\tilde{N}]^{D}}\bold{1}_{P}(\n+\m \mod [\tilde{N}]^{D})\chi_{N}(\n)\cdot\Phi'_{0,a}(g'_{a}(\n+\m \mod [\tilde{N}]^{D})\cdot e_{X'_{a}})\Bigr\vert>2^{D}C\lambda(M^{\ast}),
\end{equation}
where the residue class $\n+\m \mod [\tilde{N}]^{D}$ is taken in $[\tilde{N}]^{D}=\{1,\dots,\tilde{N}\}^{D}$ instead of the more conventional $\{0,\dots,\tilde{N}-1\}^{D}$. Therefore, there exist $J=J_{1}\times\dots\times J_{D}\subseteq [\tilde{N}]^{D}$ and $\m'=(m_{1},\dots,m_{D})\in [\tilde{N}]^{D}$, such that for all $1\leq i\leq D$, either $J_{i}=\{1,\dots,\tilde{N}-m_{i}\}$ and $m'_{i}=m_{i}$, or $J_{i}=\{\tilde{N}-m_{i}+1,\dots,\tilde{N}\}$ and $m'_{i}=m_{i}-\tilde{N}$, and that

\begin{equation}\label{e815}
	\Bigl\vert\E_{\n\in[\tilde{N}]^{D}}\bold{1}_{P}(\n+\m')\bold{1}_{J}(\n)\bold{1}_{[N]^{D}}(\n)\chi(\n)\cdot\Phi'_{0,a}(g'_{a}(\n+\m')\cdot e_{X'_{a}})\Bigr\vert>C\lambda(M^{\ast}).
\end{equation}
Note that $\bold{1}_{P}(\n+\m')\bold{1}_{J}(\n)\bold{1}_{[N]^{D}}(\n)=\bold{1}_{P'}(\n)$ for some $D$-dimensional artihematic progression $P'\subseteq[\tilde{N}]^{D}$. Denoting $g'_{a,\m}(\n):=g'_{a}(\n+\m')\in\poly_{D}(G'_{a,\bullet})$, we deduce from (\ref{e815}) that 
\begin{equation}\label{e816}
	\Bigl\vert\E_{\n\in[\tilde{N}]^{D}}\bold{1}_{P'}(\n)\chi(\n)\cdot\Phi'_{0,a}(g'_{a,\m'}(\n)\cdot e_{X'_{a}})\Bigr\vert>C\lambda(M^{\ast}).
\end{equation}
Since $\int_{X'_{a}}\Phi'_{0,a}\,dm_{X'_{a}}=0$ and $\Vert\Phi'_{0,a}\Vert_{\Lip(\X'_{a})}=\Vert\Phi'_{a}\Vert_{\Lip(\X'_{a})}=1$, by Theorem \ref{key}, there exists a function 
$\zeta\colon\N\to\mathbb{R}_{+}$
such that if $N$ is sufficiently large, then $(g'_{a}(\n)\cdot e_{X'_{a}})_{\n\in [\tilde{N}]^{D}}$ is not totally $\zeta(M^{\ast})$-equidistributed on $X'_{a}$ with respect to $\X'_{a}$. By choosing $\zeta$ to be the function defined above (which is a function of $M^{\ast}$ depending only on $\delta,\X,\B$ and $\Omega$), we get a contradiction to Property 2. This finishes the proof.
\end{proof}

\subsection{Deducing the strong $U^{d}$ structure theorem from the weak one}

Our last step is to finish the proof of Theorem \ref{nU3s}, i.e. the strong $U^{d}$ structure theorem.
By using an iterative argument of energy increment, we can deduce that the weak $U^{d}$ structure theorem (Theorem \ref{nU3w}) implies Theorem \ref{nU3s}. As the method is identical to Section 8.10 in ~\cite{FH}, we omit the proof. 

\section{Partition regularity properties}\label{s:par}
In this section, we explain how Theorem \ref{nU3s} can be applied to deduce partition regularity properties. 

\subsection{Statement of the main result on partition regularity problems}

We start with a technical definition which captures the algebraic structure behind partition regularity problems.

\begin{defn}[Types of polynomials]\label{typep}
	Let $\K$ be an integral tuple, $r\in\N_{+}$, and $p\in\mathbb{C}[x,y;z_{1},\dots,z_{r}]$ be a \emph{homogeneous} polynomial, meaning that $$p(tx,ty;tz_{1},\dots,tz_{r})=t^{k}p(x,y;z_{1},\dots,z_{r})$$ for all $t\in\mathbb{C}$. We say that $p$ is a \emph {$K$-type} polynomial if there exist $d\in\N_{+}$ and $a_{1},\dots,a_{d}, a'_{1},\dots,a'_{d}\in \O_{K}$ satisfying (i) $a_{i}\neq a_{j}$ and $a'_{i}\neq a'_{j}$ for all $1\leq i,j\leq d, i\neq j$; and (ii) $\{a_{1},\dots,a_{d}\}\neq \{a'_{1},\dots,a'_{d}\}$,  such that for all $m,n,k\in K$, there exist $z_{1},\dots,z_{r}\in K$ such that 
	$$p\Bigl(k\prod_{i=1}^{d}(m+a_{i}n),k\prod_{i=1}^{d}(m+a'_{i}n);z_{1},\dots,z_{r}\Bigr)=0.$$
\end{defn}	

Our main result is the following:

\begin{thm}[Partition regularity result in full generality]\label{pr}
	Let $\K$ be an integral tuple, $r\in\N_{+}$, and $p\in\mathbb{C}[x,y;z_{1},\dots,z_{r}]$ be a $K$-type polynomial. Then $p$ is partition regular over $\O_{K}$ with respect to $x$ and $y$.
\end{thm}

\begin{rem}
	Although Theorem \ref{pr} already covers many classes of equations,
	there are three important restrictions. The first is that the number of variables taking values in $U_{i}$ (i.e. $x$ and $y$) equals to $2$. The second is that the polynomial $p$ is homogeneous. The third is that we require $p(x,y;z_{1},\dots,z_{r})=0$ to have a parametrized solution of the form $x=k\prod_{i=1}^{d}(m+a_{i}n)$ and $y=k\prod_{i=1}^{d}(m+a'_{i}n)$. 
\end{rem}

We start with explaining the applications of Theorem \ref{pr}, and differ its proof to the end of the section.

\subsection{Applications of Theorem \ref{pr} to partition regularity problems}\label{s:p2}

We first provide a criteria for partition regularity properties for  quadratic equations. 
\begin{prop}[Partition regularity for quadratic equations]\label{pr2}
	Let $p$ be a quadratic equation of the form
	\begin{equation}\nonumber
		\begin{split}
			p(x,y;z)=ax^{2}+by^{2}+cz^{2}+dxy+exz+fyz
		\end{split}
	\end{equation}
	for some $a,b,c,d,e,f\in\Z$.
	Denote 
	$$\Delta_{1}(p):=\sqrt{e^{2}-4ac}, \Delta_{2}(p):=\sqrt{f^{2}-4bc}, \Delta_{3}(p):=\sqrt{(e+f)^{2}-4c(a+b+d)}.$$
	
	Suppose that $c,\Delta^{2}_{1}(p),\Delta^{2}_{2}(p)\neq 0$, and at lease one of $\Delta^{2}_{3}(p)$ and $\Delta^{2}_{1}(p)-\Delta^{2}_{2}(p)$ is non-zero, then
 $p$ is a $K$-type polynomial for $K:=\mathbb{Q}(\Delta_{1}(p),\Delta_{2}(p),\Delta_{3}(p)).$ In particular, $p$ is partition regular over $\O_{K}$ with respect to $x$ and $y$ by Theorem \ref{pr}. 
%
\end{prop}	

It is not hard to see that Proposition \ref{pr2} implies Theorem \ref{pr3}.

\begin{rem}
	The quadratic equations which are not covered by Proposition \ref{pr2} are the following: (i) $c=0$; (ii) $c\neq 0$, one of $\Delta^{2}_{1}(p),\Delta^{2}_{2}(p),\Delta^{2}_{3}(p)$ equals to 0 and the other two are equal; (iii) $c\neq 0$, one of $\Delta^{2}_{1}(p)$ and $\Delta^{2}_{2}(p)$ equals to 0, and the other one is not equal to $\Delta^{2}_{3}(p)$. Here are some examples:
	\begin{itemize}
		\item type (i): $xy+yz+xz$, $x^{2}+y^{2}+exz+fyz$ ($e,f,e+f\neq 0$);
		\item type (ii): $x^{2}+y^{2}+cz^{2}-2xy$ ($c\neq 0$);
		\item type (iii): $x^{2}+by^{2}+z^{2}+2xz$ ($b\neq 0$).
	\end{itemize}	
	
	The equations $p$ in these three cases are ``degenerate" in one way or another, which are completely different from the case discussed in Proposition \ref{pr2}. For Case (ii), it is not hard to show that  for every algebraic number field $K$, $p$ is partition regular over $\O_{K}$ with respect to $x$ and $y$ if and only if $\mathbb{Q}(\Delta_{1}(p),\Delta_{2}(p),\Delta_{3}(p))\subseteq K$. It is an interesting question to ask whether in Cases (i) and (iii) $p$ is partition regular over $\Z$ with respect to $x$ and $y$. But this is beyond the theme of this paper.
\end{rem}	

\begin{proof}[Proof of Proposition \ref{pr2}]
	Let
	\begin{equation}\nonumber
		\begin{split}
			p'(x,y;z):=p(2cx,2cy,z-ex-fy)
			=-c\Bigl(\Delta^{2}_{1}(p)x^{2}+\Delta^{2}_{2}(p)y^{2}-z^{2}+(\Delta^{2}_{3}(p)-\Delta^{2}_{1}(p)-\Delta^{2}_{2}(p))xy\Bigr).			
		\end{split}	
	\end{equation}

	(i) Suppose first that $\Delta^{2}_{3}(p)=0$ and $\Delta^{2}_{1}(p)\neq \Delta^{2}_{2}(p)$.
	By a direct computation, for all $k,m,n\in K$, we have that $p'(x',y';z')=0$ for 
	$$x':=k(m-\Delta_{2}(p)n)(m+\Delta_{2}(p)n);$$
	$$y':=k(m-\Delta_{1}(p)n)(m+\Delta_{1}(p)n);$$
	$$z':=\pm k(\Delta^{2}_{1}(p)-\Delta^{2}_{2}(p))mn.$$
	Let $a_{1}=-\Delta_{2}(p)$, $a_{2}=\Delta_{2}(p)$, $a'_{1}=-\Delta_{1}(p)$ and $a'_{2}=\Delta_{1}(p)$. We have that there exists $z\in K$ such that $$p(2kc(m+a_{1}n)(m+a_{2}n),2kc(m+a'_{1}n)(m+a'_{2}n);z)=0.$$
	Since 	 $\Delta^{2}_{1}(p),\Delta^{2}_{2}(p)\neq 0$, we have that $a_{1}\neq a_{2}$ and $a'_{1}\neq a'_{2}$. If $\{a_{1},a_{2}\}=\{a'_{1},a'_{2}\}$, then $\Delta^{2}_{1}(p)=\Delta^{2}_{2}(p)$, 
	a contradiction. This implies that $\{a_{1},a_{2}\}\neq\{a'_{1},a'_{2}\}$, and so $p$ is a $K$-type polynomial.

	\
	
	(ii) We now assume that $\Delta^{2}_{3}(p)\neq 0$.
	This case is similar to Appendix C of \cite{FH}.
	By a direct computation, for all $k,m,n\in K$, we have that $p'(x',y';z')=0$ for 
	$$x':=k\Bigl(m+c(\Delta^{2}_{2}(p)+\Delta_{2}(p)\Delta_{3}(p))n\Bigr)\Bigl(m+c(\Delta^{2}_{2}(p)-\Delta_{2}(p)\Delta_{3}(p)))n\Bigr);$$
	$$y':=k\Bigl(m+c(\Delta^{2}_{2}(p)-\Delta^{2}_{3}(p)+\Delta_{1}(p)\Delta_{3}(p))n\Bigr)\Bigl(m+c(\Delta^{2}_{2}(p)-\Delta^{2}_{3}(p)-\Delta_{1}(p)\Delta_{3}(p))n\Bigr);$$
	$$z':=\pm k\Delta_{3}(p)\Bigl(m^{2}+c(\Delta^{2}_{1}(p)+\Delta^{2}_{2}(p)-\Delta^{2}_{3}(p))mn+c^{2}\Delta^{2}_{1}(p)\Delta^{2}_{2}(p)n^{2}\Bigr).$$
	Let $a_{1}=c(\Delta^{2}_{2}(p)+\Delta_{2}(p)\Delta_{3}(p))$, $a_{2}=c(\Delta^{2}_{2}(p)-\Delta_{2}(p)\Delta_{3}(p))$, $a'_{1}=c(\Delta^{2}_{2}(p)-\Delta^{2}_{3}(p)+\Delta_{1}(p)\Delta_{3}(p))$ and $a'_{2}=c(\Delta^{2}_{2}(p)-\Delta^{2}_{3}(p)-\Delta_{1}(p)\Delta_{3}(p))$.
	Since $\Delta_{1}(p), \Delta_{2}(p), \Delta_{3}(p)\in K$, we have that $a_{1},a_{2},a'_{1},a'_{2}\in K$ and $x',y',z'\in K$.
	So for all $k,m,n\in K$, there exists $z\in K$ such that $$p(2kc(m+a_{1}n)(m+a_{2}n),2kc(m+a'_{1}n)(m+a'_{2}n);z)=0.$$
	Since 	 $\Delta^{2}_{1}(p),\Delta^{2}_{2}(p),\Delta^{2}_{3}(p),c\neq 0$, we have that $a_{1}\neq a_{2}$ and $a'_{1}\neq a'_{2}$. If $\{a_{1},a_{2}\}=\{a'_{1},a'_{2}\}$, then $a_{1}+a_{2}=a'_{1}+a'_{2}$ and so $\Delta^{2}_{3}(p)=0$, a contradiction. This implies that $\{a_{1},a_{2}\}\neq\{a'_{1},a'_{2}\}$, and so $p$ is a $K$-type polynomial.
\end{proof}

Another application of Theorem \ref{pr} is to prove Corollary \ref{45}.

\begin{proof}[Proof of Corollary \ref{45}]
	By using an identity of G\'erardin:
	$$(m^{2}-n^{2})^{4}+(2mn+m^{2})^{4}+(2mn+n^{2})^{4}=2(m^{2}+mn+n^{2})^{4},$$
	we have
	that the polynomial $p(x_{1},x_{2};z_{1},z_{2}):=x^{2}_{1}-2x^{2}_{2}+z^{2}_{1}-z^{2}_{2}$ is of $\mathbb{Q}(\sqrt{-3})$-type (by setting $x_{1}=k(m+n)(m-n)$ and $x_{2}=k(m+\frac{-1+\sqrt{-3}}{2}n)(m+\frac{-1-\sqrt{-3}}{2}n)$). Since $\Z[\frac{1+\sqrt{-3}}{2}]$ is the ring of integers of $\mathbb{Q}(\sqrt{-3})$, Theorem \ref{pr} implies that $p$ is partition regular over  $\Z[\frac{1+\sqrt{-3}}{2}]$ with respect to $x_{1}$ and $x_{2}$.
\end{proof}

\subsection{Multiplicative measure preserving systems}
To study Theorem \ref{pr}, we introduce the multiplicative density of a subset of a number field.
Let $\K$ be an integral tuple and $I_{1},I_{2},\dots$ be an enumeration of prime ideals in $\O_{K}$ (the number of primes ideals are countable) such that $N(I_{1})\leq N(I_{2})\leq \dots$. Let $(\Phi_{N})_{N\in\mathbb{N}}$ be the sequence of finite subsets of $\O_{K}$ defined by
$$\Phi_{N}:=\{n\in\O_{K}\colon (n)\supseteq (I_{1}I_{2}\dots I_{N})^{N}\}.$$
Then $(\Phi_{N})_{N\in\mathbb{N}}$ is a \emph{multiplicative F${\o}$lner sequence} on  $\O_{K}$, meaning that for all $a\in \OO$,\footnote{$\OO:=\mathcal{O}_{K}\backslash\{0\}$.}	
\begin{equation}\nonumber
	\begin{split}
		\lim_{N\rightarrow\infty}\frac{\vert a^{-1}\Phi_{N}\triangle\Phi_{N}\vert}{\vert \Phi_{N}\vert}=0,
	\end{split}
\end{equation}
where $a^{-1}\Phi_{N}:=\{a^{-1} x\in\O_{K}\colon x\in\Phi_{N}\}=\{y\in\O_{K}\colon ay\in\Phi_{N}\}$.

\begin{defn}[Multiplicative density]
	Let $\K$ be an integral tuple.
	The \emph{(upper) multiplicative density} of a subset $E$ of $\O_{K}$ (with respect to the multiplicative F${\o}$lner sequence $(\Phi_{N})_{N\in\mathbb{N}}$) is defined to be
	\begin{equation}\nonumber
		\begin{split}
			d_{\mult,K}(E):=\limsup_{N\rightarrow\infty}\frac{\vert E\cap\Phi_{N}\vert}{\vert \Phi_{N}\vert}.
		\end{split}
	\end{equation}
\end{defn}
When there is no confusion, we write $d_{\mult}(E):=d_{\mult,K}(E)$ for short.
Since $(\Phi_{N})_{N\in\mathbb{N}}$ is a multiplicative F${\o}$lner sequence, for all $E\subseteq\O_{K}$ and $a\in \OO$,
$$d_{\mult}(E)=d_{\mult}(a^{-1}E).$$

 To show Theorem \ref{pr}, our strategy is to convert the question to a recurrence problem on a special type of dynamical systems:

\begin{defn}[Action by dilation]
	Let $\K$ be an integral tuple.
	An \emph{action by dilation} over $\bold{K}$ on a probability space $(X,\mathcal{D},\mu)$ is a family $(T_{n})_{n\in \OO}$ of invertible measure preserving transformations of $(X,\mathcal{D},\mu)$ that satisfy $T_{1}=id$ and $T_{m}\circ T_{n}=T_{mn}$ for all $m,n\in \OO$.
\end{defn}

Note that an action by dilation $(T_{n})_{n\in \OO}$ can be extended to a measure preserving action $(T_{n})_{n\in K^{\ast}}$ by defining $T_{m/n}=T_{m}\circ T_{n}^{-1}$ for all $m,n\in\OO$. We remark that $(T_{n})_{n\in K^{\ast}}$ is well defined even though $\O_{K}$ may not be a unique factorization domain. In fact, 
let $m/n=m'/n'$ for some $m,n,m',n'\in\OO$. Then 
\begin{equation}\label{wd}
	T_{m}\circ T_{n'}=T_{mn'}=T_{m'n}=T_{m'}\circ T_{n},
\end{equation}	
and so $T_{m}\circ T_{n}^{-1}=T_{m'}\circ T_{n'}^{-1}.$
Since $\OO$ with multiplication is a discrete amenable semi-group, we have the Furstenberg correspondence principle (see for example  Theorem 2.1 of ~\cite{BM2} and Theorem 6.4.17 of ~\cite{B3}):

\begin{thm}[Furstenberg correspondence principle] \label{FC}
	Let $\K$ be an integral tuple and
	$E$ be a subset of $\O_{K}$. Then there exist an action by dilation $(T_{n})_{n\in \OO}$ on a probability space $(X,\mathcal{D},\mu)$ and a set $A\in\mathcal{D}$ with $\mu(A)=d_{\mult}(E)$ such that for every $k\in\mathbb{N}_{+}$ and $n_{1},\dots,n_{k}\in\O_{K}$, we have
	\begin{equation}\nonumber
		\begin{split}
			d_{mult}(n_{1}^{-1}E\cap\dots\cap n_{k}^{-1}E)\geq\mu(T_{n_{1}}^{-1}A\cap\dots\cap T_{n_{k}}^{-1}A).
		\end{split}
	\end{equation}
\end{thm}

Let $\mathcal{M}_{K}^{c}$ denote the collection all \emph{completely multiplicative functions}  $\chi\colon \O_{K}\to \mathbb{C}$ with modulus equals to 1, meaning that $\chi(mn)=\chi(m)\chi(n)$ for all $m,n\in\OO$, and that $\vert\chi\vert\equiv 1$. Every $\chi\in \mathcal{M}_{K}^{c}$ can be extended to a multiplicative function on $K^{\ast}$ by setting
$$\chi(m/n):=\chi(m)\overline{\chi(n)}$$
for all $m,n\in\OO$. $\chi(m/n)$ is well defined by a reason similar to (\ref{wd}). 
Endowing $\mathcal{M}_{K}^{c}$ with the pointwise multiplication and the topology of pointwise convergence, $\mathcal{M}_{K}^{c}$ is a compact Abelian group with the constant function $\bold{1}$ being the unit element. Moreover, $\mathcal{M}_{K}^{c}$ is the dual group of $K^{\ast}$.

Let  $(X,\mathcal{D},\mu)$ be a probability space with an action by dilation $(T_{n})_{n\in \OO}$. For every $f\in L^{2}(\mu)$, by the spectral theorem, there exists a positive finite measure $\nu$ (called the \emph{spectral measure  of $f$}) on the dual group $\mathcal{M}_{K}^{c}$ of
$K^{\ast}$ such that for all $m,n\in\O_{K}^{\ast}$,
\begin{equation}\label{spec}
	\begin{split}
		\int_{X}{T_{m}f\cdot T_{n}\overline{f}}d\mu=\int_{X}{T_{m/n}f\cdot \overline{f}}d\mu=\int_{\mathcal{M}_{K}^{c}}{\chi(m/n)}d\nu(\chi)
		=\int_{\mathcal{M}_{K}^{c}}{\chi(m)\overline{\chi}(n)}d\nu(\chi).
	\end{split}
\end{equation}

The following lemma can be deduced by the same argument on pages 64--65 of ~\cite{FH}:
\begin{lem}[Positivity properties for spectrum measures]\label{37}
	Let $\K$ be an integral tuple and
	$(X,\mathcal{D},\mu)$ be a probability space with an action by dilation $(T_{n})_{n\in \OO}$. 
	Let $A\in\mathcal{D}$ with $\mu(A)>0$ and $\nu$ be the spectral measure of the function $\bold{1}_{A}$. Then
	\begin{equation}\label{nu}
		\nu(\{\bold{1}\})>0 \text{ and } \int_{\mathcal{M}_{K}^{c}}\chi(m)\overline{\chi}(n)d\nu(\chi)\geq 0 \text{ for all } m,n\in \OO.
	\end{equation}	
\end{lem}

In order to prove Theorem \ref{pr}, it suffices to 
show the following multiple recurrence property for multiplicative functions:

\begin{prop}[Multiple recurrence property for multiplicative functions]\label{c4}
	Let $\K$ be an integral tuple. Let $d\in\N_{+}$ and $a_{1},\dots,a_{d},a'_{1},\dots,a'_{d}\in\O_{K}$ be such that (i) $a_{i}\neq a_{j}$ and $a'_{i}\neq a'_{j}$ for all $1\leq i,j\leq d$; and (ii) $\{a_{1},\dots,a_{d}\}\neq \{a'_{1},\dots,a'_{d}\}$. Let $\nu$ be a probability measure on $\mathcal{M}_{K}^{c}$ satisfying (\ref{nu}). Then there exist $m,n\in\O_{K}$ such that 
	$\prod_{i=1}^{d}(m+a_{i}n)$ and $\prod_{i=1}^{d}(m+a'_{i}n)$ are distinct and nonzero,
	and that
	\begin{equation}\nonumber
		\begin{split}
			\int_{\mathcal{M}_{K}^{c}}{\prod_{i=1}^{d}\chi(m+a_{i}n)\prod_{i=1}^{d}\overline{\chi}(m+a'_{i}n)}d\nu(\chi)>0.
		\end{split}
	\end{equation}
\end{prop}

We postpone the proof of Proposition \ref{c4} to the next section, but explain first how to derive Theorem \ref{pr} from Proposition \ref{c4}.

\begin{proof}[Proof of Theorem \ref{pr} assuming Proposition \ref{c4}]
	Let $d\in\N_{+}$ and $a_{1},\dots,a_{d}, a'_{1},\dots,a'_{d}\in \O_{K}$ be as in Definition \ref{typep} for the $K$-type polynomial $p(x,y;z_{1},\dots,z_{r})$.   
	By the sub-additivity of $d_{\mult}$, in order to show the partition regularity of $p$, it suffices to show that for all $E\subseteq \O_{K}$ with $d_{\mult}(E)>0$, there exist $m,n,k\in\O_{K}$ such that 
	$
	x:=k\prod_{i=1}^{d}(m+a_{i}n)$ and $y:=k\prod_{i=1}^{d}(m+a'_{i}n)
	$	
	are distinct and nonzero elements in $E$.
	It suffices to show that there exist $m,n\in\O_{K}$ such that 	
	\begin{equation}\label{cc}
		\text{$\prod_{i=1}^{d}(m+a_{i}n)$ and $\prod_{i=1}^{d}(m+a'_{i}n)$ are distinct and non-zero,}
	\end{equation}	
	and that $$d_{\mult}\Bigl(\prod_{i=1}^{d}(m+a_{i}n)^{-1}E\cap \prod_{i=1}^{d}(m+a'_{i}n)^{-1}E\Bigr)>0.$$
	Let the probability space $(X,\mathcal{D},\mu)$, the action by dilation $(T_{n})_{n\in\OO}$ and the set $A\in\mathcal{D}$ with $\mu(A)=d_{\mult}(E)>0$ be as in Theorem \ref{FC}.
	By Theorem \ref{FC}, it suffices to show that there exist $m,n\in\O_{K}$ such that 	(\ref{cc}) holds and that 
	$$\mu(T_{\prod_{i=1}^{d}(m+a_{i}n)}^{-1}A\cap T_{\prod_{i=1}^{d}(m+a'_{i}n)}^{-1}A)=\int_{\mathcal{M}_{K}^{c}}{\prod_{i=1}^{d}\chi(m+a_{i}n)\prod_{i=1}^{d}\overline{\chi}(m+a'_{i}n)}d\nu(\chi)>0,$$
	where $\nu$ is the spectrum measure of $\bold{1}_{A}$. By Lemma \ref{37} and Proposition \ref{c4}, we are done.
\end{proof}	

\subsection{A sketch of the proof of Proposition \ref{c4}}

Since
$$\lim_{N\rightarrow\infty}\frac{\Bigl\vert\Bigl\{(m,n)\in \iota_{\B}([N]^{D})\times \iota_{\B}([N]^{D})\colon \text{ two of } \prod_{i=1}^{d}(m+a_{i}n), \prod_{i=1}^{d}(m+a'_{i}n) \text{ and  0 are equal}\Bigr\}\Bigr\vert}{N^{2D}}=0,$$
in order to finish the proof of Proposition \ref{c4}, it suffices to show the following:

\begin{prop}[Multiple averages for multiplicative functions]\label{c5}
	Let $\K$ be an integral tuple. Let $d\in\N_{+}$ and $a_{1},\dots,a_{d},a'_{1},\dots,a'_{d}\in\O_{K}$ be such that (i) $a_{i}\neq a_{j}$ and $a'_{i}\neq a'_{j}$ for all $1\leq i,j\leq d$; and (ii) $\{a_{1},\dots,a_{d}\}\neq \{a'_{1},\dots,a'_{d}\}$. Let $\nu$ be a probability measure on $\mathcal{M}_{K}^{c}$ satisfying (\ref{nu}). Then 
	\begin{equation}\label{temp9}
		\begin{split}
			&\quad\liminf_{N\rightarrow\infty}\mathbb{E}_{m,n\in \iota_{\B}([N]^{D})}\int_{\mathcal{M}_{K}^{c}}{\prod_{i=1}^{d}\chi(m+a_{i}n)\prod_{i=1}^{d}\overline{\chi}(m+a'_{i}n)}d\nu(\chi)
			\\&=\liminf_{N\rightarrow\infty}\mathbb{E}_{\m,\n\in [N]^{D}}\int_{\mathcal{M}_{K}^{c}}{\prod_{i=1}^{d}\chi\circ\iota_{\B}(\m+\n A_{\B}(a_{i}))\prod_{i=1}^{d}\overline{\chi}\circ\iota_{\B}(\m+\n A_{\B}(a'_{i}))}d\nu(\chi)
			>0.
		\end{split}
	\end{equation}
\end{prop}

The proof of Proposition \ref{c5} is similar to Proposition 10.4 of ~\cite{FH} and Proposition 3.3 of ~\cite{S}. We omit the proof but stress the differences.

Suppose first that $a_{i}=a'_{j}$ for some $1\leq i,j\leq d$. We assume without loss of generality that $a_{d}=a'_{d}$. 
Since $\chi$ is of modulus 1, 
$${\prod_{i=1}^{d}\chi(m+a_{i}n)\prod_{i=1}^{d}\overline{\chi}(m+a'_{i}n)}={\prod_{i=1}^{d-1}\chi(m+a_{i}n)\prod_{i=1}^{d-1}\overline{\chi}(m+a'_{i}n)}.$$
This implies that we may remove the terms $m+a_{d}n$ and $m+a'_{d}n$ simultaneously from the statement of  Proposition \ref{c5} and replace $d$ with $d-1$. Since $\{a_{1},\dots,a_{d}\}\neq \{a'_{1},\dots,a'_{d}\}$, we can not remove all of $a_{1},\dots,a_{d},a'_{1},\dots,a'_{d}$ by using this induction. In conclusion, it suffices to prove Proposition \ref{c5} under the additional assumption that $d\geq 1$ and all of $a_{1},\dots,a_{d},a'_{1},\dots,a'_{d}$ are distinct. By a change of variables, we may further assume that one (and only one) of them is 0.

For $A\in M_{D\times D}(\mathbb{Z})$, let $H(A)$ denote the height of $A$.
Let $$\ell:=\sum_{i=1}^{d}H(A_{\B}(a_{i}))+\sum_{i=1}^{d}H(A_{\B}(a'_{i}))+10$$
and $\tilde{N}$ be the
smallest prime number (in $\N$) greater than $10D\ell N$.  Let $\chi_{N}\colon\Rr\to \mathbb{C}$ denote the truncated function given by $\chi_{N}(\n)=\chi\circ \iota_{\B}(\n)$ for all $\n\in\{1,\dots,N\}^{D}$ and $\chi_{N}(\n)=0$ otherwise. In order to show (\ref{temp9}), it suffices to show that 
\begin{equation}\label{temp10}
	\begin{split}
		\liminf_{N\rightarrow\infty}\mathbb{E}_{\m,\n\in \Rr}\int_{\mathcal{M}_{K}^{c}}\bold{1}_{[N]^{D}}(\n)\prod_{i=1}^{d}\chi_{N}(\m+\n A_{\B}(a_{i}))\prod_{i=1}^{d}\overline{\chi}_{N}(\m+\n A_{\B}(a'_{i}))d\nu(\chi)
		>0
	\end{split}
\end{equation}
(in fact the left hand side of (\ref{temp9}) equals to a constant multiple of (\ref{temp10}) for a reason similar to (10.13) of \cite{FH}).

Applying Theorem \ref{nU3s} for  $\Omega=10D\ell$ and with $d$ replaced by $2d-1$, we may decompose the truncated function $\chi_{N,\tilde{N}}:=\chi_{N}$ into the sum $\chi_{N}=\chi_{N,s}+\chi_{N,u}+\chi_{N,e}$ satisfying the statements in Theorem \ref{nU3s},
and
expand the left hand side of (\ref{temp10}) into $3^{2d}$ terms. Let $\epsilon>0$ be a sufficiently small error term. By a similar argument as in the proof of Proposition 10.5 in ~\cite{FH} (the estimation of the $A_{3}(N)$ term on pages 71--72), we have that
\begin{equation}\nonumber
	\begin{split}
		\liminf_{N\rightarrow\infty}\mathbb{E}_{\m,\n\in \Rr}\int_{\mathcal{M}_{K}^{c}}\bold{1}_{[N]^{D}}(\n)\prod_{i=1}^{d}\chi_{N,s}(\m+\n A_{\B}(a_{i}))\prod_{i=1}^{d}\overline{\chi}_{N,s}(\m+\n A_{\B}(a'_{i}))d\nu(\chi)
	\end{split}
\end{equation}
is bounded below by a positive number which is independent of $\e$ (to obtain such an estimate, one needs to invoke the property (\ref{nu}) of the measure $\nu$ and use an immediate generalization of Lemma 10.6 of \cite{FH}).

Now
it suffices to show that all other terms are negligible. A term is obviously $O(\e)$ if it contains the expression $\chi_{N,e}$. 
Since $\Vert\chi_{N,u}\Vert_{U^{2d-1}(\Rr)}\ll\e$,
it suffices to show that all terms containing the expression $\chi_{N,u}$ are negligible, which holds immediately if one can show that 
\begin{equation}\label{temp8}
	\begin{split}
		\Bigl\vert\mathbb{E}_{\m,\n\in \Rr}\bold{1}_{[N]^{D}}(\n)\prod_{i=1}^{2d}f_{i}(\m+\n A_{\B}(a_{i}))\Bigr\vert
		\leq C\min_{1\leq i\leq 2d}\Vert f_{i}\Vert^{\frac{1}{D+1}}_{U^{2d-1}(\Rr)}+\frac{10}{\tilde{N}}
	\end{split}
\end{equation}
for all functions $f_{1},\dots,f_{2d}\colon\Rr\to \mathbb{C}$ with modulus at most 1 for some  $C:=C(a_{1},\dots,a_{2d})>0$. The proof of (\ref{temp8}) is a straightforward generalization of Lemma 10.7 in ~\cite{FH}, and so we are done. It is worth noting that in the proof of (\ref{temp8}), we need to use the fact that for all $a\in\O_{K}$ such that $\vert N_{\Bb}(a)\vert<\tilde{N}$, the map $\x\to\x A_{\Bb}(a) \mod\Rr$ is a bijection from $\Rr$ to itself.

\appendix 

\section{Equivalent definitions for aperiodic functions}\label{appe}
In this appendix,
we show that the two definitions of aperiodic functions (\ref{ap1}) and (\ref{ap2}) are equivalent. 


\begin{lem}
	Let $f\colon\Z\to\mathbb{C}$ be a function with modulus at most 1. Then 
	\begin{equation}\label{pa1}
		\lim_{N\to\infty}\frac{1}{N}\sum_{n=0}^{N-1} f(an+b)=0
	\end{equation}	
	for all $a,b\in\Z, a\neq 0$ if and only if
	\begin{equation}\label{pa2}
		\lim_{N\to\infty}\sup_{L\in\N_{+},a,b\in\Z,a\neq 0}\Bigl\vert\frac{1}{2N+1}\sum_{n=-N}^{N}\bold{1}_{P_{a,b,L}}(n)\cdot f(n)\Bigr\vert=0,
	\end{equation}
	where $P_{a,b,L}:=\{am+b\in\Z\colon 0\leq m\leq L-1\}$.
\end{lem}	
\begin{proof}
   Equation (\ref{pa2}) obviously implies (\ref{pa1}). Now suppose that (\ref{pa2}) fails for some function $f$. Then there exist $\e>0$, $N_{i},L_{i}\in\N_{+}$ and $a_{i},b_{i}\in\Z, a_{i}\neq 0$ for $i\in\N$ such that $N_{i+1}>N_{i}$ and that  
    \begin{equation}\label{pa3}
    	\Bigl\vert\frac{1}{2N_{i}+1}\sum_{n=-N_{i}}^{N_{i}}\bold{1}_{P_{a_{i},b_{i},L_{i}}}(n)\cdot f(n)\Bigr\vert>\e
    \end{equation}
    for all $i\in\N$. Since $\vert f\vert\leq 1$, we get from (\ref{pa3}) that
    $$\e<\frac{1}{2N_{i}+1}\vert P_{a_{i},b_{i},L_{i}}\vert\leq \frac{\lfloor\frac{ 2N_{i}+1}{\vert a_{i}\vert}\rfloor+1}{2N_{i}+1}.$$
    So if $i$ is sufficiently large, then $\vert a_{i}\vert\leq \frac{2}{\e}$. Since $a_{i}$ only take finitely many values, there exist infinitely many $i$ such that these $a_{i}$ take a same value $a_{0}$, and $b_{i}\equiv b_{0} \mod a_{0}$ for some $0\leq b_{0}<\vert a_{0}\vert$. In conclusion, there exist an infinitely sequence of integers $N_{i}\in\N$ (which is still denoted by $N_{i}$), and $M_{i},M'_{i}\in\Z$ with $-N_{i}\leq a_{0}M_{i}+b_{0}\leq a_{0}M'_{i}+b_{0}\leq N_{i}$ such that 
    	\begin{equation}\label{pa4}
    		\Bigl\vert\frac{1}{2N_{i}+1}\sum_{n=M_{i}}^{M'_{i}}f(a_{0}n+b_{0})\Bigr\vert>\e.
    	\end{equation}
    	We may assume without loss of generality that $a_{0}>0$ as the other case is similar. Then $M_{i}-M'_{i}\leq 2N_{i}$. 
    	
    	Assume that (\ref{pa1}) holds for all $a,b\in\Z, a\neq 0$. Then there exists $M_{0}:=M_{0}(\e)>0$ such that for all $M>M_{0}$, we have that
    		\begin{equation}\nonumber
    			\Bigl\vert \sum_{n=0}^{M-1} f(a_{0}n+b_{0})\Bigr\vert,\Bigl\vert \sum_{n=-M+1}^{0} f(a_{0}n+b_{0})\Bigr\vert<M\e/2.
    		\end{equation}
    		In other words, for all $M\in\N_{+}$,
    		\begin{equation}\nonumber
    			\Bigl\vert \sum_{n=0}^{M-1} f(a_{0}n+b_{0})\Bigr\vert,\Bigl\vert \sum_{n=-M+1}^{0} f(a_{0}n+b_{0})\Bigr\vert<\max\{M\e/2,M_{0}\}.
    		\end{equation}	
    		So we may deduce from (\ref{pa4}) that
    		\begin{equation}\label{pa5}
    			(2N_{i}+1)\e<(M'_{i}-M_{i})\e/2+2M_{0}\leq (2N_{i}+1)\e/2 +2M_{0},
    		\end{equation}
    which is	a contradiction if $i$ is sufficiently large. This contradiction  implies that (\ref{pa1}) implies (\ref{pa2}).
\end{proof}

\end{document}